\documentclass{article}
\usepackage[francais,english]{babel} % guillemets \og text \fg{}
\usepackage{aeguill} % guillemets

\usepackage[margin=35mm]{geometry}
\usepackage{amsthm,amsmath,amssymb}
\usepackage{mathtools}  %provides \coloneqq
\usepackage{url}
\usepackage{tikz}
\usepackage{pgfplots}
\usetikzlibrary{calc,matrix,shapes,arrows,positioning,shadows,backgrounds}
\newcommand{\deptx}{0.37}
\newcommand{\depty}{0.37}
\newcommand{\eps}{0.02}
\definecolor {gr}{rgb}{0.7,0.7,0.7}
\newcommand{\cube}{
\rotatebox{0}{
\raisebox{-0.1em}{
$\begin{tikzpicture}[xscale=0.2,yscale=0.2,anchor=base]
\draw[color=gr](0+\deptx+\eps,0+\depty+\eps)--(1+\deptx,0+\depty);
\draw[color=gr](0+\deptx,1+\depty)--(0+\deptx+\eps,0+\depty+\eps);
\draw[color=gr](0+\eps,0+\eps)--(0+\deptx+\eps,0+\depty+\eps);
\draw(0+\eps,0+\eps)--(1,0)--(1+\deptx,0+\depty)--(1+\deptx-\eps,1+\depty-\eps)--(0+\deptx,1+\depty)--(0,1)--cycle;
\draw(0,1)--(1-\eps,1-\eps);
\draw(1,0)--(1-\eps,1-\eps);
\draw(1+\deptx-\eps,1+\depty-\eps)--(1-\eps,1-\eps);
\end{tikzpicture}$
}} }

\usepackage{xcolor}
\usepackage{hyperref}
\newtheorem{theorem}{Theorem}[section]
\newtheorem{proposition}[theorem]{Proposition}
\newtheorem{corollary}[theorem]{Corollary}
\newtheorem{lemma}[theorem]{Lemma}
\theoremstyle{remark}
\newtheorem{remark}{Remark}

\newcommand{\tL}{\mathtt 1}                            % typewriter 1
\newcommand{\tO}{\mathtt 0}                            % typewriter 0

\newcommand{\digitsum}{\mathsf s}
\newcommand{\smallspace}{\hspace{0.5pt}} %small space before \digitsum
\newcommand{\digit}{\delta}
\newcommand{\ijkl}{j}
\newcommand{\eqdef}{\coloneqq}

\newcommand{\nL}{n_{\mathtt{1}} }
\newcommand{\nA}{n_{\mathrm A} }

\newcommand{\nO}{n_{\mathtt{0}} }

\newcommand{\nLO}{n_{\mathtt{10}} }
\newcommand{\nLL}{n_{\mathtt{11}} }
\newcommand{\mLL}{m_{\mathtt{11}} }
\newcommand{\JLO}{J_{\mathtt{10}} }
\newcommand{\JLL}{J_{\mathtt{11}} }
\newcommand{\JLLm}{J_{\mathtt{11}} }

\newcommand{\nOO}{n_{\mathtt{00}} }
\newcommand{\nOL}{n_{\mathtt{01}} }
\newcommand{\JOO}{J_{\mathtt{00}} }
\newcommand{\JOL}{J_{\mathtt{01}} }
\newcommand{\JL}{J_{\mathtt{1}} }
\newcommand{\JO}{J_{\mathtt{0}} }

\newcommand{\Iold}{\mathbb N\setminus[u,\lambda)}
\newcommand{\Inew}{I}
\newcommand{\partition}{\mathfrak P}

\newcommand{\sO}{s_{\mathtt 0}}
\newcommand{\sL}{s_{\mathtt 1}}
\newcommand{\ts}{\underline s}
\newcommand{\ThueMorse}{\mathsf t}

\newcommand{\hO}{h_{\mathtt{0}} }

\newcommand{\nAprime}{m_{\mathtt{0}}}
\newcommand{\nBprime}{m_{\mathtt{1}}}
\newcommand{\indicator}{\mathbf 1}
\newcommand{\multiple}{\mathfrak m}
\newcommand{\tmultiple}{\tilde{\mathfrak m}}

\newcommand{\Qtilde}{\tilde\alpha}
\newcommand{\Qprime}{\alpha}

\newcommand{\Szero}{S_{\mathsf 0}}
\newcommand{\Sone}{S_{\mathsf 1}}
\newcommand{\Stwo}{S_{\mathsf 2}}
\newcommand{\Sthree}{S_{\mathsf 3}}
\newcommand{\Sfour}{S_{\mathsf 4}}
\newcommand{\Sfive}{S_{\mathsf 5}}
\newcommand{\Ssix}{S_{\mathsf 6}}
\newcommand{\Sseven}{S_{\mathsf 7}}
\newcommand{\Seight}{S_{\mathsf 8}}
\newcommand{\Snine}{S_{\mathsf 9}}

\newcommand{\Ezero}{E_{\mathsf 0}}
\newcommand{\Eone}{E_{\mathsf 1}}
\newcommand{\Etwo}{E_{\mathsf 2}}
\newcommand{\Ethree}{E_{\mathsf 3}}
\newcommand{\Efourm}{E'_{\mathsf 4}}
\newcommand{\Efour}{E_{\mathsf 4}}
\newcommand{\Efive}{E_{\mathsf 5}}
\newcommand{\Esix}{E_{\mathsf 6}}
\newcommand{\Eseven}{E_{\mathsf 7}}
\newcommand{\Eeight}{E_{\mathsf 8}}
\newcommand{\Enine}{E_{\mathsf 9}}
\newcommand{\Eten}{E_{\mathsf {10}} }
\newcommand{\Eeleven}{E_{\mathsf {11}} }
\newcommand{\Etwelve}{E_{\mathsf {12}} }
\newcommand{\Ethirteen}{E_{\mathsf {13}} }
\newcommand{\Efourteen}{E_{\mathsf {14}} }

\newcommand{\kappaO}{{\kappa_{\mathtt 0}} }

\newcommand{\etaO}{\eta_{\mathtt 0}}
\newcommand{\etaL}{\eta_{\mathtt 1}}
\newcommand{\GetaO}{\mathcal G_{\mathtt 0}}
\newcommand{\GetaL}{\mathcal G_{\mathtt 1}}
\newcommand{\Geta}{\mathcal G}

\newcommand\namehighlight[1]{{\sc #1}}
\newcommand\SX{S_{\mathsf{X}}}
\newcommand\SY{S_{\mathsf{Y}}}
\newcommand\SZ{S_{\mathsf{Z}}}

\DeclareMathOperator{\logp}{\log^+\!} 
\DeclareMathOperator{\e}{\mathrm{e}}                   % exp(2\pi ix)
\DeclareMathOperator{\LandauO}{\mathcal O}             % Big O-notation

\title{
Thue--Morse along the sequence of cubes
}
\author{
Lukas Spiegelhofer\footnote{The author acknowledges support by the
ANR (Agence Nationale de la Recherche, France)
and the FWF (Austrian Science Fund):
joint project \emph{ArithRand} (grant numbers ANR-20-CE91-0006 (ANR) and I4945-N (FWF)),
and project P36137-N (FWF).}\\
Montanuniversit\"at Leoben, Austria\\
}
\date{}
\begin{document}
\maketitle

\begin{abstract}
The Thue--Morse sequence $\mathsf t=\mathtt{01101001}\cdots$
is an \emph{automatic sequence} over the alphabet $\{\mathtt 0,\mathtt 1\}$.
It can be defined as the binary sum-of-digits function $\mathsf s:\mathbb N\rightarrow\mathbb N$, reduced modulo~$2$,
or by using the substitution $\mathtt 0\mapsto\mathtt{01}$, $\mathtt1\mapsto\mathtt{10}$.
We prove that the asymptotic density of the set of natural numbers $n$ satisfying $\mathsf t(n^3)=\mathtt 0$ equals $1/2$.
Comparable results, featuring asymptotic equivalence along a polynomial as in our theorem, were previously only known for the linear case~[A.~O.~Gelfond, Acta Arith. 13 (1967/68), 259--265], and for the sequence of squares.
The main theorem in~[C.~Mauduit and J.~Rivat, Acta Math. 203 (2009), no. 1, 107--148] was the first such result for the sequence of squares.

Concerning the sum-of-digits function along polynomials $p$ of degree at least three, previous results were restricted either to lower bounds (such as for the numbers $\#\{n<N:\mathsf t(p(n))=\mathtt 0\}$), or to sum-of-digits functions in ``sufficiently large bases''.
By proving an asymptotic equivalence for the case of the Thue--Morse sequence, and a cubic polynomial,
we move one step closer to the solution of the third \emph{Gelfond problem} on the sum-of-digits function (1967/1968), op. cit.
\end{abstract}
%\allowdisplaybreaks

\tableofcontents
%{{{ MSC 2020, key words
\renewcommand{\thefootnote}{\fnsymbol{footnote}} 
% 	11A63   	Radix representation; digital problems
%   11K16     Normal numbers, radix expansions, Pisot numbers, Salem numbers, good lattice points, etc. [See also 11A63]
%   05A16     Asymptotic enumeration
\renewcommand{\thefootnote}{\arabic{footnote}}
%}}}

%\input{cubes_introduction.tex}
%{{{ sec_intro
\section{Introduction}\label{sec_intro}
The {\sc Thue--Morse} sequence $\mathsf{t}$ over the alphabet $\{\tO,\tL\}$
is the unique fixed point of the substitution defined by
\[\tO\mapsto\tO\tL,\quad\tL\mapsto\tL\tO\]
that starts with $\tO$:
\[\mathsf{t}=
\tO\tL\tL\tO\tL\tO\tO\tL
\tL\tO\tO\tL\tO\tL\tL\tO
\tL\tO\tO\tL\tO\tL\tL\tO
\tO\tL\tL\tO\tL\tO\tO\tL
\cdots.\]
In the present paper, we prove that the subsequence $n\mapsto \mathsf{t}\bigl(n^3\bigr)$ attains each of its two values with asymptotic frequency $1/2$, more precisely,
%{{{ eqn_result
\begin{equation*} %\label{eqn_result}
\lim_{N\rightarrow\infty}
\frac1N
\#\bigl\{
n\in\{0,\ldots,N-1\}:\mathsf{t}\bigl(n^3\bigr)=\tO
\bigr\}
=1/2.
\end{equation*}
%}}} eqn_result

Being an \emph{automatic sequence}~\cite{AlloucheShallit1999,AlloucheShallit2003}, the Thue--Morse sequence it is closely linked to the base-$q$ expansion of integers, where $q=2$.
Assume that $q\geq2$ is an integer.
Every $n\in\mathbb N$ can be written in a unique way as
\begin{equation*} %\label{eqn_unique_baseq_expansion}
n=\sum_{\ijkl=0}^{\ell-1} \digit_\ijkl q^\ijkl,
\end{equation*}
where $\ell\geq 0$ is an integer and
$\bigl(\delta_0,\ldots,\delta_{\ell-1}\bigr)\in\{0,\ldots,q-1\}^{\ell}$ (the \emph{base-$q$ expansion of $n$}),
and either $\ell=0$ or $\delta_{\ell-1}\neq 0$.
Due to this uniqueness, functions $n\mapsto \delta_\ijkl(n)$
and $n\mapsto\ell(n)$ are defined (reusing notation),
and we call $\delta_\ijkl(n)$ the
\emph{digit} of $n$ (in base $q$) \emph{at the index} $\ijkl$, and $\ell(n)$ the \emph{length of the base-$q$ expansion} of $n$.
%the coefficients $\tilde\delta_\ijkl$ define functions of $n$, which we denote by $\delta_\ijkl$. %(n)$ for the base-$q$ digit of $n$ at position $\ijkl$,
%Let $\ell(n)$ denote the quantity $\ell$ appearing in~\eqref{eqn_unique_baseq_expansion} --- the \emph{length} of the base-$q$ expansion.
For simplicity of notation, we also set $\delta_\ijkl(n)=0$ for $\ijkl\geq \ell(n)$.
Note that the base-$q$ expansion of $0$ is the empty string, and $\ell(0)=0$.

We define
\begin{equation*} %\label{eqn_def_digitsum}
\digitsum_q(n)\eqdef\sum_{\ijkl=0}^{\ell(n)-1}\digit_\ijkl(n)
=\sum_{\ijkl\geq 0}\digit_\ijkl(n).
\end{equation*}
This number is the
%\emph{binary sum of digits of $n$} or
%\emph{base-$q$ sum of digits of $n$} or
\emph{sum of digits of $n$ in base $q$},
and $\digitsum_q$ is the
%\emph{binary sum-of-digits function} or
%\emph{base-$q$ sum-of-digits function} or
\emph{sum-of-digits function in base $q$}.
%TODO lexicographically largest, from the greedy algorithm

For prime numbers $p$, The function $\digitsum_p(n)$ features prominently in
\namehighlight{Legendre}'s formula on the $p$-valuation of $n!$ (the exponent of the largest power of $p$ dividing $n!$):
\begin{equation}\label{eqn_legendre}
\nu_p(n!)=\frac{n-\digitsum_p(n)}{p-1}.
\end{equation}
Indeed, a necessary and sufficient condition for $m\in A_k\eqdef[1,n]\cap p^k\mathbb N$ is
\[\delta_0(m)=\delta_1(m)=\cdots=\delta_{k-1}(m)=0\quad
\textsf{and}\quad
\sum_{j\geq k}\delta_j(m)p^j\leq \sum_{j\geq k}\delta_j(n)p^j.\]
Since the valuation $\nu_p$ is additive, it follows that
\begin{align*}\nu_p(n!)
&=
\sum_{1\leq m\leq n}
\nu_p(m)
=
\sum_{1\leq k<\ell(n)}
\#\bigl([1,n]\cap p^k\mathbb N\bigr)
\\&=
\sum_{1\leq k<\ell(n)}
\sum_{k\leq j<\ell(n)}
\delta_j(n)p^{j-k}
=
\sum_{1\leq j<\ell(n)}
\delta_j(n)
\sum_{1\leq k\leq j}
p^{j-k}
\\&=
\sum_{0\leq j<\ell(n)}
\delta_j(n)
\frac{p^j-1}{p-1}
=\frac{n-\digitsum_p(n)}{p-1}.
\end{align*}

In particular, we recover the simple fact that $\digitsum_p(n)\equiv n\bmod p-1$.
Since $p^k\equiv 1\bmod p-1$, this latter identity is also valid for general integer bases $p\ge2$, a principle that forms the basis of procedures such as ``casting out nines'' and 
\guillemotleft\,preuve par neuf\,\guillemotright.
%The strong link expressed by Legendre's formula~\eqref{eqn_legendre}
In particular, we can clearly see the origin of the condition $\gcd(m,q-1)=1$, commonly present in theorems on the distribution of $\digitsum_q(n)$ modulo $m$~\cite{Besineau1972,DrmotaMauduitRivat2011,Gelfond1968,Kim1999,MauduitRivat2009,MauduitRivat2010}.

Studying the $p$-valuation of binomial coefficients has a long history, going back at least to the $19^{\text{th}}$ century.
Assume that $p$ is a prime.
Kummer~\cite{Kummer1852} proved that the highest power $p^k$ dividing $\binom nt$ is the number of \emph{borrows} occurring in the subtraction $n-t$ in base $p$.
For overviews on the topic, we refer to the surveys~\cite{Granville1992,Singmaster1980} by Granville and Singmaster, respectively.

Applying Legendre's formula~\eqref{eqn_legendre} three times, we arrive at the representation
\begin{equation*} %\label{eqn_three_times}
\digitsum_p(n+t)-\digitsum_p(n)=\digitsum_p(t)-(p-1)\nu_p\biggl(\binom{n+t}n\biggr).
\end{equation*}
This intimate connection between the $p$-valuation of binomial coefficients, carries, and the sum-of-digits function in base $p$ provides motivation, in the opinion of the author, to study the function $\digitsum_p$, and its correlations~\cite{SpiegelhoferWallner2021}.

O.~A.~\namehighlight{Gelfond}~\cite{Gelfond1968} considered the sum-of-digits function in base $q$ along arithmetic progressions.
He proved that, as soon as $\gcd(m,q-1)=1$, the base-$q$ sum-of-digits function along $a+d\mathbb N$, reduced modulo $m$, attains each value $\in\{0,\ldots,m-1\}$ with asymptotic density $1/m$.

The three questions he posed at the end of that paper came to be known as ``Gelfond problems (1967/1968)''. Their content can be summarized, informally, as follows.

\begin{enumerate}
\item Study the joint distribution in residue classes of sum-of-digits functions in different bases.
\item Find the number of prime numbers $p\leq x$ such that $\digitsum_q(p)\equiv \ell\bmod m$.
\item Given a polynomial $P$ such that $P(n)\in\mathbb N$ for $n\in\mathbb N$, study the distribution of $\digitsum_q(P(n))$ in residue classes.
\end{enumerate}

The first problem was settled by \namehighlight{B\'esineau}~\cite{Besineau1972} and \namehighlight{Kim}~\cite{Kim1999}.
\namehighlight{Mauduit} and \namehighlight{Rivat}, in two major papers~\cite{MauduitRivat2009,MauduitRivat2010} solved the second problem as well as a special case of the third problem, concerning the polynomial $P(x)=x^2$.
For the case of the
Thue--Morse sequence,
\begin{equation*} %\label{eqn_TM_def}
\mathsf{t}(n)=\digitsum_2(n) \bmod 2,
\end{equation*}
the latter result yields the following statement.
\begin{equation}\label{eqn_MR2009}
\mbox{\begin{minipage}{0.9\textwidth}
The set of positive integers $n$ such that $\mathsf t(n^2)=\mathtt 0$
has asymptotic density $1/2$.\end{minipage}}
\end{equation}
(More precisely, Mauduit and Rivat handled all bases $q\ge2$, and provided an error term too.)
Clearly, replacing $\tO$ by $\tL$ yields an equivalent statement.

The sequence
$\bigl(\mathsf t(n^2)\bigr)_{n\ge0}
=
\mathtt{0110110111110010111110110100\cdots} %1101111110111101101001}\cdots
$
can be found as entry \href{https://oeis.org/A228039}{\texttt{A228039}} in Sloane's \emph{Online Encyclopedia of Integer Sequences}~\cite{OEIS}.
Generalizing~\eqref{eqn_MR2009}, it was even shown to be \emph{normal} by \namehighlight{Drmota}, \namehighlight{Mauduit}, and \namehighlight{Rivat}~\cite{DrmotaMauduitRivat2019}: each finite sequence of $\tO$s and $\tL$s, of length $L\ge1$, appears with asymptotic frequency $2^{-L}$ in this sequence.

%This statement in particular is the motivation for the present paper.
%
Beginning with the paper~\cite{MauduitRivat2009} on the digits of $n^2$, interest in the third Gelfond problem~\cite{Gelfond1968} came up again.
%All of the numerous attempts to
Efforts invested in order to extend Mauduit and Rivat's results on $\digitsum_q(n^2)$~\cite{MauduitRivat2009} did not yet produce a polynomial of higher degree for which analogous results hold for all bases (for example, uniform distribution in residue classes $a+m\mathbb Z$, where $\gcd(m,q-1)=1$).
Partial results exist~\cite{DartygeTenenbaum2006,DrmotaMauduitRivat2011,Moshe2007,Stoll2012}, in particular, it is known that $\mathsf t(n^3)=\tO$ infinitely often; moreover, the distribution of $s_q$ in residue classes, along polynomials, is understood for ``large bases'' $q$.

%}}} sec_intro

%\input{cubes_main.tex}
%{{{ sec_main
\section{The main result}
In the present paper, we settle the case $\digitsum_q\bigl(P(n)\bigr)\bmod m$, for $(q,m)=(2,2)$ and the polynomial $P(n)=n^3$, of Gelfond's third problem from 1967/1968.
\begin{theorem}\label{thm_main}
There exist real numbers $c>0$ and $C$ such that for all $x\ge1$,
\begin{equation}\label{eqn_main}
\left\lvert
\#\bigl\{n<x:
\mathsf{t}\bigl(n^3\bigr)=\tO\bigr\}-\frac x2\right\rvert
\leq Cx^{1-c}.
\end{equation}
\end{theorem}
This theorem is a statement on the sequence
\[\bigl(\mathsf t(n^3)\bigr)_{n\ge0}
=
\mathtt{0110100010000100100000010110\cdots},\]
which is recorded as entry \href{https://oeis.org/A365089}{\texttt{A365089}} in the OEIS~\cite{OEIS}.
In particular, the following holds
\begin{equation}\label{eqn_main_corollary}
\mbox{The sets of $n\in\mathbb N$ satisfying $\mathsf t(n^3)=\tO$ or $\tL$, respectively, have asymptotic density $1/2$.}
\end{equation}

\bigskip
\noindent
We introduce an exponential sum $\Szero$, whose smallness is sufficient for~\eqref{eqn_main} to hold for all $x\ge1$.
Given an integer $\nu\ge0$ and a real number $\xi$, let us define
\begin{equation}\label{eqn_S1_def}
\Szero\bigl(\nu,\xi\bigr)
\eqdef
\frac1{2^\nu}
\sum_{0\leq n<2^\nu}
\e\left(\frac12\smallspace\digitsum_2\bigl(n^3\bigr)
+n\xi
\right),
\end{equation}
where $\e(x)=\exp(2\pi i x)$.
Applying Lemma~\ref{lem_vinogradov} in order to extend the summation index in the
 main theorem to the next higher power $2^\nu$ of $2$,
it can be seen easily that it is sufficient to prove the following proposition.
%{{{ prp_sufficient
\begin{proposition}\label{prp_sufficient}
There exist absolute constants $c>0$ and $C$ such that for all integers $\nu\ge0$,
\begin{equation*} %\label{eqn_sufficient}
\sup_{\xi\in\mathbb R}
\smallspace
\bigl\lvert \Szero\bigl(\nu,\xi\bigr)\bigr\rvert
\leq C
e^{-\nu c}.
\end{equation*}
\end{proposition}
%}}} prp_sufficient

The remainder of this paper is devoted to the proof of Proposition~\ref{prp_sufficient}.
\subsection{Notation}\label{sec_notation}
Henceforth, only the sum-of-digits function in base $2$ will be considered.
We write $\digitsum\eqdef\digitsum_2$ and simply call this the \emph{sum-of-digits function}.
Digits $\digit_\ijkl(n)$ are understood to be digits in base $2$ from now on.

\smallskip\noindent
Let us define periodic functions of a real variable $x$, with period $1$: 
\[
\begin{aligned}
\e(x)&\eqdef\exp(2\pi i x),&
\lfloor x\rfloor&\eqdef\max\{n\in\mathbb Z:n\leq x\},\\
\lVert x\rVert&\eqdef\min\bigl\{\lvert x-n\rvert:n\in\mathbb Z\bigr\},& % (the distance of $x$ to the nearest integer).
\{x\}&\eqdef x-\lfloor x\rfloor.
\end{aligned}
\]

\smallskip\noindent
We use the shorthand
\[\logp(x)\eqdef
\begin{cases}1,&x<e;\\
\log x,&x\geq e,
\end{cases}
\]
which avoids some case distinctions. In particular, $\sum_{1\leq j\leq x}1/j\ll\logp(x)$ for all $x\ge0$.

\smallskip\noindent
We make use of the Iverson bracket notation $[\mathcal R]$, for a relation $\mathcal R$.
It yields $1$ if $\mathcal R$ is satisfied, and $0$ otherwise.

\smallskip\noindent
If $A,B\subseteq\mathbb R$, we write
\[A+B\eqdef\{a+b:(a,b)\in A\times B\},\] and similarly for ``$-$''.
By slight abuse of notation, we will also write
$A\pm x\eqdef\{a\pm x:a\in A\}$ for $x\in\mathbb R$ and $A\subset\mathbb R$.

\smallskip\noindent
For a subset $I\subseteq\mathbb N$, and $n\in\mathbb N$, let
%{{{ eqn_restricted_integer
\begin{equation*} %\label{eqn_restricted_integer}
n^I\eqdef\sum_{j\in I}\digit_j(n)q^j.
\end{equation*}
%}}} eqn_restricted_integer

\smallskip\noindent
Assume that $I\subseteq \mathbb R$ is any set, and let $\indicator_I$ denote the indicator function of the set $I\subseteq\mathbb R$,
\[\indicator_I(x)=\begin{cases}1,&x\in I;\\0,&x\not\in I.\end{cases}\]
The set of natural numbers is denoted by $\mathbb N$, and contains $0$.
For an integer $n\ge1$, we set $\nu_2=\max\{k\ge0:2^k\mid n\}$.

\smallskip\noindent
In order to avoid ambiguities, we do not use terms like ``the $\ijkl$-th digit of $n$'' in order to denote the coefficient $\digit_j$ in the expansion $n=\sum_{0\leq i\leq \nu}\digit_iq^i$, since in everyday language, a list begins with its \emph{first} element rather than its ``$0$-th element''.
We would rather use ``the digit of $n$ at index $\ijkl$'', or similar phrases.

%{{{
\subsection{Description of the proof of Theorem~\ref{thm_main}}
\subsubsection{Overview}
A central tool in our proof is \namehighlight{van der Corput}'s inequality (Lemma~\ref{lem_vdC}),
which already proved very effective in Mauduit and Rivat's papers~\cite{MauduitRivat2009,MauduitRivat2010} on the sum of digits of squares and primes, respectively.
This lemma reduces the estimation of a sum $\sum_{n\in I}a_n$ to the estimation of certain \emph{correlations}
\begin{equation}\label{eqn_correlation}
\sum_{n\in I\cap (I-r)}a_n\overline{a_{n+r}},
\end{equation}
where $r$ may be relatively small compared to $\lvert I\rvert$.
The usefulness of this statement in the context of sum-of-digits functions can be understood, informally, by a monotonicity statement.
Consider the function $n\mapsto [n]_2$, which maps the nonnegative integer $n$ to its binary expansion, which is an element of the set
\[\mathcal A\eqdef \bigl\{z\in\{0,1\}^\mathbb N: z_i=1 \mbox{ only finitely often}\bigr\}.\]
(Note that we consider improper expansions, padded with zeros to the left.)
With respect to the usual order on $\mathbb N$ and the lexicographical order
\[x<y\Leftrightarrow x\neq y\mbox{ and }
(x_i,y_i)=(0,1)\mbox{ for }i\eqdef\max\{j\ge0:x_j\neq y_j\}
\]
on $\mathcal A$, this assignment is increasing.
Noting also that the length of the binary expansion of $n$ only grows logarithmically in $n$, we see --- informally --- that adding a ``small'' value $r$ to a given integer $n$ does not change ``too many'' binary digits, except for ``few'' cases where carry propagation over a long distance occurs.
In the definition~\eqref{eqn_S1_def} of $\Szero$, the sum-of-digits function appears only in the exponential.
Evaluating the product $a_n\overline{a_{n+r}}$ in the expression~\eqref{eqn_correlation}, it follows that the contributions of the digits with high indices cancel (except for the ``few'' exceptional cases), which leads to the \emph{truncated sum-of-digits function}
\[\digitsum^{[0,\lambda)}(n)\eqdef \digitsum\bigl(n\bmod 2^\lambda\bigr).\]
%Of course, analogous statements hold for arbitrary bases $q$.
In the paper~\cite{MauduitRivat2009} by Mauduit and Rivat, this idea was extended in order to discard digits at the lowest indices too:
van der Corput's inequality is applied again, this time $r$ is a multiple of $2^\mu$.
This leaves the lowest $\mu$ digits unchanged, and we arrive at the \emph{doubly truncated sum-of-digits function} $\digitsum^{[\mu,\lambda)}$.
Having disposed of many digits, the truncated function can be controlled, which led to a proof of the main results in~\cite{MauduitRivat2009,MauduitRivat2010}.

The author~\cite{Spiegelhofer2020} continued the investigation of this idea.
We did not stop after two applications of van der Corput's inequality,
but discarded digits with indices in certain intervals (``windows of digits'') iteratively.
This proved very useful in the context of sum-of-digits functions along \emph{very sparse arithmetic progressions}.
The common difference of such a progression may be an arbitrarily large power $N^K$ of its length $N$, and consequently $K\log_2N$ digits have to be taken into account.
After the iterated truncation of digits --- comparable to salami slicing --- we are left only with a small window $I$ of contributing digits.
The size $\lvert I\rvert$ of this residual interval is considerably smaller than $\log_2N$, while the length $N$ of summation is almost unchanged.
In this way we obtain a certain uniform distribution result of the digits in $I$, and thus the main difficulty --- ``too many contributing digits'' --- has been overcome.
The performed iterated application of van der Corput's inequality leads us to expressions of the type
\[
\prod_{\varepsilon_0,\ldots,\varepsilon_L\in\{0,1\}}
f\bigl(n+\varepsilon_0r_0+\cdots+\varepsilon_Lr_L\bigr),\]
where $f(n)=\e\bigl(\tfrac12{\digitsum^I(n)}\bigr)$.
This is reminiscent of a \namehighlight{Gowers} uniformity norm, as treated by \namehighlight{Konieczny}~\cite{Konieczny2019} for the case of the Thue--Morse sequence, and by \namehighlight{Byszewski}, \namehighlight{Konieczny}, and \namehighlight{M\"ullner}~\cite{ByszewskiKoniecznyMuellner2020} for general automatic sequences.
Along these lines, we arrived at a  statement on the \emph{level of distribution} of the Thue--Morse sequence~\cite{Spiegelhofer2020}.

%extended this technique: we apply van der Corput-type inequalities \emph{repeatedly}
%Using suitable choices of the parameters, the contributions of the digits in certain intervals will cancel.
%In this manner, we may discard certain digits, which leads to truncated sum-of-digits functions $\digitsum^I$.
%Mauduit and Rivat~\cite{MauduitRivat2009,MauduitRivat2010} introduced van der Corput's inequality to the study of digital problems.
\namehighlight{Drmota}, \namehighlight{M\"ullner}, and the author~\cite{DrmotaMuellnerSpiegelhofer2021} extended this method to the \namehighlight{Zeckendorf} sum-of-digits function~\cite{DrmotaMuellnerSpiegelhofer2021}.
This function yields the minimal number of Fibonacci numbers needed to write a nonnegative integer $n$ as their sum.
The level of distribution of the Zeckendorf sum-of-digits function forms an essential ingredient in the proof of the main theorems of that paper, which establishes theorems on the representation of prime numbers as sums of different, non-adjacent Fibonacci numbers.

%As soon as this digit reduction is completed, the main problem in the proofs of the level of distribution-statements has been overcome.
%The few digit combinations on this small window are attained, simply put, with uniform probability.
%Using a \emph{Gowers uniformity norm}~\cite{Gowers2001}, this allowed us to obtain a nontrivial gain in the appearing exponential sums.
\subsubsection{The three main steps in the proof of Proposition~\ref{prp_sufficient}}
In the present paper, we will apply ``digit slicing'' again.
Before that, however, we have to take great care to eliminate the consequences of the nonlinearity introduced by $n^3$, which is the main difficulty and makes up the major part of the paper (see Sections~\ref{sec_linearize} and~\ref{sec_uncoupling}).
%\subsection{Detailed plan of the paper}

\smallskip
The proof of Proposition~\ref{prp_sufficient} starts in \textbf{Section~\ref{sec_lemmas}}, which contains a series of lemmas.
This section is followed by the main proof, which is structured into three main parts:
%\begin{equation}\label{eqn_three_steps}
\[
\mbox{$\mathrm I\cdot$ Linearization,\quad
$\mathrm{II}\cdot$ Uncoupling,\quad
$\mathrm{III}\cdot$ Elimination.
}
\]
%\end{equation}

These three steps are handled in Sections~\ref{sec_linearize} to~\ref{sec_uniform}.
Section~\ref{sec_finishing} combines the arguments from these three sections, which completes the proof.
We proceed to the description of our three main steps, which is somewhat more technical.
\begin{itemize}
\item[$\mathrm{I}\cdot$] Linearization.
In \textbf{Section~\ref{sec_linearize}}, we state and prove our \emph{key result}, Proposition~\ref{prp_linearize}.
This subsequently leads to the statement of Corollary~\ref{cor_linearize}, from which we will continue in the sections thereafter.
At this point, the cube in the argument of $\digitsum$ has already been removed,
leaving only a \emph{linear problem} to be handled (see~\eqref{eqn_K_def},~\eqref{eqn_Qprime_def},~\eqref{eqn_S9_linear3},~\eqref{eqn_linearize}). %Corollary~\ref{cor_linearize}).
Roughly speaking, using van der Corput's inequality (Lemma~\ref{lem_vdC}), combined with the \emph{carry lemma} (Lemma~\ref{lem_carry}), we eliminate the contribution of binary digits of $n^3$ above $\lambda=2\nu(1+\Xi)$, where $N=2^\nu$, and $\Xi$ is a small constant to be chosen later.
We decompose $\{0,\ldots,2^\nu-1\}$ into arithmetic progressions with difference $2^\rho$, where $\rho=(1-2\Xi)\nu$
%\LS{Replace $2\Xi$ by $\Xi$!}
is slightly smaller than $\nu$, see~\eqref{eqn_first_split},~\eqref{eqn_nL_split}.
Since $3\rho\ge\lambda$, the cubes do not enter in an essential way when we proceed along these progressions.
Digits on the window $[2\rho,\lambda)$, where we still have quadratic behaviour, are detected by means of a trigonometric (``Vaaler'') polynomial,
of degree $H$.
This detection transfers the remaining nonlinearity into a trigonometric polynomial.
By another application of van der Corput's inequality, the argument of this trigonometric polynomial becomes linear too. We wish to note, in order to avoid possible confusion, that in the actual proof this application of the inequality comes first, for technical reasons.
%The arithmetic progression is thinned out further~\eqref{eqn_nL_split}

Note that the length of our progressions --- which is $2^{\nu-\rho}\asymp 2^{2\Xi\nu}$ --- is considerably smaller than the number $2^{\lambda-2\rho}\asymp2^{6\Xi\nu}$ of digit combinations that we have to detect in order to remove the window $[2\rho,\lambda)$.

The remainder of the \emph{linearizing} part consists in accumulating the additional terms arising from the detection of digits in the \emph{critical interval} $[2\rho,\lambda)$.
To this end, we will use summation by parts~\eqref{eqn_S6_split} and Dirichlet approximation~\eqref{eqn_x_close} in a suitable manner.
After this procedure, the additionally generated terms are captured by a 
geometric sum $\varphi_H$, see~\eqref{eqn_S3_S8}.
We have thus exchanged the nonlinearity of the problem for an additional factor $\varphi_H$, which is recorded in Corollary~\ref{cor_linearize}.
%Adding a lot of technicalities completes the treatment of the key result.

\item[$\mathrm{II}\cdot$] Uncoupling.
In the just obtained Corollary~\ref{cor_linearize}, the sum-of-digits part --- contained in the expression $\Seight$ --- is \emph{coupled} to the newly introduced geometric sum $\varphi_H$,
in the sense that both factors depend on the outer summation variables.
In \textbf{Section~\ref{sec_uncoupling}}, we
\emph{uncouple} these two terms, by applying Dirichlet approximation a second time.
We choose an integer $T$ suitably, and decompose the summation (over $\nLO$) into arithmetic progressions with common difference $T$.
Along each of these progressions, the geometric sum $\varphi_H$ is almost constant due to Dirichlet's approximation theorem.
Therefore we can apply in a profitable manner the $(p,q)=(1,\infty)$-case of H\"older's inequality on this decomposition.
More precisely, Lemma~\ref{lem_glycerol} is used, consisting of two applications of H\"older's inequality.
Thus, the ``Dirichlet part'' (a geometric sum) and the ``sum-of-digits part'' separate, and can be handled independently.

At this point we wish to revisit a remark from the first proof step.
The length $2^{2\Xi\nu}$ of our innermost summation (over $\nLL$) is much shorter than the number $2^{6\Xi\nu}$ of detected digit combinations on the critical interval $[2\rho,\lambda)$, which in turn has to be smaller than the degree $H$ of the trigonometric polynomial.
Thus the innermost summation is much shorter than the additional summation over $h<H$.
We have therefore reason to wonder how this can even yield a nontrivial estimate.

%Section~\ref{sec_uncoupling} gives an answer to this question.
The answer is given by the two-step uncoupling process, started already in Section~\ref{sec_linearize}. In that section, we translated the additional summation of length $H$ into a geometric sum $\varphi_H$.
In Section~\ref{sec_uncoupling}, we separate $\varphi_H$ from the main part, containing sum-of-digits functions. This yields an \emph{average} in $x$ over $\sum_{0\leq h<H}\e(hx)$, hence only a \emph{logarithmic factor}!

A simplified form of this argument appears again, in the treatment of the error term~$\Efour$. Van der Corput's inequality is applied after detection of the digits in $[2\rho,\lambda)$.
The error that arises from the subsequent omission of digits with indices $\ge\lambda$ --- which is greater than $2^{2\rho-\lambda}$ --- obviously cannot be multiplied by the trivial bound $H$ without losing the nontrivial estimate.
However, it easily swallows the mean value of $\varphi_H$.

%Since we linearized our problem in Section~\ref{sec_linearize} before, and due to the sensible choice of $T$, Its contribution can therefore be controlled.
The sum-of-digits part $\Seight$, evaluated along arithmetic progressions with common difference $T$, remains.

\item[$\mathrm{III}\cdot$] Elimination.
We proceed to \textbf{Section~\ref{sec_uniform}}.
In order to obtain an upper bound for $\Seight$,
we start the procedure that removes slices of digits repeatedly.
This happens by iterated use of van der Corput's inequality, see Lemma~\ref{lem_vdC_iterated}, applied to the main summation variable $\nLL$ of the sum $\Seight$, which is defined in~\eqref{eqn_S9_def}.
Note that a certain average in $\nLO$ over the expression $\Seight$ is evaluated.
We consider $\nLO\in P$, where $P$ is an arithmetic progression with difference $T$, coming from part $\mathrm{II}\cdot$.
After this procedure, most digits have been removed.
Only the \emph{truncated sum-of-digits function} $\digitsum^{[c,d)}$ is left, which depends only on the digits with indices in a small interval $[c,d)$.
For this method to work, it will prove essential that the problem has been reduced to a linear one before, as repeated application of van der Corput's inequality does not play well with polynomials of higher degree inside the sum-of-digits function.
(Note also that passing to an arithmetic subsequence with common difference $T$ does not disrupt the linearity of the problem.)
%The first two applications of Lemma~\ref{lem_vdC} cause the presence of \emph{four} 
Clearing additional complications --- for example, we have to consider \emph{four slopes} synchronously, as opposed to the easier situation in~\cite{Spiegelhofer2020} ---
we obtain a \emph{Gowers norm} of the Thue--Morse sequence, for which estimates are available~\cite{ByszewskiKoniecznyMuellner2020,Konieczny2019}.
A nontrivial estimate for the ``sum-of-digits part'' $\Seight$ follows.
This completes the treatment of the main issues, and it only remains to tie up loose ends, which happens in Section~\ref{sec_finishing}.
\end{itemize}
\section{Lemmas}\label{sec_lemmas}
We state a series of lemmas:

\bigskip

\begin{tabular}{ll}
  The inequalities of\\
  \hspace{2em}van der Corput,& Lemmas~\ref{lem_vdC_generalized},~\ref{lem_vdC},~\ref{lem_vdC_iterated};\\
  \hspace{2em}H\"older (extremal case, twofold),& Lemma~\ref{lem_glycerol};\\
  \hspace{2em}Erd\H{o}s--Tur\'an--Koksma, & Lemma~\ref{lem_ETK};\\
  \hspace{2em}Koksma--Hlawka,&Lemma~\ref{lem_Koksma_Hlawka};\\
  %is is just an application of the extremal case of the H\"older inequality, $(p,q)=(1,\infty)$.
Vaaler approximation,&Lemma~\ref{lem_vaaler2}, Corollary~\ref{cor_interval_detection_shift};\\
extending a summation range, & Lemma~\ref{lem_vinogradov};\\
The ``large sieve equality'', & Lemma~\ref{lem_full};\\
The ``carry lemma'' for $\digitsum_2$, & Lemma~\ref{lem_carry};\\
reverse summation by parts, & Lemma~\ref{lem_summation_by_parts};\\
``odd elimination'', & Lemma~\ref{lem_odd_elimination}.
\end{tabular}

\bigskip

The following generalization of van der Corput's inequality was proved in our joint paper~\cite{DrmotaMuellnerSpiegelhofer2021} with Drmota and M\"ullner, and appears as Proposition~6.14 there.
%Compare also to Lemma~5 of \namehighlight{Rivat}--\namehighlight{Sargos}~\cite{RivatSargos2001}.

While this is by no means a deep theorem (the proof is an application of the Cauchy--Schwarz inequality, combined with ``double counting''), its particular formulation enabled us to prove a level of distribution-statement for the \emph{Zeckendorf sum-of-digits function}, and, consequently, the main theorems in our paper~\cite{DrmotaMuellnerSpiegelhofer2021}.
%This generalization Corput's inequality was used in order to allow for more variability in the choice of \emph{shifts} $s$ appearing in the correlations $x_m\overline{x_{m+s}}$.
%{{{ lem_vdC_generalized
\begin{lemma}[Generalized van der Corput inequality,~\cite{DrmotaMuellnerSpiegelhofer2021}]\label{lem_vdC_generalized}
Let $I$ be a finite interval in $\mathbb Z$ containing $M$ integers and $x_n\in\mathbb C$ for $n\in I$.
Assume that $\mathcal S\subseteq\mathbb Z$ is a finite nonempty set.
Then
%\begin{equation}\label{eqn_vdC_generalized}
\[
\left\lvert \sum_{n\in I}x_n\right\rvert^2
\leq \frac{M+\max \mathcal S-\min \mathcal S}{\lvert \mathcal S\rvert^2}
\sum_{(\sO,\sL)\in \mathcal S^2}
\,
\sum_{\substack{n\in (I-\sO)\cap(I-\sL)}}
%\sum_{\substack{m\in\mathbb Z\\m+\sO\in I\\m+\sL\in I}}
x_{n+\sO}\overline {x_{n+\sL}}.
\]
%\end{equation}
The right hand side is a nonnegative real number.
\end{lemma}
%}}} lem_vdC_generalized end

In particular, choosing $\mathcal S$ to be a finite arithmetic progression $\{0,M,2M,\ldots,(R-1)M\}$, we recover Lemme~17 in the paper~\cite{MauduitRivat2009} by Mauduit and Rivat.
That lemma, in turn, is a generalization of the classical inequality of van der Corput (the case $M=1$).
%{{{ lem_vdC
\begin{lemma}[Mauduit--Rivat]\label{lem_vdC}
Let $I\subseteq\mathbb Z$ be a finite interval containing $N$ integers, and
$(z_n)_{n\in I}$ a family in $\mathbb C$.
For all integers $M\geq 1$ and $R\geq 1$ we have
\[
  \left\lvert \sum_{n\in I}a_n\right\rvert^2
\leq
  \frac{N+M(R-1)}R
  \sum_{\substack{r\in\mathbb Z\\\lvert r\rvert<R}}\left(1-\frac {\lvert r\rvert}R\right)
  \sum_{n\in I\cap (I-Mr)}z_{n}\overline{z_{n+Mr}}.
\]
The right hand side is a nonnegative real number.
\end{lemma}
%}}} lem_vdC

Iterated application of Lemma~\ref{lem_vdC} yields the following statement.
%{{{ lem_vdC_iterated
\begin{lemma}\label{lem_vdC_iterated}
%Let $f\in\mathbb R[X]$ be a polynomial of degree $e\ge0$ and
Let $Q\geq 1$ be an integer.
Assume that $J$ is a finite nonempty interval in $\mathbb Z$,
and $g:J\rightarrow \{z\in\mathbb C:\lvert z\rvert=1\}$.
For all integers $M_0,\ldots,M_{Q-1}\ge1$ and $R\ge1$, we have
\begin{equation}\label{eqn_vdC_iterated}
\begin{aligned}
\Biggl\lvert
\frac1{\lvert J\rvert}
\sum_{n\in J}
g(n)%\e\bigl(f(n)\bigr)
\Biggr\rvert^{2^Q}
&\ll
\frac1{R^Q}
\sum_{r\in\{1,\ldots,R-1\}^Q}%0\leq r_0,\ldots,r_{e-1}<R}
\bigl\lvert
K\bigl(r_0M_0,\ldots,r_{Q-1}M_{Q-1}\bigr)
\bigr\rvert
\\&\quad
+%\frac1{\lvert J\rvert}
\LandauO
\Biggl(
\frac{\bigl(M_0+\cdots+M_{Q-1}\bigr)R}{\lvert J\rvert}
+\frac1R
\Biggr),
\end{aligned}
\end{equation}
where
\[
K\bigl(m_0,\ldots,m_{Q-1}\bigr)
\eqdef
\frac1{\lvert J\rvert}
\sum_{n\in J}
\prod_{\varepsilon\in\{0,1\}^Q} %(\varepsilon_0,\ldots,\varepsilon_{e-1})
(-1)^{\lvert\varepsilon\rvert}
g\Biggl(n+\sum_{0\leq\ell<Q}\varepsilon_\ell m_\ell\Biggr).
\]

The implied constant depends only on $Q$.
\end{lemma}
%}}} lem_vdC_iterated
\begin{proof}
For $Q=1$, this statement easily follows from Lemma~\ref{lem_vdC}, where the summand $r_0=0$ yields the error $1/R$, and
omission of the condition $n+M_0r_0\in J$
causes an error bounded by $RM_0/\lvert J\rvert$.
Assume that the statement has already been established for some $Q\ge1$.

The left hand side of~\eqref{eqn_vdC_iterated} as well as the first term on the right hand side are bounded by $1$ in absolute value.
We may therefore omit, at the cost of a bigger implied constant, the mixed term  and the square of the error term when squaring the right hand side.

An application of the Cauchy--Schwarz inequality, followed by Lemma~\ref{lem_vdC}, and appending an error $RM_Q/\lvert J\rvert+1/R$ as in the base case finishes the proof by induction.
\end{proof}

A twofold application of the H\"older inequality yields the following lemma.
%{{{ lem_glycerol
\begin{lemma}\label{lem_glycerol}
Assume that $I\subseteq \mathbb Z$ is a finite nonempty set, and
$f(n),g(n)\in\mathbb C$ for $n\in I$.
%,  $H\ge1$ an integer, and set
Set
%\begin{equation}\label{eqn_grat}
\[
S\eqdef \sum_{n\in I}
f(n)\smallspace g(n). %\varphi_H(f(n)). %\alpha n+\beta).
\]
%\end{equation}
Assume that $\mathcal P$ is a partition of $I$.
Then
%\begin{equation}\label{eqn_twofold_Hoelder}
\[
\lvert S\rvert
\leq
\sum_{P\in\mathcal P}
\sup_{n\in P}
\bigl\lvert f(n)\bigr\rvert %f(n)=\varphi_H(n\alpha+\beta)
\times
\sup_{P\in\mathcal P}
\sum_{n\in P}
\bigl\lvert g(n)\bigr\rvert.
\]
%\end{equation}
\end{lemma}
%}}}

We will approximate the $1$-periodic function
\[\psi:x\mapsto x-\lfloor x\rfloor-\frac 12,\]
and, subsequently, the indicator function $\indicator_{I+\mathbb Z}$, where $I$ is an interval,
by trigonometric polynomials known as {\sc Vaaler} polynomials (see \namehighlight{Graham}--\namehighlight{Kolesnik}~\cite[Theorem A.6]{GrahamKolesnik1991}, and Vaaler~\cite{Vaaler1985}).
Let $\phi$ be the continuous extension to $[-1,1]$ of the function

\[ %\begin{equation}\label{eqn_phi_def}
\phi(t)\eqdef \pi t\bigl(1-\lvert t\rvert\bigr)\cot\pi t+\lvert t\rvert.
\] %\end{equation}
%The function $\phi$ is increasing in $[-1,0]$ and decreasing in $[0,1]$, with minima equal to $0$ and maximum equal to $1$, see Lemma~\ref{lem_phi_decreasing} below.
%and
%\[\phi_H(x)\eqdef \phi(x/H)\]
%for $\lvert x\rvert\leq H$.

%{{{ lem_vaaler2
\begin{lemma}\label{lem_vaaler2}
Let $H$ be a positive integer.
The trigonometric polynomial
\[ %\begin{equation}\label{eqn_psiH_def}
\psi_H(x)\eqdef-\sum_{\substack{h\in \mathbb Z\\1\leq\lvert h\rvert<H}}
\bigl(2\pi ih\bigr)^{-1}\phi\bigl(h/H\bigr)\e(hx).
\] %\end{equation}
satisfies
\[ %\begin{equation}\label{eqn_vaaler}
\bigl\lvert \psi(t)-\psi_H(t)\bigr\rvert \leq
\kappa_H(t),
\] %\end{equation}
where
\[ %\begin{equation}\label{eqn_kappa_def}
\kappa_H(t)=\frac{1}{2H}\sum_{\substack{h\in\mathbb Z\\\lvert h\rvert<H}}\left(1-\frac{\lvert h\rvert}H\right)\e(ht).
\] %\end{equation}
\end{lemma}
%}}} lem_vaaler2

%
%%{{{ lem_vaaler
%\begin{lemma}\label{lem_vaaler}
%Let $H$ be a positive integer.
%There exists a family $(a_H(h))_{-H<h<H}$ in $[0,1]$
%such that
%\begin{equation}\label{eqn_vaaler}
%\bigl\lvert \psi(t)-\psi_H(t)\bigr\rvert \leq
%\kappa_H(t)
%\end{equation}
%for all $t\in\mathbb R$, where
%\[
%\begin{aligned}
%%    \psi(x)&=\{x\}-\frac 12,\\
%    \psi_H(t)&=-\frac{1}{2\pi i}\sum_{\substack{h\in\mathbb Z\\1\leq \lvert h\rvert<H}}\frac{a_H(h)}{h}\e(ht),\\
%   \kappa_H(t)&=\frac{1}{2H}\sum_{\substack{h\in\mathbb Z\\\lvert h\rvert<H}}\left(1-\frac{\lvert h\rvert}H\right)\e(ht).
%\end{aligned}
%\]
%\end{lemma}
%%}}} lem_vaaler

Note that $2H\kappa_H(t)$ is the ($1$-periodic) Fej\'er kernel, which attains only nonnegative real values.
%expanding the square, one readily shows the identity
%\[  \sum_{\substack{h\in\mathbb Z\\\lvert h\rvert<H}}\bigl(H-\lvert h\rvert\bigr)e(hx)
%    =\left\lvert \sum_{h=0}^{H-1}\e(hx)\right\rvert^2,  \]
%valid for all integers $H\ge0$ and for all real numbers $x$.
From Lemma~\ref{lem_vaaler2}, we can easily obtain a trigonometric approximation of an indicator function $\indicator_{[\alpha,\beta)+\mathbb Z}$.
Applying the identity
%a rotation (by $z$), and 
\[\indicator_{[\alpha,\beta)+\mathbb Z}(x)
=(\beta-\alpha)+\psi(x-\beta)-\psi(x-\alpha) \qquad\mbox{($0\leq \alpha\leq\beta\leq 1$, $x\in\mathbb R$),}
\]
we obtain the following important corollary.
%{{{ cor_interval_detection_shift
\begin{corollary}\label{cor_interval_detection_shift}
Assume that $0\leq \alpha\leq \beta\leq 1$, and that $H\geq1$ is an integer.
The trigonometric polynomials

\[ %\begin{equation}\label{eqn_psiH_interval_def}
\begin{aligned}
\psi_{\alpha,\beta,H}(x)&\eqdef
\sum_{\substack{h\in \mathbb Z\\0\leq\lvert h\rvert<H}}
a_h(\beta-\alpha,H)
e\bigl(h(x-\alpha)\bigr),\\
\kappa_{\alpha,\beta,H}&\eqdef\sum_{\substack{h\in\mathbb Z\\-H<h<H}}
b_h(\alpha,\beta,H)\e(hx),
\end{aligned}
\] %\end{equation}
where
\[ %\begin{equation}\label{eqn_vaaler_coeff_shift}
\begin{aligned}
a_0(\delta,H)&\eqdef\delta;\\
a_h(\delta,H)&\eqdef
-\bigl(2\pi ih\bigr)^{-1}\phi\bigl(h/H\bigr)
\bigl(1-\e(-\delta h)\bigr)
&&\mbox{for }1\leq\lvert h\rvert<H;\\
b_h(\alpha,\beta,H)&
\eqdef\frac1{2H}\left(1-\frac{\lvert h\rvert}{H}\right)
\bigl(\e(-h\alpha)+\e(-h\beta)\bigr)
&&\mbox{for }\lvert h\rvert<H,
\end{aligned}
\] %\end{equation}
%where
%\begin{equation}\label{eqn_}
%\begin{aligned}
%\end{aligned}
%\end{equation}
satisfy
\[ %\begin{equation}\label{eqn_vaaler_interval}
\bigl\lvert \indicator_{[\alpha,\beta)+\mathbb Z}(t)-\psi_{\alpha,\beta,H}(t)\bigr\rvert \leq
\kappa_{\alpha,\beta,H}(t).
\] %\end{equation}
In particular, the right hand side of this inequality is a nonnegative real number.
\end{corollary}
%}}} cor_interval_detection_shift
%The function $\phi$ is actually decreasing on $[0,1]$, a fact that we will exploit when using summation by parts.
%%{{{ lem_phi_decreasing
%\begin{lemma}\label{lem_phi_decreasing}
%The function $\phi$ satisfies $\phi(x)=\phi(-x)$, and is decreasing on $[0,1]$.
%\end{lemma}
%%}}} lem_phi_decreasing
%\begin{proof}
%\ldots\ldots
%\end{proof}
%
Assume that $d\ge1$ is an integer and $\alpha=(\alpha(m))_{m\in I}$ a sequence in $\mathbb R^d$, where $I\subseteq \mathbb N$ is an interval.
Let $\mathbb T^d$ be the $d$-dimensional torus:
in this paper, this will just be the set $[0,1)^d$.
An \emph{interval} in $\mathbb T^d$ is a subset obtained from an axis-parallel box in $\mathbb R^d$, reduced modulo $1\times\cdots\times 1$.
For integers $M$ such that $[0,M)\subseteq I$,
let us define the discrepancy
\[ %\begin{equation}\label{eqn_discrepancy_def}
D_M(\alpha)
\eqdef
\sup_{J\subseteq \mathbb T^d\text{ interval}}
\Biggl\lvert
\frac1M\sum_{0\le m<M}
\indicator_J\bigl(\alpha(m)\bigr)
-
\lambda_d(J)
\Biggr\rvert,
\] %\end{equation}
where $\lambda_d$ is the $d$-dimensional Lebesgue measure.

The inequality of Erd\H{o}s--Tur\'an--Koksma is well known.
%{{{ lem_ETK
\begin{lemma}\label{lem_ETK}
Let $d$ be a positive integer.
For vectors $h,k\in\mathbb R^d$, define
\[\mu(h)\eqdef\prod_{0\leq i<d}\max\bigl(1,\lvert h_i\rvert\bigr)
\quad\mbox{and}\quad
h\cdot k\eqdef\sum_{0\leq i<d}h_ik_i.\]
There exists a constant $C=C(d)$ such that for all integers $N\geq 1$, all sequences $x=(x_j)_{0\leq j<N}$ in $\mathbb R^d$, and all integers $H\geq 1$ we have
\begin{equation}\label{eqn_ETK}
D_N(x)\leq C
\left(\frac 1H+
\sum_{\substack{h\in \mathbb Z^d\\0<\lVert h\rVert_\infty<H}}
\frac 1{\mu(h)}
\left\lvert\frac 1N\sum_{0\leq n<N}\e(h\cdot x_n)
\right\rvert
\right).
\end{equation}
\end{lemma}
%}}}
We will also need the Koksma--Hlawka inequality.
%{{{ lem_Koksma_Hlawka
\begin{lemma}\label{lem_Koksma_Hlawka}
Let $f:[0,1]\rightarrow\mathbb R$ have bounded variation $V(f)$ in the sense of Hardy and Krause,
and $x=(x_n)_{0\leq n<N}$ a sequence in $[0,1)$.
Then
%\begin{equation}\label{eqn_KH}
\[
\Biggl\lvert
\frac1N\sum_{0\leq n<N}f(x_n)-\int_0^1f(t)\,\mathrm dt\Biggr\rvert
\leq V(f)D_N(x).
\]
%\end{equation}
\end{lemma}
%}}} lem_Koksma_Hlawka
\begin{remark}
The bounds for the discrepancy that we get from the Erd\H{o}s--Tur\'an--Koksma inequality are indeed also valid for our ``rotation invariant'' discrepancy $D_N$. One proof of this inequality uses Vaaler approximation as above, and we easily see that we lose at most a constant factor when we admit general intervals in $\mathtt T^d$, as opposed to intervals $\prod_{0\leq i<d}[x_i,y_i)$, where $0\leq x_i\leq y_i\leq 1$.
\end{remark}

By means of the following standard lemma we may extend the range of a summation, introducing only a logarithmic factor into our estimates.
%{{{ lem_vinogradov
\begin{lemma}\label{lem_vinogradov}
Let $x\leq y\leq z$ be real numbers and $a_n\in\mathbb C$ for
$n\in [x,z)\cap\mathbb Z$.
Then
\begin{equation*}
\left\lvert\sum_{x\leq n<y}a_n\right\rvert
\leq
\int_0^1{\min\left\{\lceil y\rceil-x,\frac1{2\lVert \xi\rVert}\right\}
         \left\lvert\sum_{x\leq n<z}a_n\e(n\xi)\right\rvert}
\,\mathrm d\xi.
\end{equation*}

\end{lemma}
%}}} lem_vinogradov
%{{{ Proof of lem_vinogradov
\begin{proof}
Since
\[\int_0^1\e(k\xi)\,\mathrm d\xi=\begin{cases}1,&k=0;\\0,&k\in\mathbb Z\setminus\{0\},\end{cases}\] it follows that
\begin{equation*}
\begin{aligned}
\sum_{x\leq n<y}a_n
&=\sum_{x\leq n<z}a_n\sum_{x\leq m<y}\delta_{n-m,0}
=\int_0^1\sum_{x\leq m<y}\e(-m\xi)
 \sum_{x\leq n<z}a_n\e(n\xi)\,\mathrm d\xi.
\end{aligned}
\end{equation*}
Using a geometric series, we obtain
\[\left\lvert \sum_{x\leq n<y}a_n\right\rvert\leq
\int_0^1\min\Bigl(\bigl\lvert[x,y)\cap\mathbb Z\bigr\rvert,2\bigl \lvert1-\e(-\xi)\bigr \rvert^{-1}\Bigr)
 \left\lvert\sum_{x\leq n<z}a_n\e(n\xi)\right\rvert\,\mathrm d\xi.
\]
By the inequality
\begin{equation}\label{eqn_cos_poly_bound}
\cos 2\pi x\leq 1-8\lVert x\rVert^2,
\end{equation}
valid for all real $x$, we have
\[\lvert 1-\e(-\xi)\rvert^2
=2(1-\cos 2\pi \xi)\geq
16\lVert\xi\rVert^2,\]
from which the statement follows.
The bound~\eqref{eqn_cos_poly_bound} can be shown easily, considering $x\in\{0,1/2\}$, where both sides are identical,
$x\in(0,1/4)$, where the first derivative of $1-8x^2-\cos2\pi x$ is positive, and $x\in(1/4,1/2)$, where the second derivative is negative.
\end{proof}
%}}} Proof of lem_vinogradov

%{[(
The following elementary ``large sieve equality'' will be used to uncouple the summation variable $j$, see~\eqref{eqn_S2_S3}.
%{{{ lem_full
\begin{lemma}\label{lem_full}
Let $M\geq 1$ be an integer, and $(a_m)_{0\leq m<M}$ a family in $\mathbb C$.
Then
\[ %\begin{equation}\label{eqn_full}
\sum_{0\leq h<M}
\left\lvert
\sum_{0\leq m<M}
a_m\e\bigl(-hmM^{-1}\bigr)
\right\rvert^2
=
M\sum_{0\leq m<M}\lvert a_m\rvert^2.
\] %\end{equation}
\end{lemma}
%}}} lem_full
%{{{ Proof of lem_full
\begin{proof}
Expanding the square and interchanging summations, we see that the left hand side equals
\begin{equation*}
%\sum_{0\leq h<M}
%\left\lvert
%\sum_{0\leq m<M}
%a_m\e\bigl(-hmM^{-1}\bigr)
%\right\rvert^2
%&=
\sum_{0\leq \nAprime<M}
\sum_{0\leq \nBprime<M}
a_{\nAprime}\overline{a_{\nBprime}}
\sum_{0\leq h<M}
\e\bigl(-h(\nAprime-\nBprime)M^{-1}\bigr)
=
M
\sum_{0\leq m<M}
a_m\overline{a_m}.
%
%M\sum_{0\leq m<M}\lvert a_m\rvert^2.
\qedhere
\end{equation*}

\end{proof}
%}}} Proof of lem_full
%)]}
%{[(
In the proof of our main theorem, we will make essential use of the truncated sum-of-digits function~\cite{MauduitRivat2009,MauduitRivat2010}.
For a set $I\subseteq \mathbb R$, let
\[ %\begin{equation}\label{eqn_truncated_def}
\digitsum^I(n)\eqdef\digitsum\bigl(n^I\bigr)=\sum_{\ijkl\in I\cap\mathbb N}\digit_\ijkl(n).
\] %\end{equation}
%Morever, for a nonnegative integer $\mu$, we use the shorthand
%\begin{equation}\label{eqn_interval_def} %(
%[\mu]\eqdef[0,\mu).%\{0,\ldots,\mu-1\}.  %]
%\end{equation}
Let $L\ge0$ be an integer.
The function                                       %(
$\digitsum^{[0,L)}:\mathbb N\rightarrow\mathbb N$, %]
which is the $2^L$-periodic continuation of the restriction $\digitsum\vert_{\{0,\ldots,2^L-1\}}$, will play a particularly important role in our proof.

The following ``carry lemma'' will enable us to discard the most significant digits in our sum-of-digits functions.
%{{{ carry lemma
\begin{lemma}\label{lem_carry}
%Carry lemma for n^3
Let $\lambda\geq 0$, $r\geq 0$, and $A,B$ be integers such that $0\leq A\leq B$.
Then
\begin{align*}
\hspace{3em}&\hspace{-3em}  %(
\#\bigl\{n\in [A,B):       %]
\bigl\lfloor n^3/2^\lambda\bigr\rfloor\neq
\bigl\lfloor (n+r)^3/2^\lambda\bigr\rfloor
\bigr\}
\\&\leq 
\bigl((B-A)B^2/2^\lambda+1\bigr)\bigl(\bigl(3B^2r+3Br^2+r^3\bigr)/(3A^2)+1\bigr).
%\bigl(B^3/2^\lambda+1\bigr)\bigl(\bigl(3B^2r+3Br^2+r^3\bigr)/(3A^2)+1\bigr).
%\bigl(7N^3/2^\lambda+1\bigr)\bigl(\bigl(12N^2r+6Nr^2+r^3\bigr)/(3N^2)+1\bigr).
\end{align*}
In particular, for each subset $M\subseteq\mathbb N$ we have
\[ %\begin{equation}\label{eqn_carry}
\begin{aligned}
\hspace{3em}&\hspace{-3em}  %(
\#\bigl\{n\in [A,B):       %]
\digitsum^M\bigl((n+r)^3\bigr)-\digitsum^M\bigl(n^3\bigr) %((
\neq \digitsum^{M\cap[0,\lambda)}\bigl((n+r)^3\bigr) 
-\digitsum^{M\cap[0,\lambda)}\bigl(n^3\bigr)\bigr\}        %]]
\\&\leq C
r\frac{B^2}{A^2}\bigl((B-A)B^22^{-\lambda}+1\bigr)
%\cdots.
%\\&
%140\,r(N^32^{-\lambda}+1).
%\left(\frac{7N^3}{2^\lambda}+1\right)
%\left(7r+1\right).
\end{aligned}
\] %\end{equation}
with $C=10/3$. %an absolute implied constant ($10/3$ is admissible). % independent of $\lambda$, $r$, and $N$ ($140$ 
\end{lemma}
\begin{proof}
Set $L\eqdef 3B^2+3Br^2+r^3$.
Since $(n+r)^3-n^3\leq L$ for all  %(
$n\in[A,B)$,                       %]
we only have to exclude the integers $n$ satisfying
\[ %\begin{equation}\label{eqn_to_exclude}  %(
n^3\in [-L,0)+2^\lambda\mathbb Z.       %]
\] %\end{equation}
Only in this case it can happen that the binary digits at indices $\ge\lambda$ change (that is,
$\lfloor n^3/2^\lambda\rfloor\neq\lfloor(n+r)^3/2^\lambda\rfloor$)
when passing from $n^3$ to $(n+r)^3$.

The increasing sequence
$a=(n^3)_{A\leq n<B}$ satisfies
$a_{B-1}-a_A\leq B^3-A^3$.
It follows that $a$ hits at most
$(B^3-A^3)/2^\lambda+1$ intervals of the form  %(
$I_k=[-L,0)+k2^\lambda$ (where $k\geq0$). %]

Moreover, $a_{n+1}-a_n\geq 3A^2$ for $A\leq n<B$, and thus the sequence $(a_n)_n$ stays in the same interval $I_k$ for at most $L/(3A^2)+1$ indices $n$.

The number of exceptional intervals times the maximal number of indices lying in such an interval gives 
$\bigl((B^3-A^3)/2^\lambda+1\bigr)\bigl(\bigl(3B^2r+3Br^2+r^3\bigr)/(3A^2)+1\bigr)$
exceptional integers. This proves the first part.

Concerning the second part, we note that the condition
\[\digitsum^M((n+r)^3)-\digitsum^M(n^3)                              %((
=\digitsum^{M\cap[0,\lambda)}((n+r)^3)-\digitsum^{M\cap[0,\lambda)}(n^r) %]]
\]
is satisfied if $n^3$ and $(n+r)^3$ have the same binary digits with indices $\ge\lambda$.
The statement easily follows by contraposition in the case $1\le r\leq B-A$.
For $r=0$ the statement is vacuous.
For $r>B-A$ the right hand side of the statement to be proved is greater than $B-A$, which is a trivial upper bound for the left hand side,
and therefore the statement is true also in this case.
\end{proof}
%}}}

The following lemma is just a variant of summation by parts, and we state the proof for completeness.
\begin{lemma}\label{lem_summation_by_parts}
Let $R$ be a ring, $M\ge0$ an integer, and $a_m,b_m\in R$ for $0\leq m<M$.
Then
\[ %\begin{equation}\label{eqn_summation_by_parts}
\sum_{0\leq m<M}a_mb_m
=
b_0\sum_{0\leq m<M}a_m
+\sum_{1\leq \ell<M}
\bigl(b_\ell-b_{\ell-1}\bigr)
\sum_{\ell\leq m<M}a_m.
\] %\end{equation}
\end{lemma}
\begin{proof}
This is trivial for $M\le1$.
By induction, we have for $M\ge1$
\begin{align*}
\sum_{0\leq m<M+1}a_mb_m
&=
b_0\sum_{0\leq m<M}a_m
+\sum_{1\leq \ell<M}
\bigl(b_\ell-b_{\ell-1}\bigr)
\sum_{\ell\leq m<M}a_m
+
a_Mb_M
\\&=
b_0\sum_{0\leq m<M+1}a_m
+\sum_{1\leq \ell<M+1}
\bigl(b_\ell-b_{\ell-1}\bigr)
\sum_{\ell\leq m<M+1}a_m
\\&+
a_Mb_M
-b_0a_M
-\sum_{1\leq \ell<M}
\bigl(b_\ell-b_{\ell-1}\bigr)
a_M
-(b_M-b_{M-1})a_M.
\end{align*}
The last line is a telescoping sum and equals zero.
\end{proof}

At several occasions, we apply Dirichlet's approximation theorem, where we require the factor to be odd.

%{{{ lem_odd_elimination
\begin{lemma}[Odd elimination]\label{lem_odd_elimination}
For nonnegative integers $\ell$, $\kappa$, $\mu$, $\omega$ such that
%$\mu\ge0$ and
$\ell\geq \kappa$,
we define the property $\mathcal P$ by the equivalence
\[
\begin{aligned}
\mathcal P(\omega,\ell,\kappa,\mu)
\quad\Longleftrightarrow\quad 
&\forall\omega_0\in\{0,\ldots,2^\mu-1\}
\;\exists M\in\{1,3,5,\ldots,2^{5\kappa+7}-1\}:
%\ \textsf{such that}
\\&
\bigl(M\bigl(2^\mu\omega+\omega_0\bigr)\bigr)^{[\ell-\kappa,\ell)}=0.
\end{aligned}
\]
Assume that $\ell,\kappa\ge1$ are integers and %$\ell\ge4\kappa+3$.
$\mu\eqdef\ell-4\kappa-4\ge0$.
Then 
\begin{equation}\label{eqn_odd_elimination}
\begin{aligned}
\hspace{8em}&\hspace{-8em}
\#\Bigl\{
\omega\in\{0,\ldots,2^{4\kappa+4}-1\}:
\mathcal P(\omega,\ell,\kappa,\mu)
\Bigr\}
\geq 2^{3\kappa+4}\bigl(2^\kappa-1\bigr).
\end{aligned}
\end{equation}
\end{lemma}
%}}}
%{{{ proof of lem_odd_elimination
\begin{proof}
%Note that $0$ is an element of the set in question.
%We therefore
Assume that $I\subseteq\{1,\ldots,2^{4\kappa+4}-1\}$ is an interval of length $2^{\kappa}$.
%Let $\omega_1\in\{0,\ldots,2^{3\kappa}-1\}$ be given.
We first show that all but at most one $\omega\in I$ %\{0,\ldots,2^\kappa\}$
have the property that
\begin{equation}\label{eqn_Dirichlet}
\begin{aligned}
&\forall\omega_0\in\{0,\ldots,2^\mu-1\}
\;\exists M\in\{1,2,\ldots,2^{\kappa+2}-1\}:
\\&\bigl(M\bigl(2^\mu\omega+\omega_0\bigr)\bigr)^{[\ell-\kappa-1,\ell)}=0
\ \textsf{and}\ 
\bigl(M\bigl(2^\mu\omega+\omega_0\bigr)\bigr)^{[\mu,\ell-\kappa-1)}\ne0.
\end{aligned}
\end{equation}
Let $S\in\{0,\ldots,2^{4\kappa+4}-2^\kappa\}$ (the left endpoint of the interval $I$).
Choose  $M\in\{1,\ldots,2^{\kappa+2}-1\}$, by Dirichlet's approximation theorem, in such a way that
  $\bigl(M2^\mu S\bigr)^{[\ell-\kappa-2,\ell)}=0$.
Assume that we also have
$\bigl(M2^\mu S\bigr)^{[\mu,\ell-\kappa-1)}=0$.
By size restrictions %(note that $0<S+1\leq S+2^\kappa-1<2^{4\kappa+4}$ and
%$\mu+2\kappa+2\leq\mu+3\kappa+2\le\ell-\kappa-2$)
%, moreover $\mu+2\kappa+2<\ell-\kappa-2$)
it follows that~\eqref{eqn_Dirichlet} holds for $S+1\leq \omega\leq S+2^{\kappa}-1$, and all
$\omega_0\in\{0,\ldots,2^\mu-1\}$.
Each $\omega$ for which~\eqref{eqn_Dirichlet} fails is therefore succeeded by $2^\kappa-1$ indices $\omega\bmod 2^{4\kappa+4}$ for which~\eqref{eqn_Dirichlet} holds, which proves the statement.

We need to find an odd factor. Suppose that~\eqref{eqn_Dirichlet} holds,
where $0<\omega<2^{4\kappa+4}$, and let $0\le\omega_0<2^\mu$.
Let $j\le\ell-\kappa-2$ be the index of the highest $1$ in the binary expansion of $\bigl(M\bigl(2^\mu\omega+\omega_0\bigr)\bigr)^{[\mu,\ell)}$,
and set
$K\eqdef 2^{\ell-\kappa-2-j}M$.
It follows that
\[\bigl(2K\bigl(2^\mu\omega+\omega_0\bigr)\bigr)^{[\ell-\kappa-1,\ell)}=1\]
and $K<2^{4\kappa+5}$.
From this we see that multiples of $2K$ vary the digits in $[\ell-\kappa,\ell)$ step-wise, each step occurring once or twice.
%Choose the multiple $\tilde M$ so that the digits attain $\bigl(2^\mu\omega+\omega_0\bigr)^{[\ell-\kappa,\ell)}$.
It follows that there exists $\tilde M<2^{\kappa+1}$ such that
\[\bigl(\bigl(2\tilde MK+1\bigr)\bigl(2^\mu\omega+\omega_0\bigr)\bigr)^{[\ell-\kappa,\ell)}=0.\]
Clearly, $2\tilde MK+1<2^{5\kappa+7}$.
We have shown that in each interval $I$ as above we can find at most one exceptional $\omega$.
In $\{0,\ldots,2^{4\kappa+4}-1\}$ there is therefore not enough space for more than $2^{3\kappa+4}$ exceptions.
The proof is complete.
\end{proof}
%}}}

%Note that $c-b$ --- the width of the new window --- will be much smaller than $c-a$, the width of the original window.
%We can therefore require $H$ to be much smaller than the length of summation, and also much smaller than $2^c$.

%%{{{ lem_fourfold
%\begin{lemma}\label{lem_fourfold}
%Assume that $(\alpha_0,\alpha_1,\alpha_2,\alpha_3)$
%are uniformly distributed. \LS{Formulierung lieber einfach als Vierfachsumme \"uber diese Variablen, mit Diskrepanz als Fehlerterm!}
%%run through the admissible set arising from the variation of $\nOL$, $\sO$, $\sL$, these exponential sums are small (for each $h$!).
%The discrepancy of the $4$-tuples
%\[n\cdot\bigl(\alpha_0,\alpha_1,\alpha_2,\alpha_3\bigr)\]
%is small on average.
%\end{lemma}
%%}}} lem_fourfold
%\begin{proof}
%Motivated by the Erd\H{o}s--Tur\'an inequality,
%let us consider the exponential sum
%\begin{equation}\label{eqn_ET}
%\sum_{0\leq n<N}
%\e\Biggl(
%n\sum_{0\leq q<Q}h_q\alpha_q
%\Biggr)
%\leq
%\min\Biggl(N,
%\biggl\lVert\sum_{0\leq q<Q}h_q\alpha_q\biggr\rVert^{-1}
%\Biggr).
%\end{equation}
%Let us assume that $h_0\neq0$.
%Considering the discrepancy of $n\alpha_0$ 
%
%
%\end{proof}

%)]}

%}}} sec_lemmas

%\input{cubes_linearization.tex}
%{{{
\section{Linearizing the cubic problem}\label{sec_linearize}
\subsection{Statement of the result}
The aim of Section~\ref{sec_linearize} is to prove the \emph{key result} of this paper, which is Proposition~\ref{prp_linearize} below.
Roughly speaking, in this section  we will transform the problem of estimating $\Szero$ into a linear one, while introducing a geometric sum as additional factor.
The precise statement of the main result of this section involves some definitions.
Let us begin with the abbreviation
\begin{equation}\label{eqn_J_def}
J\eqdef\{0,\ldots,2^\nu-1\}.
\end{equation}
For 
\begin{equation}\label{eqn_S9_requirements}
\begin{aligned}
&u,\nu,\rho,\tau,\zeta,\nLO,\nOL,\nOO,\sO,\sL,m,r\in\mathbb Z
\quad\mbox{such that}\\
&u\geq\nu\geq\rho\geq\tau\geq\zeta\geq0,\\
%$0\leq\zeta\leq\tau\leq\rho\leq\nu\leq u$,
&\nLO\in\JLO\eqdef\{0,\ldots,2^{\rho-\tau}-1\},\\
&\nOL\in\JOL\eqdef\{0,\ldots,2^{\tau-\zeta}-1\},\\
&\nOO\in\JOO\eqdef\{0,\ldots,2^\zeta-1\},\\
&\sO,\sL,m,r\ge0,
\end{aligned}
\end{equation}
we define
\begin{equation}\label{eqn_S9_def}
\begin{aligned}
\hspace{1cm}&\hspace{-1cm}
\Seight(u,\nu,\rho,\tau,\zeta,\nLO,\nOL,\nOO,\sO,\sL,m,r)
\\&\eqdef
\frac1{2^{\nu-\rho}}
\sum_{0\leq\nLL<2^{\nu-\rho}}
\prod_{\substack{\varepsilon\in\{0,1\}\\\ts\in\{\sO,\sL\} }}
\e\Bigl(
  \tfrac12\smallspace\digitsum^{[0,u)}\Bigl(\bigl(\nLL 2^\rho
\\&\hspace{2em}+\nLO2^\tau
+\nOL2^\zeta+\ts\smallspace m2^\tau
+\nOO+\varepsilon r
\bigr)^3\Bigr)
\Bigr).
\end{aligned}
\end{equation}

Moreover,
for integers $\lambda,H\ge0$, and $t\in\mathbb R$, we set
\begin{equation}\label{eqn_varphi_def}
\varphi_H(t)\eqdef
\sum_{0\leq h<H}
\e(ht),
%=\frac{1-\e(Ht)}{1-\e(t)}.
\end{equation}

\begin{equation}\label{eqn_K_def}
\begin{aligned}
K(\lambda,\tau,\zeta,\sO,\sL,\nLO,\nOL,m)&\eqdef
\nLO
\frac{6\smallspace\nOL(\sO-\sL)m}{2^{\lambda-2\tau-\zeta}}
\\&
\hspace{2em}+
\frac{3\smallspace \nOL^2(\sO-\sL)m}{2^{\lambda-\tau-2\zeta}}
+\frac{3\smallspace \nOL\bigl(\sO^2-\sL^2\bigr)m^2}{2^{\lambda-2\tau-\zeta}}.
\end{aligned}
\end{equation}
%For convenience, we will use the term \emph{Dirichlet kernel} in order to refer to the function $\varphi_H$, although the proper ($1$-periodic) Dirichlet kernel
%$D_{H-1}$ satisfies
%\[D_{H-1}(x)=2\realpart\varphi_H(x)-1.\]
%In particular, upper bounds for the magnitude and derivative of $\varphi_H$
%We will only be interested in bounding the $L^1$ mean and the magnitude of the derivative of $\varphi_H$, and in these respects, $\varphi_H$ and $D_{H-1}$ behave similarly.

Throughout this section we will assume that~\eqref{eqn_S9_requirements} is satisfied, and that
\begin{align}
B,H,S\ge1,\lambda\ge u\quad\mbox{are integers},\label{eqn_additional_requirements0}\\
\multiple:\JOL\rightarrow\{1,\ldots,B\}.\label{eqn_additional_requirements01}
\end{align}

We gather the error terms arising in the proof of our key proposition (which is Proposition~\ref{prp_linearize} below), and group them into five terms.

\begin{itemize}
\item
The first error term comes from trigonometric approximation of an indicator function of an interval.
\begin{equation}\label{eqn_Ezero_def}
\Ezero\bigl(\nu,u,\lambda,H\bigr)
\eqdef
\frac{2^{\lambda-u}}{H}
\sum_{\substack{h\in2^{\lambda-u}\mathbb Z\\\lvert h\rvert<H}}
\left\lvert
\frac1{2^\nu}
\sum_{0\leq n<2^\nu}
\e\left(\frac{hn^3}{2^\lambda}\right)
\right\rvert
\end{equation}

\item
The second error comes from our first application of van der Corput's inequality (in the form of Lemma~\ref{lem_vdC_generalized}).
It arises when we replace the condition $n\in(I-m\sO)\cap(I-m\sL)$ by $n\in I$:
% (here we use the notation of Lemma~\ref{lem_vdC_generalized}).
\[ %\begin{equation}\label{eqn_E1_def}
\Eone(\nu,\tau,B,H,S)\eqdef
\frac{SBH}{2^{\nu-\tau}}.
\] %\end{equation}

\item
The third error is introduced by the secondary term $S_E$ in an application of summation by parts.

\begin{equation}\label{eqn_E2_def}
\Etwo\eqdef
SH^22^{\nu-\rho}\frac1{2^\tau}\sum_{0\leq\nO<2^\tau}
\biggl\lVert
\frac{\tmultiple(\nO)\nO}{2^{\lambda-\tau-\rho}}
\biggr\rVert
\end{equation}

\item
The fourth error accounts for the replacement of $K$ by $K'$ in the argument of $\varphi_H$ (see~\eqref{eqn_K_def},~\eqref{eqn_Kprime_def}), that is, the lowest $\zeta$ digits are removed, and therefore $\nOO$ disappears.
\begin{equation}\label{eqn_E3_def}
\Ethree\eqdef H^22^\zeta\bigl(SB2^{\rho+\tau-\lambda}+S^2B^22^{2\tau-\lambda}\bigr).
\end{equation}
Note that the factor $H^2$ stems from the maximal slope of the geometric sum $\varphi_H$.
\item
The fifth error term arises from our second application of van der Corput's inequality, on the sum over $\nOO\in\JOO$:
omitting the condition $0\leq \nOO+r<2^\zeta$ from a sum,
as well as the summand $r=0$,
and discarding the digits with indices in $[\lambda,\infty)$,
we obtain the contribution
\begin{equation}\label{eqn_E4m_def}
\Efourm\eqdef\frac{R}{2^\zeta}+\frac1R+\frac{R2^{2\nu}}{2^\lambda}.
\end{equation}
Taking an average over a geometric sum into account, we arrive at the contribution
\begin{equation}\label{eqn_E4_def}
\Efour\eqdef\biggl(\frac{R}{2^\zeta}+\frac1R+\frac{R2^{2\nu}}{2^\lambda}\biggr)\log\nu.
\end{equation}
\end{itemize}

Our \emph{key result} is the following.
It essentially reduces the statement of Theorem~\ref{thm_main} to a linear problem.

%{{{ prp_linearize
\begin{proposition}\label{prp_linearize}
Assume that $B,H,S\ge1$ are integers, and let~\eqref{eqn_S9_requirements},~\eqref{eqn_additional_requirements0} be satisfied.
Assume that 
\begin{equation}\label{eqn_additional_requirements1}
2\tau-\rho\ge0,\quad\lambda\geq u,\quad 2^{\lambda-u}\mid H.
\end{equation}
Then
\[ %\begin{equation}\label{eqn_first_milestone}
\begin{aligned}
\bigl\lvert \Szero(\nu,\xi)\bigr\rvert^4&\leq 
\Snine+\LandauO\bigl(\Ezero+\Eone+\Etwo+\Ethree+\Efour\bigr),
\end{aligned}
\] %\end{equation}
where
%{{{ eqn_Sa_def
\begin{equation}\label{eqn_Sa_def}
\begin{aligned}
\Snine&\eqdef
\frac1{S^2}
\sum_{\sO,\sL}
\frac1{2^{\rho-\zeta}}
\sum_{\substack{\nOL\in\JOL\\\nLO\in\JLO}}
\Biggl(
  \frac1R\sum_{\lvert r\rvert<R}\frac1{2^\zeta}\sum_{\nOO}
\bigl\lvert \Seight\bigl(
  %u,\nu,\rho,\tau,\zeta,
\nLO,\nOL,\nOO,\sO,\sL,\multiple(\nOL),r\bigr)\bigr\rvert\Biggr)^{1/2}
\\&\hspace{3cm}\times\bigl\lvert \varphi_H\bigl(K\bigl(\lambda,\tau,\zeta,\sO,\sL,\nLO,\nOL,\multiple(\nOL)\bigr)\bigr)\bigr\rvert.
\end{aligned}
\end{equation}
%}}} eqn_Sa_def
\end{proposition}
%}}} prp_linearize

The parameters appearing as arguments of $\Seight$ will be chosen later, in Sections~\ref{sec_uniform},~\ref{sec_finishing}
%at the end of the proof of Theorem~\ref{thm_main}, in Section~\ref{sec_finishing}.
The values to be chosen includes the important collection $\multiple$ of odd integers found with the help of a variant of Dirichlet's approximation theorem (Lemma~\ref{lem_odd_elimination}).

Proposition~\ref{prp_linearize} is the first milestone in our proof of Theorem~\ref{thm_main}, and the remainder of section~\ref{sec_linearize} will be concerned with its proof.
The implied constants appearing in estimates in this proof are absolute (that is, uniform in the variables appearing in~\eqref{eqn_S9_requirements},~\eqref{eqn_additional_requirements0},~\eqref{eqn_additional_requirements1} and satisfying these restrictions).
As a first step in the treatment of $\Szero$, we will partition the summation index set $J$ according to the binary digits with indices in the \emph{critical interval} $[u,\lambda)$.
\subsection{Detecting digits in the critical interval}\label{sec_detection}
%We want to eliminate the digits with indices in the interval $[u,\lambda)$ of length
We set $\kappaO\eqdef \lambda-u$, which is the length of our critical interval.
The sets
\[P(u,\lambda,j)
\eqdef
\bigl\{
n\in\mathbb N:
\bigl\lfloor n^3/2^\lambda\bigr\rfloor\equiv j\bmod 2^\kappaO
\bigr\}
\]
of natural numbers having given binary digits in this interval,
form a partition of $\mathbb N$,
where $j\in\{0,\ldots,2^\kappaO-1\}$.
Running through all digit sequences of length $\kappaO$ by using the variable $j$, we obtain
\begin{equation}\label{eqn_S1_digit_blocks}
\Szero(u,\lambda,\nu,\xi)=
\sum_{0\leq j<2^\kappaO}
(-1)^{\digitsum(j)}
\frac1{2^\nu}
\sum_{n\in J\cap P(u,\lambda,j)}
\e\Bigl(
  \tfrac12\smallspace\digitsum^{\mathbb N\setminus[u,\lambda)}\bigl(n^3\bigr)
+n\xi\Bigr)
\end{equation}
for all integers $\nu\ge0$, $0\le u\le \lambda$, and for all real numbers $\xi$,
where $J$ is defined in~\eqref{eqn_J_def}.
Eliminating the complications introduced by the additional sum over $j$ will constitute one of the main difficulties of the proof of the main theorem.

In order to detect the property $n\in P(u,\lambda,j)$ in an ``analytical'' way, we use trigonometric approximation.
%We will detect the property $n\in P(j)$ by trigonometric approximation.
More specifically, we apply Corollary~\ref{cor_interval_detection_shift}.
Note that
\[
n\in P(u,\lambda,j)\quad\mbox{if and only if}\quad
\left\{\frac{n^3}{2^\lambda}\right\}
\in \left[\frac{j}{2^\kappaO},\frac{j+1}{2^\kappaO}\right).
\]
Let us define
\[ %\begin{equation}\label{eqn_I_def}
%\Iold\eqdef\mathbb N\setminus[u,\lambda),\quad
I\eqdef\mathbb N\setminus[u,\infty),\quad
\alpha(j)\eqdef \frac j{2^\kappaO},\quad
\beta(j)\eqdef \frac{j+1}{2^\kappaO}.
\] %\end{equation}
Then~\eqref{eqn_S1_digit_blocks} can be written as
\begin{equation}\label{eqn_S2_transformation}
\begin{aligned}
\hspace{4em}&\hspace{-4em}
\Szero(\nu,u,\lambda,\xi)=
\sum_{\substack{0\leq j<2^\kappaO}}
(-1)^{\digitsum(j)}
\frac1{2^\nu}
\sum_{n\in J}
\e\Bigl(\tfrac12\digitsum^{\Iold}\bigl(n^3\bigr)+n\xi\Bigr)
\indicator_{[\alpha(j),\beta(j))+\mathbb Z}\left(\frac{n^3}{2^\lambda}\right)
\end{aligned}
\end{equation}
for all integers $\nu\ge0$, $0\le u\le \lambda$, and for all real numbers $\xi$.

Next, we replace the indicator functions by trigonometric polynomials, using Corollary~\ref{cor_interval_detection_shift}.
For real numbers $\alpha,\beta,x$, we note the trivial identity
\begin{equation}\label{eqn_trivial_triple}
\begin{aligned}
\hspace{4em}&\hspace{-4em}
\indicator_{[\alpha,\beta)+\mathbb Z}(x)
&=
A_{\alpha,\beta,H}(x)+
\Bigl(\indicator_{[\alpha,\beta)+\mathbb Z}(x)
-A_{\alpha,\beta,H}(x)\Bigr).
\end{aligned}
\end{equation}

The \emph{main term} $A_{\alpha,\beta,H}(x)$
will lead to the sum $\Stwo$ below, while the term in parentheses is treated as an error term.
Its contribution to the sum $\Szero(\nu,\lambda)$ is
\begin{equation}\label{eqn_trivial_triple_applied}
\begin{aligned}
\Szero^{(0)}(\nu,u,\lambda,H)&\eqdef
\sum_{\substack{0\leq j<2^\kappaO}}
(-1)^{\digitsum(j)}
\frac1{2^\nu}
\sum_{n\in J}
\e\Bigl(\tfrac12\digitsum^{\Iold}\bigl(n^3\bigr)\bigr)
\Bigr)
\\&\hspace{2em}\times
\left(\indicator_{[\alpha(j),\beta(j))+\mathbb Z}\left(\frac{n^3}{2^\lambda}\right)
-A_{\alpha(j),\beta(j),H}\left(\frac{n^3}{2^\lambda}\right)\right).
\end{aligned}
\end{equation}
Using Corollary~\ref{cor_interval_detection_shift}, we may estimate the absolute value of the difference, written in the second line of~\eqref{eqn_trivial_triple_applied}, by a nonnegative trigonometric polynomial.
This will allow us to discard this absolute value again. Using also the triangle inequality and the bound 
\[\bigl\lvert b_h\bigl(2^{-\kappaO},H\bigr)\bigr\rvert\leq 1/H,\]
we get
\[ %\begin{equation}
\begin{aligned}
\bigl\lvert \Szero^{(0)}\bigl(\nu,u,\lambda,H\bigr)\bigr\rvert
&\leq
\sum_{0\leq j<2^\kappaO}
\sum_{\substack{h\in\mathbb Z\\\lvert h\rvert<H}}
\e\left(-\frac{hj}{2^{\kappaO}}\right)
b_h\bigl(2^{-\kappaO},H\bigr)
\frac1{2^\nu}
\sum_{n\in J}
\e\left(\frac{hn^3}{2^\lambda}\right)
\\&=
\sum_{\substack{h\in\mathbb Z\\\lvert h\rvert<H}}
b_h\bigl(2^{-\kappaO},H\bigr)
\hspace{0.5em}
\sum_{0\leq j<2^\kappaO}
\e\left(-\frac{hj}{2^{\kappaO}}\right)
\hspace{0.5em}
\frac1{2^\nu}
\sum_{n\in J}
\e\left(\frac{hn^3}{2^\lambda}\right)
\\&=
2^\kappaO
\sum_{\substack{h\in2^\kappaO\mathbb Z\\\lvert h\rvert<H}}
b_h\bigl(2^{-\kappaO},H\bigr)
\frac1{2^\nu}
\sum_{n\in J}
\e\left(\frac{hn^3}{2^\lambda}\right)
\\&\leq
E_0\bigl(\nu,u,\lambda,H\bigr),
\end{aligned}
\] %\end{equation}
where
$\Ezero$ is defined in~\eqref{eqn_Ezero_def}.

The main term coming from the application of~\eqref{eqn_trivial_triple} to the sum~\eqref{eqn_S2_transformation} is
\[ %\begin{equation}\label{eqn_S2_def}
\begin{aligned}
\hspace{6em}&\hspace{-6em}
\Sone(\nu,u,\lambda,H,\xi)\eqdef 
\sum_{\substack{\hO\in\mathbb Z\\0\le\lvert \hO\rvert<H}}
a_{\hO}\bigl(2^{-(\lambda-u)},H\bigr)\,
\sum_{0\leq j<2^{\lambda-u}}
\e\bigl(\tfrac12 \digitsum(j)-j\hO2^{-(\lambda-u)}\bigr)
\\&\times
\frac1{2^\nu}
\sum_{0\leq n<2^\nu}
\e\biggl(\frac12\digitsum^{\mathbb Z\setminus[u,\lambda)}\bigl(n^3\bigr)
+\frac{\hO n^3}{2^\lambda}+n\xi\biggr),
\end{aligned}
\] %\end{equation}
and we have the estimate
\begin{equation}\label{eqn_S2_estimate}
\Szero(\nu,\xi)
=
\Sone(\nu,u,\lambda,H,\xi)
+\LandauO\bigl(E_0(\nu,u,\lambda,H)\bigr),
\end{equation}
with an absolute implied constant.

We consider the main term $\Sone(\nu,u,\lambda,H,\xi)$.
We apply Cauchy--Schwarz on the sum over $\hO$,
and use the estimate
\[\bigl\lvert a_h(2^{-(\lambda-u),H})\bigr\rvert\leq \frac1{2^{\lambda-u}}:\]
\[ %\begin{equation}\label{eqn_S2_CS}
\begin{aligned}
\bigl\lvert \Sone(\nu,u,\lambda,H,\xi)\bigr\rvert^2
&\leq
\frac{\Stwo(\nu,u,\lambda,H,\xi)}
{2^{2(\lambda-u)}}
\sum_{\substack{\hO\in\mathbb Z\\\lvert\hO\rvert<H}}
\Biggl\lvert
\sum_{0\leq j<2^{\lambda-u}}
\e\bigl(
\tfrac12 \digitsum(j)-\hO j2^{-(\lambda-u)}
\bigr)
\Biggr\rvert
^2,
\end{aligned}
\] %\end{equation}
where
\[ %\begin{equation}\label{eqn_S3_def}
\Stwo(\nu,u,\lambda,H,\xi)
\eqdef
\sum_{\substack{\hO\in\mathbb Z\\\lvert\hO\rvert<H}}
\Biggl\lvert
\frac1{2^\nu}
\sum_{0\leq n<2^\nu}
\e\biggl(\tfrac12\digitsum^{\mathbb Z\setminus[u,\lambda)}\bigl(n^3\bigr)
+\frac{\hO n^3}{2^\lambda}+n\xi\biggr)
\Biggr\rvert^2.
\] %\end{equation}
Adding the summand $h=0$ and applying Lemma~\ref{lem_full}, we obtain
\begin{equation*} %\label{eqn_}
\begin{aligned}
\hspace{4em}&\hspace{-4em}
\sum_{\substack{\hO\in\mathbb Z\\\lvert\hO\rvert<H}}
\Biggl\lvert
\sum_{0\leq j<2^{\lambda-u}}
\e\bigl(
\tfrac12 \digitsum(j)-\hO j2^{-(\lambda-u)}
\bigr)
\Biggr\rvert
^2
\\&\leq
\frac{2H}{2^{\lambda-u}}
\sum_{0\leq \hO<2^{\lambda-u}}
\Biggl\lvert
\sum_{0\leq j<2^{\lambda-u}}
\e\bigl(
\tfrac12 \digitsum(j)-\hO j2^{-(\lambda-u)}
\bigr)
\Biggr\rvert
^2
=
2H2^{\lambda-u}.
\end{aligned}
\end{equation*}
It follows that
\begin{equation}\label{eqn_S2_S3}
\begin{aligned}
\bigl\lvert \Sone(\nu,u,\lambda,H,\xi)\bigr\rvert^2
&\leq
\frac{2H}{2^{\lambda-u}}
\Stwo(\nu,u,\lambda,H,\xi).
\end{aligned}
\end{equation}

Note that we have the trivial bound
$E_0(u,\nu,\lambda,H)\ll1$, and also $\Sone\ll1$ by~\eqref{eqn_S2_estimate}.
From~\eqref{eqn_S2_estimate} and~\eqref{eqn_S2_S3} it follows therefore that
\begin{equation}\label{eqn_S1_S3}
\bigl\lvert \Szero(\nu,\xi)\bigr\rvert^2
\ll
%\left\lvert
\frac H{2^{\kappaO}}
\Stwo(\nu,u,\lambda,H,\xi)
+E_0(\nu,u,\lambda,H).
\end{equation}

The error term $E_0$ will be estimated by Lemma~\ref{lem_Ezero_estimate} near the end of the proof of Theorem~\ref{thm_main}.
Being the error arising from trigonometric approximation of an interval of length $2^{-(\lambda-u)}$, via Corollary~\ref{cor_interval_detection_shift}, we will need $2^{\lambda-u}=o(H)$ in order to get a nontrivial estimate for it.
The additional nontrivial factor $H/2^{\lambda-u}$ will be accounted for,
as we will show $\Stwo=o(A^{-1})$ (see Section~\ref{sec_finishing}).

\subsection{Introducing correlations}
Let us introduce an (integer) split point $\tau\in[\lambda/3,\nu]$.
We decompose the summation variable $n$ (having binary length $\le\nu$) at the index $\tau$: set
%{{{ eqn_first_split
\begin{equation}\label{eqn_first_split}
\begin{aligned}
n&=2^\tau \nL+\nO, \quad\mbox{where}
\quad
\left\{
\begin{array}{l}
\nL\in \JL\eqdef\bigl\{0,\ldots,2^{\nu-\tau}-1\bigr\},\\[1mm]
\nO\in \JO\eqdef\bigl\{0,\ldots,2^\tau-1\bigr\}.
\end{array}
\right.
\end{aligned}
\end{equation}
%}}} eqn_first_split
Using this decomposition, we obtain
\[ %\begin{equation}\label{eqn_S3_split}
\begin{aligned}
\Stwo(\nu,u,\lambda,H,\xi)
&=
\sum_{\substack{\hO\in\mathbb Z\\\lvert\hO\rvert<H}}
\Biggl\lvert
\frac1{2^\tau}
\sum_{\nO\in\JO}
\e\bigl(\nO\xi\bigr)
\frac1{2^{\nu-\tau}}
\sum_{\nL\in\JL}
\e\Biggl(\tfrac12\digitsum^{\mathbb N\setminus[u,\lambda)}\Bigl(\bigl(\nL2^\tau+\nO\bigr)^3\Bigr)
\\&\hspace{3em}+\hO\frac{\nL^32^{3\tau}+3\nL^22^{2\tau}\nO+3\nL2^\tau\nO^2+\nO^3}{2^\lambda}+\nL2^\tau\xi\Biggr)
\Biggr\rvert^2\\
&\leq
\sum_{\substack{\hO\in\mathbb Z\\\lvert\hO\rvert<H}}
\frac1{2^\tau}
\sum_{\nO\in\JO}
\Biggl\lvert
\frac1{2^{\nu-\tau}}
\sum_{\nL\in\JL}
\e\Biggl(\tfrac12\digitsum^{\mathbb N\setminus[u,\lambda)}\Bigl(\bigl(\nL2^\tau+\nO\bigr)^3\Bigr)
\\&
\hspace{3em}
+\hO\frac{3\nL^22^{2\tau}\nO+3\nL2^\tau\nO^2}{2^\lambda}+\nL 2^\tau\xi\Biggr)
\Biggr\rvert^2
\end{aligned}
\] %\end{equation}
for all integers $\nu,\tau,u,\lambda,H$ satisfying
$0\leq \lambda/3\leq\tau\leq \nu$, $0\leq u\leq \lambda$, $H\geq 1$, $2^{\lambda-u}\mid H$, and all $\xi\in\mathbb R$.

\subsubsection{The first application of van der Corput's inequality}
The next important step consists in an application of Lemma~\ref{lem_vdC_generalized} (the generalized inequality of van der Corput) on the sum over $\nL$.
We consider a function $\multiple:\JOL\rightarrow\{0,\ldots,B\}$
to be defined later, where $B$ is also chosen later.
For brevity, we define
\[\tmultiple:\JO\rightarrow\{0,\ldots,B\},\quad\nO\mapsto\multiple\bigl(\bigl\lfloor \nO/2^\zeta\bigr\rfloor\bigr).\]
Clearly, this function depends only on the digits of $\nO$ having indices in $[\zeta,\tau)$. Later, the role of $\lfloor \nO/2^\zeta\rfloor$ will be taken by $\nOL$.

We choose $\mathcal S=\{sm2^\tau:0\leq s<S\}$ in Lemma~\ref{lem_vdC_generalized}, which implies
\begin{equation}\label{eqn_S4_estimate}
\Stwo(\nu,u,\lambda,H,\xi)
\leq
\frac1{S^2}
\sum_{\substack{0\leq \sO<S\\0\leq \sL<S}}
\frac1{2^\tau}
\sum_{\nO\in \JO}
\bigl\lvert
\Sthree(\nu,u,\lambda,H,\tmultiple(\nO),\sO,\sL,\nO)
\bigr\rvert,
\end{equation}
where
\[ %\begin{equation}\label{eqn_S5_estimate}
\begin{aligned}
\hspace{1em}&\hspace{-1em}
\Sthree(\nu,u,\lambda,H,m,\sO,\sL,\nO)
\\&\eqdef
\frac1{2^{\nu-\tau}}
\sum_{\substack{%\nL\in\JL\\
\nL+\sO m\in\JL\\\nL+\sL m\in\JL}}
\e\Bigl(
 \tfrac12\smallspace\digitsum^{\mathbb N\setminus[u,\lambda)}\bigl(\bigl(\nL2^\tau+\nO+\sO\,m2^\tau\bigr)^3\bigr)
%\\&\hspace{2em}
-\tfrac12\smallspace\digitsum^{\mathbb N\setminus[u,\lambda)}\bigl(\bigl(\nL2^\tau+\nO+\sL\,m2^\tau\bigr)^3\bigr)
\Bigr)
\\&\times
\sum_{0\leq \hO<H}
\e\Biggl(
\hO\frac{
3(\nL+\sO\,m)^22^{2\tau}\nO+3(\nL+\sO\,m) 2^\tau\nO^2
}{2^\lambda}\biggr)
\\&\hspace{2em}
-\hO\frac{
3(\nL+\sL m)^22^{2\tau}\nO+3(\nL+\sL m) 2^\tau\nO^2
}{2^\lambda}
\Biggr)\\
&=
\frac1{2^{\nu-\tau}}
\sum_{\substack{\nL+\sO\,m\in\JL\\\nL+\sL m\in\JL}}
\e\Bigl(
 \tfrac12\smallspace\digitsum^{\mathbb N\setminus[u,\lambda)}\bigl(\bigl(\nL2^\tau+\nO+\sO\,m2^\tau\bigr)^3\bigr)
%\\&\hspace{2em}
-\tfrac12\smallspace\digitsum^{\mathbb N\setminus[u,\lambda)}\bigl(\bigl(\nL2^\tau+\nO+\sL m2^\tau\bigr)^3\bigr)
\Bigr)
\\&\times
\sum_{0\leq \hO<H}
\e\Biggl(
\hO\frac{
6\nL\nO(\sO-\sL)m2^{2\tau}
}{2^\lambda}\Biggr)
%\\&\hspace{2em}\times
\e\Biggl(
\hO
\frac{
3\nO^2(\sO-\sL)m2^\tau+3\nO(\sO^2-\sL^2)m^22^{2\tau}
}{2^\lambda}\Biggr).
\end{aligned}
\] %\end{equation}
Note that the van der Corput inequality reduced the degree of the polynomial in $\nL$ in the exponential by $1$.
A linear term remains, a fact that will prove to be of great importance in our proof.
We rewrite the unwieldy expressions above, using the abbreviations
%{{{
\begin{equation}\label{eqn_xf_def}
\begin{aligned}
x(\lambda,\tau,m,\sO,\sL,\nO)
&\eqdef
\frac{6\smallspace\nO(\sO-\sL)m}{2^{\lambda-2\tau}},\\
f(\lambda,\tau,m,\sO,\sL,\nO)
&\eqdef
\frac{3\smallspace \nO^2(\sO-\sL)m
+3\smallspace \nO\bigl(\sO^2-\sL^2\bigr)m^22^{\tau}
}{2^{\lambda-\tau}}.
\end{aligned}
\end{equation}
%}}}
Inserting $x$ and $f$,
we obtain
\begin{equation}\label{eqn_S3_S5}
\begin{aligned}
\hspace{15em}&\hspace{-15em}
\Stwo(\nu,u,\lambda,H,\xi)
\leq
\frac1{S^2}
\sum_{\substack{0\leq \sO<S\\0\leq \sL<S}}
\frac1{2^\tau}
\sum_{\nO\in \JO}
\Biggl\lvert
\sum_{0\leq \hO<H}
\e\bigl(\hO f(\lambda,\tau,\tmultiple(\nO),\sO,\sL,\nO)\bigr)
\\[-1em]&\times
\Sfour(\nu,\tau,u,\lambda,\tmultiple(\nO),\sO,\sL,\hO,\nO)
\Biggr\rvert
\end{aligned}
\end{equation}
where
\[ %\begin{equation}\label{eqn_S5_def}
\begin{aligned}
\hspace{2em}&\hspace{-2em}
\Sfour(\nu,\tau,u,\lambda,m,\sO,\sL,\hO,\nO)
\\&\eqdef
\frac1{2^{\nu-\tau}}
\sum_{\substack{\nL+\sO m\in\JL\\\nL+\sL m\in\JL}}
\e\Biggl(
\frac12\sum_{\ts\in\{\sO,\sL\}}
\smallspace\digitsum^{\Iold}\Bigl(\bigl(\nL 2^\tau+\nO+\ts\smallspace m2^\tau\bigr)^3\Bigr)
\\[-2em]&\hspace{7cm}
+\nL\smallspace\hO\smallspace x(\lambda,\tau,m,\sO,\sL,\nO)
\Biggr).
\end{aligned}
\] %\end{equation}
Note that this expression is already uniform in $\xi\in\mathbb R$.
We have to handle the conditions $\nL+\sO m, \nL+\sL m\in \JL$.
Note that the quantities $\sO m$ and $\sL m$ are bounded by $SB$.
Replacing the two conditions by $\nL\in \JL$ therefore introduces an error
$SB/2^{\nu-\tau}$. This error is multiplied by $H$ in order to account for the sum over $\hO$, which is not preceded by a factor $1/H$. This yields $\Eone$.

Let us split up the summation variable $\nL$ into two parts,
using another split point $\rho\in[\tau,\nu]$ to be chosen later.
We set
\begin{equation}\label{eqn_nL_split}
\begin{aligned}
\nL&=2^{\rho-\tau}\nLL+\nLO, \quad\mbox{where}
\quad
\left\{
\begin{array}{l}
\nLL\in \JLL\eqdef\bigl\{0,\ldots,2^{\nu-\rho}-1\bigr\},\\[1mm]
\nLO\in \JLO\eqdef\bigl\{0,\ldots,2^{\rho-\tau}-1\bigr\}.
\end{array}
\right.
\end{aligned}
\end{equation}
Note that we have
\[ %\begin{equation}\label{eqn_parameters_stack}
\frac{\lambda}3\leq\tau\leq\rho\leq\nu.
\] %\end{equation}
%While the dependence on $\xi$ is gone,
The quantity $\nL\smallspace\hO\smallspace x(\nO)$ introduces another ``twist'', in the variable $\nL$.
In order to remove it,
we will later choose the values $\tmultiple(\nO)$ suitably.
In this process, the term
\[\lVert 2^{\rho-\tau}\hO\smallspace x(\hO,\nO)\rVert,\]
which appears after replacing the sum over $\nL$ by a sum over $(\nLL,\nLO)$,
becomes very small, and only slightly perturbs the exponential sum $\Sfour$.
In this manner, applying summation by parts, we remove $x$ and uncouple $\Sfour$ from the sum over $\hO$. This first \emph{uncoupling} procedure
allows us to obtain an unrestricted sum of $\e(\hO f(\sO,\sL,\nO))$ over $\hO$,
in the expression~\eqref{eqn_S3_S5}, giving rise to a geometric sum $\varphi_H$.

In Section~\ref{sec_uncoupling}, see Proposition~\ref{prp_glycerol}, we will later continue this uncoupling process, using basically the extremal case $(p,q)=(1,\infty)$ of H\"older's inequality.
This will yield a sum in the variable $\nLO$, over the function $\varphi_H$, not obstructed by any sum-of-digits expression.

Using the decomposition~\eqref{eqn_nL_split}, we obtain
\begin{equation}\label{eqn_S5_S6}
\Sfour%(\nu,\tau,\lambda,\mu,\multiple,\sO,\sL,\hO,\nO)
=
\frac1{2^{\rho-\tau}}
\sum_{\nLO\in \JLO}
\e\bigl(\nLO\smallspace\hO\smallspace x(\hO,\nO)\bigr)
\Sfive(\nu,\tau,u,\lambda,m\sO,\sL,\hO,\nO,\nLO)
+\LandauO(\Eone),
\end{equation}
where
\[ %\begin{equation}\label{eqn_S6_def}
\begin{aligned}
\hspace{3em}&\hspace{-3em}
\Sfive(\nu,\tau,u,\lambda,m,\sO,\sL,\hO,\nO,\nLO)
\\&\eqdef
\frac1{2^{\nu-\rho}}
\sum_{\nLL\in \JLLm}
\e\Biggl(
\frac12\sum_{\ts\in\{\sO,\sL\}}
\smallspace\digitsum^{\Iold}\bigl(\bigl(2^\rho\nLL+2^\tau \nLO+\nO+\ts\smallspace m2^\tau)^3\bigr)
\Biggr)
\\&\hspace{6cm}\times
\e\bigl(\nLL2^{\rho-\tau}\hO\smallspace x(\lambda,\tau,m,\sO,\sL,\nO)\bigr).
\end{aligned}
\] %\end{equation}

We use summation by parts on the sum over $\nLL$, more precisely,
the version given in Lemma~\ref{lem_summation_by_parts}.
Rewriting $\Sfive$ according to this lemma, we obtain
\begin{equation}\label{eqn_S6_split}
\Sfive(\nu,\tau,u,\lambda,m,\sO,\sL,\hO,\nLO,\nO)
=\Sseven+S_{\mathrm{E}},
\end{equation}
where
\begin{align*}
\Sseven(\sO,\sL,m,\nLO,\nO)&\eqdef
\frac1{2^{\nu-\rho}}
\sum_{\nLL\in\JLLm}
\prod_{\ts\in\{\sO,\sL\}}
\e\Bigl(
\tfrac12\smallspace\digitsum^{\Iold}\Bigl(\bigl(\nLL 2^\rho+\nLO2^\tau+\nO+
\ts\smallspace m2^\tau\bigr)^3\Bigr)
\Bigr),\\
S_{\mathrm{E}}(\sO,\sL,m,\hO,\nLO,\nO)&\eqdef
\bigl(1-\e\bigl(-2^{\rho-\tau}\hO\smallspace x\bigr)\bigr)
\sum_{\substack{\mLL\in \JLLm\\\mLL\ge1}}
\e(\mLL 2^{\rho-\tau}\hO x)
\Ssix^{(\mLL)}(\sO,\sL,m,\nLO,\nO),
\end{align*}
and
\[\Ssix^{(\mLL)}(\sO,\sL,m,\nLO,\nO)\eqdef
\frac1{2^{\nu-\rho}}
\sum_{\substack{\nLL\in \JLLm\\\nLL\ge\mLL}}
\prod_{\ts\in\{\sO,\sL\}}
\e\Bigl(
\tfrac12\smallspace\digitsum^{\Iold}\Bigl(\bigl(\nLL 2^\rho+\nLO2^\tau+\nO+
\ts\smallspace m2^\tau\bigr)^3\Bigr)
\Bigr).
\]

The elementary inequality
$\lVert na\rVert\leq n\lVert a\rVert$ implies
\begin{equation}\label{eqn_x_close}
\bigl\lVert 2^{\rho-\tau}\hO\smallspace x(m)\bigr\rVert
=\left\lVert
2^{\rho-\tau}6(\sO-\sL)\hO\frac{m\nO}{2^{\lambda-2\tau}}
\right\rVert
\leq 6SH\biggl\lVert \frac{mn_0}{2^{\lambda-\tau-\rho}}\biggr\rVert,
\end{equation}
which yields the bound 
\begin{equation}\label{eqn_SE_estimate}
\frac H{2^\tau}\sum_{0\leq\nO<2^\tau}
\bigl\lVert 2^{\rho-\tau}\hO\smallspace x(\tmultiple(\nO))\bigr\rVert
\ll \Etwo,
\end{equation}
where $\Etwo$ is defined in~\eqref{eqn_E2_def}.
This average over $\nO\in\{0,\ldots,2^\tau-1\}$ appears in~\eqref{eqn_S3_S5}.
Having disposed of the twisting term
$\e\bigl(\nLL 2^{\rho-\tau}\hO\smallspace x\bigr)$, which is present in the definition of $\Sfive$,
we thus have removed the coupling of the variables $\hO$ and $\nLL$.
Consulting~\eqref{eqn_S3_S5},~\eqref{eqn_S5_S6},~\eqref{eqn_S6_split}, and~\eqref{eqn_SE_estimate}, we obtain
\begin{equation}\label{eqn_S3_S8}
\begin{aligned}
\hspace{4em}&\hspace{-4em}
\Stwo(\nu,\xi,\lambda,\mu,H)
\leq
\frac1{S^2}
\sum_{\substack{0\leq \sO<S\\0\leq \sL<S}}
\frac1{2^\tau}
\sum_{\nO\in\JO}
\Biggl \lvert
\sum_{0\leq \hO<H}
\e\bigl(\hO f(\tmultiple(\nO),\sO,\sL,\nO)\bigr)
\\&%\hspace{1em}
\times\frac1{2^{\rho-\tau}}
\sum_{\nLO\in\JLO}
\e(\nLO\smallspace \hO\smallspace x(\hO,\nO))
\bigl(\Sseven(\sO,\sL,\tmultiple(\nO),\nLO,\nO)+S_{\mathrm E}\bigr)
\Biggr \rvert
+\LandauO(\Eone)
\\
&\leq
\frac1{S^2}
\sum_{\substack{0\leq \sO<S\\0\leq \sL<S}}
\frac1{2^\tau}
\sum_{\nO\in\JO}
\frac1{2^{\rho-\tau}}
\sum_{\nLO\in\JLO}
\Biggl \lvert
\Sseven(\sO,\sL,\tmultiple(\nO),\nLO,\nO)
\\&%\hspace{1em}
\times
\sum_{0\leq \hO<H}
\e\bigl(\hO\bigl(f(\sO,\sL,\tmultiple(\nO),\nO)+\nLO\smallspace x(\tmultiple(\nO)),\sO,\sL,\nO)\bigr)\bigr)
\Biggr \rvert
\\&+\LandauO(\Eone+\Etwo)
\\&\leq
\frac1{S^2}
\sum_{\substack{0\leq \sO<S\\0\leq \sL<S}}
\frac1{2^\rho}
\sum_{(\nLO,\nO)\in\JLO\times\JO}
\bigl \lvert
\Sseven(\sO,\sL,\tmultiple(\nO),\nLO,\nO)
\bigr \rvert
\\[-.5em]&
\hspace{11em}
\times
\bigl \lvert\varphi_H\bigl(f(\tmultiple(\nO),\sO,\sL,\nO)+\nLO\smallspace x(\tmultiple(\nO),\sO,\sL,\nO)\bigr)
\bigr \rvert
\\[0.5em]&%\hspace{1em}
+\LandauO(\Eone+\Etwo),
\end{aligned}
\end{equation}
where $\varphi_H$ is defined by \eqref{eqn_varphi_def}.
We see that Dirichlet approximation, present in the choice of $\multiple(\nOL)$ (see Section~\ref{sec_finishing}) enables us to isolate a geometric sum $\varphi_H$.

Up to now, the digits with indices in $[\lambda,\infty)$ are still present.
In order to remove them, we are going to split up $\nO$ at the index $\zeta$ (to be chosen later:
\begin{equation}\label{eqn_nO_decomposition}
\begin{aligned}
\nO&=\nOL2^\zeta+\nOO, \quad\mbox{where}\quad
\left\{
\begin{array}{l}
\nOL\in \JOL\eqdef\bigl\{0,\ldots,2^{\tau-\zeta}-1\bigr\},\\[1mm]
\nOO\in \JOO\eqdef\bigl\{0,\ldots,2^\zeta-1\bigr\}.
\end{array}
\right.
\end{aligned}
\end{equation}
In the following we will not use $\tmultiple$ any longer, as $\multiple(\nOL)=\tmultiple(\nO)$.
We first note that the argument of $\varphi_H$ depends only in a weak way on the lowest digits of $\nO$.
Defining
\begin{equation}\label{eqn_Kprime_def}
K'(m,\sO,\sL,\nLO,\nO)\eqdef f(m,\sO,\sL,\nO)+\nLO\smallspace x(m,\sO,\sL,\nO),
\end{equation}
we see from the definition~\eqref{eqn_xf_def} that
\[ %\begin{equation}\label{eqn_f_weakly_varying}
\begin{aligned}
\hspace{15em}&\hspace{-15em}
K'\bigl(m,\sO,\sL,\nLO,\nOL2^\zeta+a\bigr)-K'\bigl(m,\sO,\sL,\nLO,\nOL2^\zeta+b\bigr)
\\&\ll
2^\zeta\bigl(SB2^{\rho+\tau-\lambda}+S^2B^22^{2\tau-\lambda}\bigr)
\end{aligned}
\] %\end{equation}
for $a,b\in\JOO$, $\lvert\sO\rvert,\lvert\sL\rvert<S$, and $(\nLO,\nOL)\in\JLO\times\JOL$,
with some absolute implied constant.
Inserting the definition~\eqref{eqn_E3_def} of $\Ethree$,
this implies
\begin{equation}\label{eqn_varphiH_estimate}
\varphi_H\bigl(K'(m,\sO,\sL,\nLO,\nOL2^\zeta+a)\bigr)-
\varphi_H\bigl(K'(m,\sO,\sL,\nLO,\nOL2^\zeta+b)\bigr)
\ll \Ethree.
\end{equation}
Replacing $K'$ by $K$ (defined in~\eqref{eqn_K_def}),
and using~\eqref{eqn_nO_decomposition} and~\eqref{eqn_varphiH_estimate},
we obtain
\begin{equation}\label{eqn_J00_abkapselung}
\begin{aligned}
\hspace{2em}&\hspace{-2em}
\frac1{2^\rho}
\sum_{(\nLO,\nOL,\nOO)\in \JLO\times\JOL\times\JOO}
\bigl\lvert \Sseven(\sO,\sL,\nLO,\nOL2^\zeta+\nOO)\bigr\rvert\bigl\lvert\varphi_H\bigl(K(\sO,\sL,\nLO,\nOL)\bigr)\bigr\rvert
\\&\ll
\frac1{2^{\rho-\zeta}}
\sum_{(\nLO,\nOL)\in\JLO\times\JOL}
\bigl\lvert\varphi_H\bigl(K(\sO,\sL,\nLO,\nOL)\bigr)\bigr\rvert
\\&\hspace{3cm}\times
\frac1{2^\zeta}
\sum_{\nOO\in\JOO}
\bigl\lvert \Sseven\bigl(\sO,\sL,\nLO,\nOL2^\zeta+\nOO\bigr)\bigr\rvert
+
\LandauO\bigl(\Ethree\bigr).
\end{aligned}
\end{equation}
\subsubsection{The second application of van der Corput's inequality}
As the expression $\Sseven$ is now summed freely in the variable $\nOO$, we may apply van der Corput's inequality.
There exist complex numbers
\[\varepsilon_{\nOO}=\varepsilon_{\nOO}(\nLO,\nOL),
\quad
\lvert\varepsilon_{\nOO}\rvert=1,
\]
 such that
\[\bigl\lvert
\Sseven(\sO,\sL,\nLO,\nOL2^\zeta+\nOO)
\bigr\rvert
=\varepsilon_{\nOO}
\Sseven(\sO,\sL,\nLO,\nOL2^\zeta+\nOO).\]
Clearly, $\varepsilon_{\nOO}$ does not depend on the index $\nLL$.
In particular, we obtain for all $(\nLO,\nOL)\in\JLO\times\JOL$
\begin{equation}\label{eqn_S8_S9}
\begin{aligned}
\hspace{2em}&\hspace{-2em}
\frac1{2^\zeta}
\sum_{\nOO\in \JOO}
\bigl\lvert
\Sseven\bigl(\sO,\sL,m,\nLO,\nOL2^\zeta+\nOO\bigr)
\bigr\rvert
=\frac1{2^{\nu-\rho}}
\sum_{\nLL\in\JLLm}
\frac1{2^\zeta}
\sum_{\nOO\in \JOO}
\varepsilon_{\nOO}
\\&\times
\prod_{\ts\in\{\sO,\sL\}}
\e\Bigl(
\tfrac12\smallspace\digitsum^{\Iold}\Bigl(\bigl(\nLL 2^\rho+\nLO2^\tau
+\nOL2^\zeta+\ts\smallspace m2^\tau+\nOO
\bigr)^3\Bigr)
\Bigr)
\\&\leq
\Biggl(
\frac1{2^{\nu-\rho}}
\sum_{\nLL\in\JLLm}
\Biggl\lvert
\frac1{2^\zeta}
\sum_{\nOO\in \JOO}
\varepsilon_{\nOO}
\\&\times
\prod_{\ts\in\{\sO,\sL\}}
\e\Bigl(
\tfrac12\smallspace\digitsum^{\Iold}\Bigl(\bigl(\nLL2^\rho+\nLO2^\tau
+\nOL2^\zeta
+\ts\smallspace m2^\tau
+\nOO
\bigr)^3\Bigr)
\Bigr)
\Biggr\rvert^2
\Biggr)^{1/2}
\\&\leq
\Biggl(
\frac1{2^{\nu-\rho}}
\sum_{\nLL\in\JLLm}
\frac{1+(R-1)/2^\zeta}{R}
\sum_{\lvert r\rvert<R}
\left(1-\frac{\lvert r\vert}R\right)
\frac1{2^\zeta}
\sum_{\nOO\in \JOO\cap(\JOO-r)}
\varepsilon_{\nOO+r}
\overline{\varepsilon_{\nOO}}
\\&
\times
\prod_{\substack{\varepsilon\in\{0,1\}\\\ts\in\{\sO,\sL\} }}
\e\Bigl(
\tfrac12\smallspace\digitsum^{\Iold}\Bigl(\bigl(\nLL2^\rho+\nLO2^\tau
+\nOL2^\zeta+\ts\smallspace m2^\tau+\nOO+\varepsilon r
\bigr)^3\Bigr)
\Bigr)
\Biggr)^{1/2}.
\end{aligned}
\end{equation}

Only at this point, we apply the \emph{carry lemma}, Lemma~\ref{lem_carry},
in order to
\[\mbox{replace\quad$\Iold=[0,u)\cup[\lambda,\infty)$\quad by\quad$\Inew=[0,u)$.}\]
This elimination of the upper digits, with indices in $[\lambda,\infty)$, was the purpose of having split up the summation over $\nO$ (at $\zeta$).
%Since $\nLL<2^{\nu-\rho}-1$, see~\eqref{eqn_JLLm_def},
The condition
\begin{equation}\label{eqn_additional2}
SB<2^{\rho-\tau}
\end{equation}
implies that
\[\bigl\lvert
\nLL2^\rho+\nLO2^\tau+\nOL2^\zeta+\ts\smallspace\multiple(\nOL)2^\tau+\nOO\bigr\rvert
\leq 2^\nu+2^\rho\leq 2^{2\nu}.
\]
%Let us remove the digits with indices in $[\lambda,\infty)$.
Let us exclude the summand $\nLL=0$. %, and also $\nOO=0$.
Lemma~\ref{lem_carry}, applied to the case
\[
\begin{aligned}
A&\eqdef
\nLL2^\rho+\nLO2^\tau+\nOL2^\zeta+\ts\smallspace\multiple(\nOL)2^\tau,\quad
B\eqdef A+2^\zeta-r,
\end{aligned}
\]
and $M=\mathbb N\setminus[u,\lambda)$,
allows us to replace $\mathbb N\setminus[u,\lambda)$ in~\eqref{eqn_S8_S9} by
$[0,u)$.
Note that $A\geq2^\rho$ by our assumption that $\nLL\neq0$, therefore an error
\begin{equation*}%\label{eqn_E8_def}
\frac1{2^\zeta}
r\frac{(2^\rho+2^\zeta)^2}{(2^\rho)^2}\bigl(2^{2\nu-\lambda+\zeta}+1\bigr)
\ll
\frac{R\smallspace2^{2\nu}}{2^\lambda}+\frac R{2^\zeta}
\end{equation*}
arises from applying the carry lemma.
Removing also the term $r=0$, an error $1/R$ is introduced, leading to $\Efourm$.

We obtain \begin{equation}\label{eqn_S8_S9_2}
\begin{aligned}
\hspace{3em}&\hspace{-3em}
\Biggl(
\frac1{2^\zeta}
\sum_{\nOO\in \JOO}
\bigl\lvert
\Sseven\bigl(\sO,\sL,m,\nLO,\nOL2^\zeta+\nOO\bigr)
\bigr\rvert
\Biggr)^2
\\&\ll
\frac1{R}
\sum_{1\leq\lvert r\rvert<R}
\frac1{2^\zeta}
\sum_{\nOO\in \JOO}
\bigl\lvert
\Seight\bigl(u,\nu,\rho,\tau,\zeta,\sO,\sL,m,\nLO,\nOL,\nOO,r\bigr)
\bigr\rvert
+\Efourm,
\end{aligned}
\end{equation}
where
$\Seight$ and $\Efourm$ are defined in~\eqref{eqn_S9_def} and~\eqref{eqn_E4m_def} respectively. 
%{{{ finishing the proof of prp_linearize
\subsubsection{Finishing the proof or Proposition~\ref{prp_linearize}}\label{sec_prp_linearize_finishing}
It remains to prove that the contribution of $\Efourm$ is still not too large when the sum over $\varphi_H$ in~\eqref{eqn_J00_abkapselung} is taken into account.
After all, the product $H\Efourm$ yields the trivial estimate for $\Szero$ and therefore cannot be used. Instead, we use the fact that $\Efourm$ does not depend on the outer summations over $\nLO$ and $\nOL$ in~\eqref{eqn_J00_abkapselung}, and we can take an average over $\varphi_H$ in these variables.
In Section~\ref{sec_losing_only_a_logarithm}, see~\eqref{eqn_SX_estimate}, we will in fact prove such an estimate, where we also have to take the error term~$\Esix$ (see~\eqref{eqn_Esix_def}) into account.
Summarizing, we only lose a logarithmic factor, yielding $\Efour$.

Collecting the error terms and combining
~\eqref{eqn_S1_S3},
~\eqref{eqn_S3_S8},
~\eqref{eqn_varphiH_estimate},
~\eqref{eqn_J00_abkapselung}, and
~\eqref{eqn_S8_S9_2},
the proof of Proposition~\ref{prp_linearize} is complete.
\qed

%{{{ 
\subsection{Restricting some parameters in Proposition~\ref{prp_linearize}}
Let us now assume that $\lambda,u,\rho,\tau,\zeta$
are chosen in such a way that
\begin{equation}\label{eqn_parameters_stack}
\zeta\leq\frac{\lambda}3\leq\tau\leq\frac u2\leq\rho\leq\nu\leq\frac{\lambda}2.
\end{equation}
In our sum $\Seight$ (see~\eqref{eqn_S9_def}), the term $\nLL$ is accompanied by $2^\rho$.
As a consequence, we obtain the very important fact that terms involving a factor $\nLL^2$ (or $\nLL^3$) do not play a role any longer.
The elimination of the digits with indices in $[u,\lambda)$
was performed precisely for this purpose!

Thus, since $\nLL$ appears only linearly, the problem is essentially reduced to a linear one.

We introduce the abbreviations
\begin{equation}\label{eqn_AalphaB_def}
\begin{aligned}
A(\varepsilon,\ts)&\eqdef
  \nLO2^\tau+\nOL2^\zeta+\nOO+\varepsilon r+\ts\smallspace\multiple(\nOL)2^\tau,\\
\Qtilde(\varepsilon,\ts)&\eqdef 3A(\varepsilon,\ts)^2,\quad
\beta(\varepsilon,\ts)\eqdef\left\lfloor
\frac{
A(\varepsilon,\ts)^3}{2^\rho}\right\rfloor,\quad
c(\varepsilon,\ts)\eqdef \digitsum^{[0,\rho)}\left(A(\varepsilon,\ts)^3\right).
\end{aligned}
\end{equation}
Expanding the cube in the definition~\eqref{eqn_S9_def} of $\Seight$,
we obtain
\[ %\begin{equation}\label{eqn_S9_linear}
\begin{aligned}
\hspace{4em}&\hspace{-4em}
\Seight(u,\nu,\rho,\tau,\zeta,\nLO,\nOL,\nOO,\sO,\sL,m,r)
\\
&=
\frac1{2^{\nu-\rho}}
\sum_{\nLL\in \JLL} 
\prod_{\substack{\varepsilon\in\{0,1\}\\
\ts\in\{\sO,\sL\}
}}
\e\Bigl(\tfrac12\smallspace\digitsum^{I}\Bigl(
\nLL\Qtilde(\varepsilon,\ts)2^\rho+A(\varepsilon,\ts)^3
\Bigr)
\Bigr)
\\&=
\frac1{2^{\nu-\rho}}
\sum_{\nLL\in \JLL}
\prod_{\substack{\varepsilon\in\{0,1\}\\
\ts\in\{\sO,\sL\}
}}
\e\Bigl(\tfrac12\smallspace\digitsum^{[0,u-\rho)}(\nLL\Qtilde(\varepsilon,\ts)+\beta(\varepsilon,\ts))
+\tfrac12c(\varepsilon,\ts)
\Bigr).
\end{aligned}
\] %\end{equation}
The values of $A$, $\Qtilde$, and $c$ do not depend on $\nLL$.
Note that $\ts$ always appears with a factor $2^\tau$, and so each summand in the argument of $\digitsum^{[0,u-\rho)}$ that contains $\ts$ is divisible by $2^\tau$.
Therefore the contributions of the lowest $\tau$ digits arising from $\ts=\sO$ and $\ts=\sL$, respectively, are identical and cancel.
This justifies the replacement of $[0,u-\rho)$ by $[\tau,u-\rho)$:
\begin{equation}\label{eqn_S9_linear2}
\begin{aligned}
\hspace{5em}&\hspace{-5em}
\bigl\lvert
\Seight(u,\nu,\rho,\tau,\zeta,\nLO,\nOL,\nOO,\sO,\sL,m,r)
\bigr\rvert
\\&=
\frac 1{2^{\nu-\rho}}
\Biggl\lvert
\sum_{\nLL\in \JLL}
\prod_{\substack{\varepsilon\in\{0,1\}\\
\ts\in\{\sO,\sL\}} }
\e\Bigl(\tfrac12\smallspace\digitsum^{[\tau,u-\rho)}\bigl(\nLL\Qtilde(\varepsilon,\ts)+\beta(\varepsilon,\ts)\bigr)
\Bigr)
\Biggr\rvert.
\end{aligned}
\end{equation}

In~\eqref{eqn_S9_linear2}, we replace
$\Qtilde(\varepsilon,\ts)$ according to~\eqref{eqn_AalphaB_def}, expanding the square.
Note that
$\zeta$ will be chosen sufficiently small ($\zeta\leq\lambda/6$ is sufficient, which is guaranteed by our choices~\eqref{eqn_parameter_choices} later on),
so that
\begin{equation}\label{eqn_tau_large}
u-\rho\le2\tau-\zeta\le2\tau\le\tau+\rho\le2\rho.
\end{equation}
It follows that, integer multiples of
$2^{2\tau-\zeta}$, $2^{2\tau}$, $2^{\tau+\rho}$, and $2^{2\rho}$
do not contribute to $\digitsum^{[\tau,u-\rho)}$, and may be discarded.
Therefore
\[\Qtilde(\varepsilon,\ts)
\equiv
\Qprime(\varepsilon,\ts)
\bmod 2^{u-\rho},
\]
where
\begin{equation}\label{eqn_Qprime_def}
%\boxed{
\begin{aligned}
\Qprime(\varepsilon,\ts)
\eqdef&\,
6\smallspace\nA(\varepsilon)2^\tau\nLO
+\,6\smallspace\ts\smallspace\multiple(\nOL)\smallspace\nA(\varepsilon)2^\tau
+\,3\smallspace\nA(\varepsilon)^2
%6\,(\nOL2^\zeta+\nOO+\varepsilon r)2^\tau\nLO
%\\+&\,6\,\ts\multiple(\nOL)(\nOL2^\zeta+\nOO+\varepsilon r)2^\tau
%\\+&\,3\,(\nOL2^\zeta+\nOO+\varepsilon r)^2.
\end{aligned}
%}
\end{equation}
and
\[\nA(\varepsilon)\eqdef \nOL2^\zeta+\nOO+\varepsilon r.\]
As $\digitsum^{[\tau,u-\rho)}$ is periodic with period $2^{u-\rho}$, we may replace $\Qtilde$ by $\Qprime$:
\begin{equation}\label{eqn_S9_linear3}
\begin{aligned}
\hspace{5em}&\hspace{-5em}
\bigl\lvert
\Seight(u,\nu,\rho,\tau,\zeta,\nLO,\nOL,\nOO,\sO,\sL,m,r)
\bigr\rvert
\\&=
\frac 1{2^{\nu-\rho}}
\Biggl\lvert
\sum_{\nLL\in \JLL}
\prod_{\substack{\varepsilon\in\{0,1\}\\
\ts\in\{\sO,\sL\}} }
\e\Bigl(\tfrac12\smallspace\digitsum^{[\tau,u-\rho)}(\nLL\smallspace \Qprime(\varepsilon,\ts)+\beta(\varepsilon,\ts))
\Bigr)
\Biggr\rvert.
\end{aligned}
\end{equation}

For convenience, we summarize the above arguments in the following corollary.
Note that this is just Proposition~\ref{prp_linearize} under the additional hypothesis~\eqref{eqn_parameters_stack}, which transforms $\Seight$ into the simpler form~\eqref{eqn_S9_linear3}.
\begin{corollary}\label{cor_linearize}
Assume that the requirements~\eqref{eqn_S9_requirements},~\eqref{eqn_additional_requirements0},~\eqref{eqn_additional_requirements1}, and~\eqref{eqn_parameters_stack} are satisfied. Then

\begin{equation}\label{eqn_linearize}
\begin{aligned}
\bigl\lvert \Szero(\nu,u,\lambda,\xi)\bigr\rvert^4&\leq 
\sum_{\sO,\sL}
\frac1{2^{\rho-\zeta}}
\sum_{\substack{\nOL\in\JOL\\\nLO\in\JLO}}
\bigl\lvert \varphi_H\bigl(K(\lambda,\tau,\zeta,\nLO,\nOL,\sO,\sL,\multiple(\nOL)\bigr)\bigr\rvert
\\&\hspace{1.5cm}\times
\Biggl(
\sum_{\lvert r\rvert<R}\sum_{\nOO}
\,
\bigl\lvert \Seight(u,\nu,\rho,\tau,\zeta,\nLO,\nOL,\nOO,\sO,\sL,\multiple(\nOL),r)\bigr\rvert
\Biggr)^{1/2}
\\&+\LandauO\bigl(\Ezero+\Eone+\Etwo+\Ethree+\Efour\bigr),
\end{aligned}
\end{equation}
where $\Seight$ is given by~\eqref{eqn_S9_linear3}, and $\Ezero$--$\Efour$ are defined immediately before Proposition~\ref{prp_linearize}.
\end{corollary}
%}}}
%}}}

%\input{cubes_uncoupling.tex}
%{{{ sec_uncoupling
\section{Decoupling the geometric sum and the sum of digits}\label{sec_uncoupling}
We are going to decompose the sum over $\nLO\in\JLO$ into sums along short arithmetic progressions. The following elementary lemma will be of use.

%{{{ lem_decomposition
\begin{lemma}\label{lem_decomposition}
Assume that $I\subseteq\mathbb Z$ is a finite interval,
and that $T,V\ge1$ are integers such that $TV\leq\lvert I\rvert$.
There exists a partition $\partition$ of $I$ such that
\begin{equation}\label{eqn_V_introduction}
\begin{aligned}
\hspace{3em}&\hspace{-3em}
\mbox{for all $P\in\partition$:}\\
&\mbox{(1) $P=\bigl(T\mathbb Z+a\bigr)\cap [x,y]$ for some $(a,x,y)\in\mathbb Z\times I\times I$,}\\
&\mbox{(2) $V/2\leq \lvert P\rvert\leq V$.}
\end{aligned}
\end{equation}

\end{lemma}
%}}} lem_decomposition

%{{{ proof of lem_decomposition
\begin{proof}
For each $a\in\mathbb Z$ we have %it follows that
\[W\eqdef\bigl\lvert\bigl(a+T\mathbb Z\bigr)\cap I\bigr\rvert
\geq \bigl\lfloor \lvert I\rvert/T\bigr\rfloor\geq V.\]
Each of these sets can be decomposed into shorter arithmetic progressions $P$ satisfying $V/2\leq\lvert P\rvert\leq V$.
This is clear for $2\mid V$, otherwise we take $M$ shorter progressions of length $(V+1)/2$ so that $W-V\leq M(V+1)/2<W-(V-1)/2$. The remaining interval $P'$ has length $V\geq \lvert P'\rvert\geq (V-1)/2+1\geq V/2$.
\end{proof}
%}}}

For each $\nOL\in\JOL$, we are going to choose a factor $T(\nOL)\in\mathbb N$ later, see Section~\ref{sec_Efivesix_estimate}, with the help of Lemma~\ref{lem_odd_elimination}.
The values $T(\nOL)$ are supposed to be odd, and eliminate digits of $\nOL$ directly below $\lambda-2\tau-\zeta$.
This is needed in order to guarantee that $K$ does not change much when $\nOL$ is varied, see~\eqref{eqn_K_def}.
More precisely, for given $\nOL\in\JOL$ we want to find an integer $T$ such that
\begin{equation}\label{eqn_varepsilon_small_enough}
T\in2\mathbb N+1,\quad
T<2^{5\etaO+7},
\quad\mbox{and}\quad
\biggl\lVert T\frac{\nOL}{2^{\lambda-2\tau-\zeta}}\biggr\rVert<2^{-\etaO},
\end{equation}
where the parameter $\etaO$ is chosen later.
For such an odd factor $T$ to exist for most $\nOL$, it is sufficient (by Lemma~\ref{lem_odd_elimination}) to have enough digits below $\lambda-2\tau-\zeta$ for an odd. In other words, we only have to require
\begin{equation}\label{eqn_eta_condition}
\lambda-2\tau-\zeta\ge 4\etaO+4.
\end{equation}
We will decompose $\JLO$ into arithmetic progressions with difference $T$, according to Lemma~\ref{lem_decomposition}, and the sensible choice of $T$ implies that $\varphi_H(K(m))$ is well behaved along these progressions.
%For now, we assume that
%\begin{equation}\label{eqn_T_function}
%T:\JOL\rightarrow\mathbb N
%\end{equation}
%is any function.
%For this, we are going to use Lemma~\ref{lem_odd_elimination},
The set of \emph{good} indices $\nOL\in\JOL$ --- such that a factor $T$ as in~\eqref{eqn_varepsilon_small_enough} exists --- will be denoted by $\GetaO(\etaO)$.
%{{{ eqn_GetaO_def
\begin{equation}\label{eqn_GetaO_def}
\GetaO(\etaO)\eqdef
\bigl\{
\nOL\in\JOL:\mbox{\eqref{eqn_varepsilon_small_enough} holds for some }T
\bigr\}.
\end{equation}
%}}} eqn_GetaO_def

Let us assume that $T,V\ge1$ are integers such that
\begin{equation}\label{eqn_TV_restriction}
TV\leq \lvert\JLO\rvert,
\end{equation}
and that $\partition$ is a partition of $\JLO$ such that~\eqref{eqn_V_introduction} is satisfied for $I=\JLO$.

The quantity $\etaO$ will be chosen large enough so that $2^{-\etaO}$ eliminates factors of size $\LandauO(H^2SB)$:
consulting the definition of $K$ (equation~\eqref{eqn_K_def}), and~\eqref{eqn_linearize}, we see that the argument of $\varphi_H$
will vary by steps $\ll 2^{-\etaO}SB$ along arithmetic subsequences (in $\nLO$) with difference $T$.

Moreover, it can be seen easily, by expanding the definition of $\varphi_H$, that
%%{{{ lem_varphi_derivative
%\begin{lemma}\label{lem_varphi_derivative}
%Define
%\begin{equation}
%The derivative of $\varphi_H$ satisfies
\begin{equation}\label{eqn_varphi_derivative}
\varphi_H'(x)\ll H^2.
\end{equation}
%\end{lemma}
%%}}} lem_varphi_derivative
Equation~\eqref{eqn_varphi_derivative} together with the 
mean value theorem, using the definition~\eqref{eqn_K_def} and the choice of $T$~\eqref{eqn_varepsilon_small_enough}
imply that the function $\varphi_H\circ K$ along $P\in\partition$
varies by an amount
$\ll 2^{-\etaO}H^2SBV$.
This implies, for all $\nLO\in P\in\partition$,
\begin{equation*}
\sup_{m\in P}
\bigl\lvert
\varphi_H(K(m))
\bigr\rvert
\leq
\bigl\lvert\varphi_H(K(\nLO))\bigr\rvert
+\LandauO(\Efive),
\end{equation*}
where
%{{{ eqn_Efive_def
\begin{equation}\label{eqn_Efive_def}
\Efive\eqdef
H^2SBV %2^{-\etaO},
\sup_{\nOL\in\GetaO}
\biggl\lVert T(\nOL)\frac{\nOL}{2^{\lambda-2\tau-\zeta}}\biggr\rVert,
\end{equation}
%}}} eqn_Efive_def
and therefore
%{{{ eqn_declass
\begin{equation}\label{eqn_declass}
\begin{aligned}
\frac1{\lvert \partition\rvert}
\sum_{P\in\partition}
\sup_{m\in P}
\bigl\lvert
\varphi_H(K(m))
\bigr\rvert
&=
\frac1{\lvert \partition\rvert}
\sum_{P\in\partition}
\frac1{\lvert P\rvert}
\sum_{\nLO\in P}
\bigl\lvert\varphi_H(K(\nLO))\bigr\rvert
+\LandauO\bigl(\Efive\bigr)
\\&\ll
\frac1{\lvert \JLO\rvert}
\sum_{\nLO\in\JLO}
\bigl\lvert\varphi_H(K(\nLO))\bigr\rvert
+\Efive.
\end{aligned}
\end{equation}
%}}}
The estimate~\eqref{eqn_declass} embodies the simple observation that $\varphi_H(K(\cdot))$ is ``almost constant'' along arithmetic progression with difference $T$, if $T$ is chosen properly.
We see that $2^{-\etaO}$ will also have to eliminate $V$ (and the \emph{small} values $B$, $H$, $S$) in order to yield a useful result.
By Lemma~\ref{lem_odd_elimination}, we have to allow $T\gg 2^{5\etaO}\gg V^5$:
the length $V$ of the arithmetic progression $P$ will be small,
compared to its common difference $T$.
We will, somewhat arbitrarily, choose $\etaO\sim 4\lambda/139$ later on,
see also Remark~\ref{rem_after_Etwelve}.

Along these arithmetic progressions we need to establish a certain kind of uniform distribution on the remaining intervals (see~\eqref{eqn_second_factor_structure}).
We apply Lemma~\ref{lem_glycerol} on the sum over $\nLO$ in~\eqref{eqn_linearize}, while $\nOL,\sO$, and $\sL$ are parameters.
Two error terms will arise.
\begin{itemize}
\item The first error term is just $\Efive$ defined in~\eqref{eqn_Efive_def}.
%, coming from~\eqref{eqn_declass}.
It captures the variation of $\varphi_H$ along $P\in\partition$.
\item The second error term
%{{{ eqn_Esix_def
\begin{equation}\label{eqn_Esix_def}
\Esix\eqdef
\frac{H\lvert\JOL\setminus\GetaO\rvert}{\lvert\JOL\rvert} %2^{\tau-\zeta}}
+\frac HS
\end{equation}
%}}} eqn_Esix_def
accounts for the trivial summands $\nOL\not\in\GetaO$ and $\sO=\sL$.
%\LS{Note that we may even multiply by $H$, since we blew up our summation by this factor. This is guaranteed by our choices of $H$ and $S$ later on.}
Under the requirement~\eqref{eqn_eta_condition} that
\[\lambda-2\tau-\zeta\ge 4\etaO+4,\]
there are still enough digits below so that Lemma~\ref{lem_odd_elimination} can be applied in a profitable manner:
%\LS{$\eta$ takes the role of $\kappa$}
there is only one ``bad index'' $\nOL$ among $2^{\etaO}$ successive indices.
This implies
\begin{equation}\label{eqn_Esix_estimate}
\Esix\ll\frac H{2^{\etaO}}+\frac HS.
\end{equation}
\end{itemize}
%{{{ prp_glycerol
\begin{proposition}\label{prp_glycerol}
Assume that $H,S,B\ge1$ are integers,
and let $\mathcal G\subseteq\JOL$ be any set such that $0\not\in \mathcal G$.
For each $\nOL\in\mathcal G$, let $\partition(\nOL)$ a partition such as in Lemma~\ref{lem_decomposition}, corresponding to the parameters
\begin{equation*}
\begin{aligned}
I&=\JLO,\quad\mbox{and}\\
T&=T(\nOL),\quad V=V(\nOL) \mbox{ are chosen later}.
\end{aligned}
\end{equation*}
We have
%{{{ eqn_S1_split
\[ %\begin{equation}\label{eqn_S1_split}
\begin{aligned}
\bigl\lvert \Snine\bigr\rvert
&\ll
\SX\cdot\SY+\Efive+\Esix,
\end{aligned}
\] %\end{equation}
%}}}
where
$\Snine$ was defined in~\eqref{eqn_Sa_def},

\smallskip
%{{{ eqn_first_factor_structure, eqn_second_factor_structure, eqn_SZ_def
\begin{align}
\SX&\eqdef
\sup_{\substack{\nOL\in\mathcal G\\\sO\neq\sL}}
\frac1{\lvert \JLO\rvert}
\sum_{\nLO\in\JLO}
\bigl\lvert\varphi_H(K(\nLO))\bigr\rvert,\label{eqn_first_factor_structure}\\
\SY&\eqdef
\frac1{2^{\tau-\zeta}}
\sum_{\nOL\in\mathcal G}
\frac1{S^2}
\sum_{\substack{\sO,\sL\\\sO\neq\sL}}
\sup_{P\in\partition(\nOL)}
\SZ^{1/2},\label{eqn_second_factor_structure}\\
\SZ&\eqdef
\frac1R\sum_{\lvert r\rvert<R}
\frac1{2^\zeta}\sum_{\nOO\in\JOO}
\frac1{\lvert P\rvert}
\sum_{\nLO\in P}
\bigl\lvert \Seight( %u,\nu,\rho,\tau,\zeta,
\nLO,\nOL,\nOO,\sO,\sL,\multiple(\nOL),r)\bigr\rvert.\label{eqn_SZ_def}
\end{align}
%}}} eqn_first_factor_structure, eqn_second_factor_structure, eqn_SZ_def

\end{proposition}
%}}}
\begin{proof}
We apply Lemma~\ref{lem_glycerol} to the sum over $\nLO$ in~\eqref{eqn_Sa_def},
followed by~\eqref{eqn_declass}.
Together with a third application of H\"older's inequality,
this explains the error term $\Efive$ as well as the provenance of the terms $\SX$ and $\SY$.
Excluding the cases $\sO=\sL$, and also $\nOL\not\in\mathcal G$,
explains the error term $\Esix$.
An application of Cauchy--Schwarz completes the proof.

\end{proof}

In Section~\ref{sec_losing_only_a_logarithm} below %Proposition~\ref{prp_log_in_mean}
we will find an upper bound for the mean $\SX$ over $\varphi_H$.
For this bound, we will exclude some more $\nOL$: we will require $\nOL\in\Geta\eqdef\GetaO\cap\GetaL$, where $\GetaO$ is needed for~\eqref{eqn_Esix_def}, and $\GetaL$ will secure small discrepancy (see~\eqref{eqn_mathcalGetaL_def}, leading to a good bound for $\SX$.
This bound will be \emph{uniform} in $\sO,\sL$, and $\nOL$, under the hypotheses that $\sO\neq\sO$ and $\nOL\in\GetaL$. 
The estimate of the expression $\SY$ will be the subject of Section~\ref{sec_uniform}.

%
%{{{ sec_losing_only_a_logarithm
\subsection{Estimate of $\SX$:\\ The geometric sum contributes only a logarithmic factor}\label{sec_losing_only_a_logarithm}
In this section, we denote
\[ %\begin{equation}\label{eqn_tildeD_def}
\tilde D_U(\alpha)\eqdef
D_U\bigl(\bigl(0,\alpha,2\alpha,\ldots,(U-1)\alpha\bigr)\bigr)
\] %\end{equation}
for brevity.
We are going to study an average discrepancy estimate, where the average is taken over multiples $m\alpha$. %In our case $\alpha$ will be small, but not too small.
This result is Lemma~3.4 in~\cite{MuellnerSpiegelhofer2017}.
%{{{ lem_MU
\begin{lemma}\label{lem_MU}
let $U\ge1$ and $\eta\ge0$ be integers. Then
%{{{ eqn_MU
\[ %\begin{equation}\label{eqn_MU}
\sum_{d<2^\eta}D_U\biggl(\frac{d}{2^\eta}\biggr)
\ll \frac{U+2^\eta}{U}\bigl(\logp U\bigr)^2.
\] %\end{equation}
%}}} eqn_MU
\end{lemma}
%}}}
We combine this average discrepancy estimate with the facts that the total variation of $\varphi_H\bigr\rvert_{[0,1]}$ is $\LandauO(H^2)$, and that the integral of $\lvert \varphi_H(x)\rvert$ is bounded by $\LandauO(\logp H)$.
Applying the Koksma--Hlawka inequality (Lemma~\ref{lem_Koksma_Hlawka}), we obtain
\[ %\begin{equation}\label{eqn_Koksma_Hlawka_applied}
\frac1U
\sum_{0\leq n<U}
\bigl\lvert\varphi_H(An+B)\bigr\rvert
\ll\logp H+ H^2\tilde D_U(A)
\] %\end{equation}
for all $A,B\in\mathbb R$, $U\ge1$, and integers $H\ge1$.
We wish to apply this estimate to $\SX$, defined in~\eqref{eqn_first_factor_structure}.
In this case, $U=\lvert\JLO\rvert$,
and the slope
\[A=\nOL6(\sO-\sL)\multiple(\nOL)/2^{\lambda-2\tau-\zeta}\]
comes from the first term of $K$~\eqref{eqn_K_def} (while $B$ is irrelevant).
Concerning the supremum over $\nOL\in\mathcal G$ in~\eqref{eqn_first_factor_structure},
we will choose
\[ %\begin{equation}\label{eqn_G_choice}
\mathcal G\eqdef \GetaO\cap\GetaL,
\] %\end{equation}
where $\GetaO$ is defined in~\eqref{eqn_GetaO_def},
%{{{
\begin{align}
\GetaL(\etaL)&\eqdef
\biggl\{
\nOL\in\{1,\ldots,\lvert\JOL\rvert-1\}:
\tilde D_U\biggl(\nOL\frac{6(\sO-\sL)\multiple(\nOL)}{2^{\lambda-2\tau-\zeta}}\biggr)
<2^{-\etaL}
\biggr\}\label{eqn_mathcalGetaL_def},
\end{align}
%}}}
and $\eta_1$ is chosen later.
%Now the set $\mathcal G_1$ was chosen precisely in such a way that the discrepancy of $\nOL/2^{\lambda-2\tau-\zeta}$ is small.
%\LS{REF!}
%Multiplying by
%\[6SB,\]
We obtain
\begin{equation}\label{eqn_SX_estimate}
\SX\ll \logp H+H^22^{-\eta_1}.
\end{equation}
%\LS{Note that $\mathcal G_0$ is responsible for the small variation along the Dirichlet kernel, while $\mathcal G_1$ ensures small discrepancy!}
The restriction to the set $\GetaL$ introduces a new error term
\[ %\begin{equation}\label{eqn_Eseven_def_2}
\Eseven\eqdef\frac{\lvert \JOL\setminus\GetaL\rvert}{\lvert \JOL\rvert}.
\] %\end{equation}
In order to bound $\Eseven$, we use the 
hypothesis~\eqref{eqn_parameters_stack} that $3\tau\ge\lambda$,
implying $\lvert\JOL\rvert=K2^{\lambda-2\tau-\zeta}$ for some integer $K$
(in fact, in~\eqref{eqn_parameter_choices} we are going to define $\tau\eqdef\lambda/3$, therefore $K=1$). 
Since $\sO-\sL\in\bigl([-S,S]\setminus\{0\}\bigr)\cap\mathbb Z$ and $\multiple(\nOL)\in\{1,\ldots,B\}$, the family
\[ %\begin{equation}\label{eqn_uniform_family}
f:\JOL\rightarrow[0,1):\nOL\mapsto\nOL\frac{6(\sO-\sL)\multiple(\nOL)}{2^{\lambda-2\tau-\zeta}}\bmod 1
\] %\end{equation}
attains each value at most $CSB\max_{s,b}2^{\nu_2(sb)}\leq C(SB)^2$ times, with an absolute constant $C$.
It is therefore sufficient to apply a mean discrepancy estimate as in Lemma~\ref{lem_MU}.
We obtain
\[\sum_{\nOL\in\JOL}\tilde D_U\bigl(f(\nOL)\bigr)
\ll (SB)^22^{3\tau-\lambda}\lambda^2,\]
where $U=\lvert\JLO\rvert$.
It easily follows that
%{{{ eqn_Eseven_estimate
\begin{equation}\label{eqn_Eseven_estimate}
\Eseven
\ll \frac{(SB)^22^{3\tau-\lambda}}{\lvert\JOL\rvert}\lambda^2
2^{\eta_1}.
\end{equation}
%}}} eqn_Eseven_estimate

The parameter $\eta_1$ can be chosen freely, but due to~\eqref{eqn_SX_estimate} it does not make much sense to choose it larger than $2\log_2 H$.

%Assuming that $\nOL\in\GetaL$, the slopes $A$ in the definition of $K$~\eqref{eqn_K_def}, where $K=A\nLO+B$,
%contribute a small average discrepancy according to~\eqref{eqn_mathcalGetaL_def} and
%~\eqref{eqn_Koksma_Hlawka_applied}.
%This contribution is uniform in $\nOL\in\GetaL$.
%}}} sec_losing_only_a_logarithm
%}}}

%\input{cubes_elimination.tex}
%{{{
\section{Eliminating many digits}\label{sec_uniform} %--- the {\sc frog}}
The main result of the present section is Proposition~\ref{prp_uniform} below.
It concerns the sum-of-digits function along four arithmetic progressions synchronously.
Roughly speaking, we reduce the estimate of our exponential sum $\Seight$ to two terms: a Gowers norm and a discrepancy term.
The Gowers term is a certain higher order correlation~\cite{Tao2012}.
In our case, it originates from the iterated application of van der Corput's inequality.
The discrepancy term captures the deviation from independent and uniform behaviour of the digits of the four slopes.
Later (see Section~\ref{sec_applying_prp_uniform} below), we will specialize to the situation given by the cubes, choosing as slopes the four values $\alpha(\varepsilon,\ts)$ defined in~\eqref{eqn_Qprime_def}.
%{{{ sec_gowers_and_discrepancy
\subsection{Bounding $\Seight$ by the sum of a Gowers norm and a discrepancy term}\label{sec_gowers_and_discrepancy}
The following proposition executes the iterated cutting away of binary digit blocks, reducing the problem (essentially) to estimating (1) a Gowers norm
$\lVert\ThueMorse\rVert^{2^Q}_{U^Q(\mathbb Z/2^\rho\mathbb Z)}$, and (2) a four-dimensional discrepancy, given by $\Eeleven$ below.
We introduce four more error terms, appearing in the proposition.

\begin{itemize}
\item The first error term comes from the iterated application of van der Corput's inequality. 
%{{{ eqn_def
\begin{equation}\label{eqn_Eeight_def}
\Eeight\eqdef
\frac{R_1}{2^\mu}\sum_{\substack{0\leq\ell\leq L\\0\leq j<4}}M_{\ell,j}
+\frac1R_1,
\end{equation}
%}}}
\item The second error term captures carry overflow appearing in one of the margins below our intervals of digits to be cut out.
%{{{ eqn_Enine_def
\begin{equation}\label{eqn_Enine_def}
\Enine\eqdef
2^{-\delta}+2^{4(L+1)}L\smallspace D_{2^\mu}\bigl(\alpha_j2^{-a_\ell}\bigr).
\end{equation}
%}}}
\item
The third error term captures the transition from sums over the index set $\{1,\ldots,R-1\}$ to sums over dyadic intervals:
%{{{ eqn_Eten_def
\begin{equation}\label{eqn_Eten_def}
\Eten(R,c,d)\eqdef \frac{2^{d-c}}{R_1}.
\end{equation}
%}}} eqn_Eten_def
The values $c$ and $d-1$ are the lower and upper ends, respectively, of the window of binary digits that remains after the elimination procedure.
We will obtain a Gowers norm, where the summation variables range over $\{0,\ldots,2^{d-c}-1\}$.
\item
The fourth error term will be used to contain the deviation from independent uniform distribution of the digits, with indices in $[c,d)$, of our four slopes $\alpha(\varepsilon,\ts)$.
This is just a four-dimensional discrepancy (with an additional factor $2^{4(d-c)}$),
%{{{ eqn_fourfold_discrepancy
\begin{equation}\label{eqn_fourfold_discrepancy}
\begin{aligned}
\Eeleven(\mu,c,d)&\eqdef
2^{4(d-c)}
\sup_{\substack{k_0,k_1,k_2,k_3\in\mathbb R\\0\leq k_0,k_1,k_2,k_3<2^{d-c}} }
\Biggl\lvert
\frac{
\#
A(k_0,k_1,k_2,k_3)
}{2^\mu}
-\frac{k_0k_1k_2k_3}{2^{4(d-c)}}
\Biggr\rvert,\\
A(k_0,k_1,k_2,k_3)&=A(\mu,c,d,k_0,k_1,k_2,k_3,\alpha_0,\alpha_1,\alpha_2,\alpha_3)
\\&\eqdef
\biggl\{
n<2^\mu:
\frac{n\alpha_j}{2^c}\in\bigl[k_j,k_j+1\bigr)+2^{d-c}\smallspace\mathbb Z
\mbox{ for }
0\leq j<4
\biggr\}.
\end{aligned}
\end{equation}
%}}}

This only differs from a regular discrepancy by the additional factor $2^{4(d-c)}$:
\[
\Eeleven(\mu,c,d)
=2^{4(d-c)}
D_{2^\mu}\biggl(
\biggl(\frac{n\alpha_0}{2^d},\frac{n\alpha_1}{2^d},\frac{n\alpha_2}{2^d},\frac{n\alpha_3}{2^d}\biggr)_{n\ge0}
\biggr)
\]
This factor care of the possible digit configurations on our window $[c,d)$:
by the Koksma--Hlawka inequality we have to consider ``variation of a step-function $\times$ discrepancy''.
\end{itemize}

%{{{ prp_uniform
\begin{proposition}\label{prp_uniform}
Assume that
$\alpha=\bigl(\alpha_0,\alpha_1,\alpha_2,\alpha_3\bigr)\in\mathbb N^4$, 
$\beta=\bigl(\beta_0,\beta_1,\beta_2,\beta_3\bigr)\in\mathbb N^4$,
$\mu,a,b,c,\kappa,L,\delta$ are nonnegative  integers,
and $a\leq b-L\kappa-\delta$.
Let us define
%{{{ eqn_hoffnung
\begin{equation}\label{eqn_hoffnung}
K(\mu,a,b,\alpha,\beta)\eqdef
\frac1{2^\mu}
\sum_{0\leq n<2^\mu}
\prod_{0\leq j<4}
\e\biggl(\frac12\digitsum^{[a,b)}\bigl(n\alpha_j+\beta_j\bigr)\biggr),
\end{equation}
%}}} eqn_hoffnung
\[ %\begin{equation}\label{eqn_a_b_def}
\left.\begin{array}{r@{\hspace{0.2em}}l}
b_\ell&\eqdef b-\ell\kappa,\\
%a_\ell&\eqdef b-(\ell+1)\kappa-\delta %b_\ell-\kappa-\delta
a_\ell&\eqdef b-(\ell+1)\kappa%b_\ell-\kappa-\delta
\end{array}
\right\}
\mbox{ for $0\leq\ell<L$,}
%\left\{\begin{array}{r@{\hspace{0.2em}}l}
%b_L&\eqdef a+\kappa,\\
%a_L&\eqdef a. %b_L-\tilde\kappa. %a-\delta.
%\end{array}
%\right.
\] %\end{equation}
and assume that $a\leq c\leq a_{L-1}$.
%\[a<b_L<b_{L-1}<\cdots<b_1<b_0=b,\]
%where $b_k=b-k\kappa$.
Suppose that the following two properties are satisfied.
\begin{enumerate}\item The \emph{odd elimination property}:
%{{{ eqn_odd_elimination_applied
\begin{equation}\label{eqn_odd_elimination_applied}
\begin{aligned}
\mbox{There exist integers $M_{\ell,j}$ for $0\leq\ell<L$, $0\leq j<4$, such that}\\
\left.\begin{array}{l}
2\nmid M_{\ell,j},\\[1mm]
1\leq M_{\ell,j}<2^{5(\kappa+\delta)+7},\\[1mm]
\bigl(M_{\ell,j}\smallspace\alpha_j\bigr)^{[a_\ell-\delta,b_{\ell})}
=0\end{array}\right\}
\quad\mbox{for}\quad
\left\{\begin{array}{ll}
0\le\ell<L,\\0\leq j<4.\end{array}\right.
\end{aligned}
\end{equation}
%}}}
\item Removing the lowest digit block:
%{{{ eqn_lowest_digits_condition
\begin{equation}\label{eqn_lowest_digits_condition}
\begin{aligned}
\mbox{There exist integers $M_{L,j}$, for $0\leq j<4$, such that}\\
\left.\begin{array}{l}
1\leq M_{L,j}<2^{3(c-a)},\\[1mm]
\bigl(M_{L,j}\smallspace\alpha_j\bigr)^{[a-\delta,c+1)}
=2^c\end{array}\right\}
\quad\mbox{for}\quad 0\leq j<4.
%0\le\ell\leq L,\\0\leq j<4.\end{array}\right.
\end{aligned}
\end{equation}
%}}} eqn_lowest_digits_condition
\end{enumerate}
Then
%{{{ eqn_uniform
\begin{equation}\label{eqn_uniform}
\begin{aligned}
\bigl\lvert K(\mu,a,b,\alpha,\beta)\bigr\rvert^{2^Q}
&\ll
\frac 1{2^{(Q+1)\rho}}
\lVert
\ThueMorse
\rVert^{2^Q}_{U^Q(\mathbb Z/2^\rho\mathbb Z)}
+\Eeight+\Enine+\Eten+\Eeleven
\end{aligned}
\end{equation}
%}}} eqn_uniform
with an implied constant only depending on $L$,
where $Q\eqdef 4(L+1)$,
%{{{ eqn_Unorm_def
\begin{equation}\label{eqn_Unorm_def}
\lVert
\ThueMorse
\rVert^{2^Q}_{U^Q(\mathbb Z/2^\rho\mathbb Z)}
\eqdef
\sum_{r\in\{0,\ldots,2^\rho-1\}^Q}
\sum_{0\leq n<2^\rho}
\prod_{\varepsilon\in\{0,1\}^Q}
\e\Biggl(\frac12\digitsum^{[0,\rho)}\bigl(n+\varepsilon\cdot r\bigr)\Biggr),
\end{equation}
%}}} eqn_Unorm_def
and
\[ %\begin{equation}\label{eqn_remaining_window}
\rho\eqdef 
a_{L-1}-c=(b-a)-L\kappa-(c-a).
\] %\end{equation}
\end{proposition}
%}}}
%\LS{The variable $\mu$ has to be larger than something involving $a$ and $c$.}
Note that the variables $\alpha_j$ are going to be given by $T\smallspace\alpha(\varepsilon,\ts)$,
and that our sum $\Seight$ is indeed of the form~\eqref{eqn_hoffnung}.
We will prove in Section~\ref{sec_odd_elimination} that the structure~\eqref{eqn_Qprime_def} of our slopes $\alpha(\varepsilon,\ts)$ (depending on $\nOL$ and $\nLO$) implies the existence of integers $M_{\ell,j}$ such as in~\eqref{eqn_odd_elimination_applied} and~\eqref{eqn_lowest_digits_condition}, \emph{for most} parameters $\nOL$ and $\nLO$.

%{{{ subsubsection part 1
\subsubsection{Proof of Proposition~\ref{prp_uniform}, part 1:
Discarding many digits} % by iterated application of van der Corput's inequality}

The factors $M_{L,0}$, $M_{L,1}$, $M_{L,2}$, $M_{L,3}$ will be chosen later in the proof.
We apply van der Corput's inequality $4(L+1)$ times, with factors $M_{\ell,0},\ldots,M_{\ell,3}$, where $0\leq\ell\leq L$.
This is achieved by Lemma~\ref{lem_vdC_iterated}, which yields
%{{{ eqn_K_K1
\[ %\begin{equation}\label{eqn_K_K1}
\bigl\lvert K\bigr\rvert^{2^{4(L+1)}}\ll
\frac1{R^{4(L+1)}}
\sum_{r\in\{1,\ldots,R-1\}^{(L+1)\times 4}}
\bigl\lvert
K_1(r,M)
\bigr\rvert
+\Eeight,
\] %\end{equation}
%}}} eqn_K_K1
where $(L+1)\times 4\eqdef\{0,\ldots,L\}\times\{0,1,2,3\}$,
%{{{ eqn_K1_def
\begin{equation}\label{eqn_K1_def}
K_1(r,M)\eqdef
\frac1{2^\mu}
\sum_{0\leq n<2^\mu}
\prod_{0\leq j<4}
\prod_{\varepsilon\in\{0,1\}^{(L+1)\times 4}}
\e\Biggl(\frac12\digitsum^{[a,b)}\Biggl(\Biggl(n+
\sum_{\substack{0\leq\ell\leq L\\0\leq i<4}}
\varepsilon_{\ell,i}r_{\ell,i}M_{\ell,i}\Biggr)\alpha_j+\beta_j\Biggr)\Biggr),
\end{equation}
%}}} eqn_K1_def
and the error term $\Eeight$ defined in~\eqref{eqn_Eeight_def}.

In order to complete the removal of the $L$ intervals $[a_\ell,b_\ell)$,
we want to avoid carry overflow on the margins $[a_\ell-\delta,a_\ell)$.
More precisely, for each $j\in\{0,1,2,3\}$ and $\ell\in\{0,\ldots,L-1\}$,
and all
$\varepsilon\in\{0,1\}^{(L+1)\times 4}$
we require
\begin{equation}\label{eqn_many_APs}
\biggl\lVert
\frac{n\alpha_j+\tilde\beta_{j,\varepsilon}}{2^{a_\ell}}
\biggr\rVert
\ge2^{-\delta},\quad\mbox{where}\quad
\tilde\beta_{j,\varepsilon}\eqdef
\sum_{\substack{0\leq \ell\leq L\\0\leq i<4}}\varepsilon_{\ell,i}r_{\ell,i}M_{\ell,i}\alpha_j+\beta_j.
\end{equation}

The number of integers $n<2^\mu$ violating~\eqref{eqn_many_APs} for some
$\varepsilon %\bigl(\varepsilon_{\ell,i}\bigr)_{\substack{0\leq\ell\leq L\\0\leq i<4}}
\in\{0,1\}^{(L+1)\times 4}$ and some $\ell\in\{0,\ldots,L-1\}$
will be estimated by (a factor times)
\[ %\begin{equation}\label{eqn_good_upper_bound}
2^\mu\Enine,
\] %\end{equation}
where $\Enine$ is defined in~\eqref{eqn_Enine_def}.
The remaining integers $n$ have the important property that the term that is summed in~\eqref{eqn_K1_def} does not depend on the digits in $[a_\ell,b_\ell)$, for all $\ell\in\{0,\ldots,L-1\}$.

We also wish to cut out the interval $[a,c)$.
This is easier, as~\eqref{eqn_lowest_digits_condition} is a quite specialized condition.
Adding $\varepsilon_{L,j}r_{L,j}M_{L,j}\alpha_j$ does not change the digits with indices in $[a,c)$ if $r_{L,j}<2^\delta$.
This is guaranteed by our choices
%\LS{TODO. $R_1$ is the van der Corput parameter for the cut out intervals!}
$R_1\asymp 2^{\Xi\nu/40}$ and $\delta\sim 3\Xi\nu/100$ later on~\eqref{eqn_small_parameter_choices},~\eqref{eqn_cutout_length}.

We see that this interval do not contribute to the expression $K_1(r,M)$ either, and no further $n$ have to be excluded.
In other words, for all but $\ll2^\mu\Enine$ exceptional $n\in\{0,\ldots,2^\mu-1\}$ we may replace $[a,b)$ by $[b_L,a_{L-1})$.
We re-insert the exceptional integers $n$ after truncating the sum-of-digits function.
For each choice of $r\in\{1,\ldots,R-1\}^{(L+1)\times4}$, we obtain
%{{{ eqn_K_K2
\[ %\begin{equation}\label{eqn_K_K2}
\bigl\lvert K\bigr\rvert^{2^Q}\ll
\frac1{R^Q}
\sum_{r\in\{1,\ldots,R-1\}^{(L+1)\times 4}}
\bigl\lvert
K_2(\mu,L,r,M,b_L,a_{L-1},\alpha,\beta)
\bigr\rvert
+\Eeight+\Enine,
\] %\end{equation}
%}}} eqn_K_K2
where $Q\eqdef 4(L+1)$, and
%where $c=b_L$, $d=a_{L-1}$, and
%{{{ eqn_K2_def
\[ %\begin{equation}\label{eqn_K2_def}
\begin{aligned}
\hspace{2em}&\hspace{-2em}
K_2(\mu,L,r,M,c,d,\alpha,\beta)\\&\eqdef
\frac1{2^\mu}
\sum_{0\leq n<2^\mu}
\prod_{\substack{0\leq j<4\\
%\varepsilon_{\ell,i}\in\{0,1\}
\varepsilon\in\{0,1\}^{(L+1)\times 4}} }
\e\Biggl(\frac12\digitsum^{[c,d)}\Biggl(\Biggl(n2^c+
\sum_{\substack{0\leq\ell\leq L\\0\leq i<4}}
\varepsilon_{\ell,i}r_{\ell,i}M_{\ell,i}\Biggr)\alpha_j+\beta_j\Biggr)\Biggr).
\end{aligned}
\] %\end{equation}
%}}} eqn_K2_def
This completes the first part of our proof of Proposition~\ref{prp_uniform}.

\bigskip\noindent
\begin{remark}
Having overcome the problem ``too many significant digits'',
there are in fact no further serious complications to be expected ---
we have reached a milestone in our proof.
\end{remark}

\bigskip\noindent
The next step is the replacement of our expression $K_2$ involving \emph{four} slopes $\alpha_j$ by a sum featuring only $\alpha_0$.
This introduces the error $\Eeleven$, which is a four-dimensional discrepancy.
Furthermore, the resulting sum will be transformed into a Gowers norm, by using a uniform distribution argument of the digits of $r_{\ell,i}M_{\ell,i}\alpha_0$.

%}}}
%{{{ subsubsection part 2
\subsubsection{Proof of Proposition~\ref{prp_uniform}, part 2:
Gowers norms and discrepancy}
%Reducing the problem to Gowers norm and discrepancy estimates}
\label{sec_part_2}

The presence of \emph{four} slopes $\alpha_j$ is basically due to our first two applications of van der Corput's inequality, see~\eqref{eqn_S4_estimate} and~\eqref{eqn_S8_S9}.
Introducing the discrepancy term $\Eeleven$ defined in~\eqref{eqn_fourfold_discrepancy} (which also allows us to discard the shift $\beta_j$), we have
%{{{
\begin{equation*}
\begin{aligned}
\hspace{2em}&\hspace{-2em}
K_2(\mu,L,r,M,c,d,\alpha,\beta)\\&=
\frac1{2^\mu}
\sum_{0\leq n<2^\mu}
\prod_{\substack{0\leq j<4\\
%\varepsilon_{\ell,i}\in\{0,1\}
\varepsilon\in\{0,1\}^{(L+1)\times 4}} }
\e\Biggl(\frac12\digitsum^{[c,d)}\Biggl(n2^c+
\sum_{\substack{0\leq\ell\leq L\\0\leq i<4}}
\varepsilon_{\ell,i}r_{\ell,i}M_{\ell,i}\alpha_j\Biggr)\Biggr)
+\LandauO\bigl(\Eeleven).
\end{aligned}
\end{equation*}
%}}}
Therefore
%{{{
\[ %\begin{equation}\label{eqn_K2_K3}
\begin{aligned}
\hspace{2em}&\hspace{-2em}
\bigl\lvert K_2(\mu,L,r,M,c,d,\alpha,\beta)\bigr\rvert\leq
\bigl\lvert
K_3
\bigr\rvert
+\LandauO\bigl(\Eeleven),
\end{aligned}
\] %\end{equation}
%}}}
where
%{{{ eqn_K3_def
\[ %\begin{equation}\label{eqn_K3_def}
K_3\eqdef
\frac1{2^\mu}
\sum_{0\leq n<2^\mu}
\prod_{\varepsilon\in\{0,1\}^{(L+1)\times 4}}
\e\Biggl(\frac12\digitsum^{[c,d)}\Biggl(n2^c+
\sum_{\substack{0\leq\ell\leq L\\0\leq i<4}}
\varepsilon_{\ell,i}r_{\ell,i}M_{\ell,i}\alpha_0\Biggr)\Biggr).
\] %\end{equation}
%}}} eqn_K3_def

As a second step in transforming the main term into a Gowers norm,
we use the requirement~\eqref{eqn_lowest_digits_condition} %oddness condition
again, but only for $j=0$.
Since the factor $M_{\ell,i}$ is odd,
the position of the least significant $\tL$ in $\alpha_0$ does not move under multiplication by $M_{\ell,i}$.
Correspondingly, we may choose $U$, a power of two, in such a way that
\[\bigl(UM_{\ell,i}\alpha_0\bigr)^{[0,c]}=2^c.\]
Note that the digits in $[c,d)$ of the multiples $r_1UM_{\ell,i}\alpha_0$,
where $r_1$ varies in an interval of length $2^{d-c}$,
attain each value exactly once.
For each of the $Q$ sums over $r_{\ell,i}$ we therefore replace the summation range $\{1,\ldots,R_1-1\}$ by $\{0,\ldots,k2^{d-c}U-1\}$ for some integer $k\ge0$,
introducing the error $\Eten$ defined in~\eqref{eqn_Eten_def}.
Writing %We decompose the ``van der Corput factor'' $r_{\ell,i}$ in the form
$r_{\ell,i}=r^{(1)}_{\ell,i}U+r^{(0)}_{\ell,i}$, where $r^{(0)}_{\ell,i}<U$,
we visit all digit combinations in $[c,d)$ in a uniform manner, as $r^{(1)}_{\ell,i}$ runs.
We immediately obtain~\eqref{eqn_uniform},
which finishes the proof of Proposition~\ref{prp_uniform}.
\qed
%}}} sec_part_2
%}}} sec_gowers_and_discrepancy

%{{{ sec_applying_prp_uniform
\subsection{Applying Proposition~\ref{prp_uniform}}\label{sec_applying_prp_uniform}

In this section, we apply Proposition~\ref{prp_uniform}
in order to derive an upper bound for $\SY$ (see~\eqref{eqn_second_factor_structure}).
For this, we have to study in detail the identity~\eqref{eqn_S9_linear3} for $\Seight$,
and form a certain average in $(\nLO,\nOO)\in P\times\JOO$ (see~\eqref{eqn_SZ_def}), followed by a supremum over $P\in\partition(\nOL)$ and a sum over $\nOL$~\eqref{eqn_second_factor_structure}.
The role of the sum over $n$ in Proposition~\ref{prp_uniform} is taken by $\nLL\in\JLL$, that is, $\mu=\nu-\rho$
The four slopes we are dealing with (see~\eqref{eqn_Qprime_def}) are given by
%{{{ eqn_alpha_choice
\[ %\begin{equation}\label{eqn_alpha_choice}
\begin{array}{r@{\hspace{1mm}}lr@{\hspace{1mm}}lr@{\hspace{1mm}}lr@{\hspace{1mm}}l}
\alpha_0&\eqdef\Qprime(0,\sO),& %T mal TODO!
\alpha_1&\eqdef\Qprime(0,\sL),& %T
\alpha_2&\eqdef\Qprime(1,\sO),& %T
\alpha_3&\eqdef\Qprime(1,\sL).  %T
\end{array}
\] %\end{equation}
%}}} eqn_alpha_choice
%\beta_0&\eqdef\beta(0,\sO),&
%\beta_1&\eqdef\beta(0,\sL),&
%\beta_2&\eqdef\beta(1,\sO),&
%\beta_3&\eqdef\beta(1,\sL).
Assume that $\beta_0,\beta_1,\beta_2,\beta_3$ are arbitrary values not depending on $n$ (note that $\beta$ does not play a role in Proposition~\ref{prp_uniform}).
We are going to eliminate windows of digits.
Let the integer parameters $L,\kappa\ge0$ be chosen later, and set
\[a\eqdef\tau, \quad b\eqdef u-\rho,\]
\begin{equation*}
\left.\begin{array}{r@{\hspace{0.2em}}l}
b_\ell&\eqdef b-\ell\kappa,\\
a_\ell&\eqdef b-(\ell+1)\kappa
\end{array}
\right\}
\mbox{ for $0\leq\ell<L$.}
\end{equation*}

We are interested in eliminating digits in the intervals $[a_\ell-\delta,b_\ell)$,
where $\delta$ is a margin, defined in~\eqref{eqn_cutout_length}.
In order to obtain an odd elimination factor, we need a margin below the interval that is about three times as big as the interval itself, plus a margin of width $\delta$.
Let us assume that
%{{{ eqn_elimination_within
\[ %\begin{equation}\label{eqn_elimination_within}
\tau\leq u-\rho-L\kappa-4(\kappa+\delta)
).
\] %\end{equation}
%}}} eqn_elimination_within
in other words,
elimination of digits takes place within the interval $[a,b)$.
Making Proposition~\ref{prp_uniform} applicable to our situation mainly concerns two issues.
\begin{enumerate}
\item The (one-dimensional) uniform distribution of digit blocks appearing in the slopes $\alpha(\varepsilon,\ts)$.
This is needed for the \emph{odd elimination property}~\eqref{eqn_odd_elimination_applied}), and will be established in Section~\ref{sec_odd_elimination}.
We will show uniform distribution of the digits of our four slopes in an interval containing $[a_\ell,b_\ell)$ (this interval will be of size $4(\kappa+\delta)+4$, see Lemma~\ref{lem_odd_elimination}), in order to guarantee the existence of \emph{odd factors} $M_{\ell,j}$ eliminating the digits on $[a_\ell-\kappa,b_\ell)$.
\item The estimate of the error~$\Eeleven$ (defined in \eqref{eqn_fourfold_discrepancy}), which is the \emph{fourfold independence} of digit blocks in the interval $[c,d)$ that remains after our iterated elimination procedure.
As soon as this is established, we may replace the sum~\eqref{eqn_hoffnung} of a fourfold product by a fourfold product of sums.
We will yield such an estimate in Section~\ref{sec_fourfold}.
\end{enumerate}

%{{{ sec_odd_elimination
\subsubsection{Verifying the odd elimination property}\label{sec_odd_elimination}
%In order to estimate $\SZ$ (defined in~\eqref{eqn_SZ_def}),
In order to apply Proposition~\ref{prp_uniform},
we need the existence of odd elimination factors
--- the ``odd elimination property''~\eqref{eqn_odd_elimination_applied} ---
 for the slopes $\alpha(\varepsilon,\ts)$ defined in~\eqref{eqn_Qprime_def}, for most parameters $\nOO$ and $r$.
For obtaining an odd elimination factor for most $\nLO\in\JLO$, we will show uniform distribution of the digits of each of our four slopes in a bigger window.
This window is about four times the size of our target interval $[a_\ell-\delta,b_\ell)$ that we want to eliminate, which is due to the term $4\kappa+4$ in Lemma~\ref{lem_odd_elimination}.
Assume that $\ell\in\{0,\ldots,L-1\}$.
We are therefore interested in the discrepancy of
%{{{ eqn_essential_slope
\begin{equation}\label{eqn_essential_slope}
\nLO\mapsto \frac{6(\nO+\varepsilon r)2^\tau\nLO}{2^{b_\ell}},
\end{equation}
%}}} eqn_essential_slope
as $\nLO$ varies in the arithmetic progression $P$, having difference $T\in2\mathbb N+1$.
Note that, due to rotation invariance of the discrepancy, we can skip the second and third summands in~\eqref{eqn_Qprime_def}.
In particular, $\ts$ is irrelevant for this issue altogether.
It remains to find an estimate for the average discrepancy
\[A\eqdef\frac1{2^\tau}\sum_{\substack{\nOO\in\JOO\\\nOL\in\JOL}}
F_1(\nOO,\nOL,\ell),\]
where
\[ %\begin{equation}\label{eqn_F1_def}
F_1(\nOO,\nOL,\ell)\eqdef
\sup_{P\in\partition(\nOL)}D_{\lvert P\rvert}
\bigl(6\bigl(\nOL2^\zeta+\nOO+\varepsilon r\bigr)T(\nOL)2^{\tau-b_\ell}\bigr),
\] %\end{equation}
see~\eqref{eqn_second_factor_structure}.
Such an estimate will be contained in the new error term $\Etwelve$
(in that error term, we need to collect the contributions coming from each $\ell\in\{0,\ldots,L-1\}$).
We need to get rid of the dependencies $\partition=\partition(\nOL)$ and $T=T(\nOL)$, and take advantage of the sum over $\nOL$.
For brevity, set 
%{{{ eqn_T0_def
\begin{equation}\label{eqn_T0_def}
T_0\eqdef 2^{5\etaO+7},
\end{equation}
%}}} eqn_T0_def
% \quad\mbox{and}\quad V\eqdef\sup_{\substack{\nOL\in\JOL\\P\in\partition(\nOL)}}\lvert P\rvert,\]
see~\eqref{eqn_varepsilon_small_enough}.
That is, the quantity $T_0$ is an upper bound for the values $T(\nOL)$.
Moreover, note that by~\eqref{eqn_V_introduction} the size of each arithmetic progression $P\in\partition$ satisfies 
\[\lvert P\rvert\asymp V.\]

By Erd\H{o}s--Tur\'an we obtain
\[
\begin{aligned}
A&\leq
\frac1{H_2}
+
\frac1{2^\tau}
\sum_{\substack{\nOO\in\JOO\\\nOL\in\JOL}}
\sup_{P\in\partition(\nOL)}
\sum_{1\leq\lvert h\rvert<H_2}
\frac 1{\lvert h\rvert}
\Biggl\lvert
\frac1{\lvert P\rvert}
\sum_{0\leq p<\lvert P\rvert}
\exp\bigl(hp\bigl(6\bigl(\nOL2^\zeta+\nOO+\varepsilon r\bigr)T(\nOL)2^{\tau-b_\ell}\bigr)\bigr)
\Biggr\rvert
\\&\leq
\frac1{H_2}
+
\frac1{2^\tau}
\sum_{\substack{\nOO\in\JOO\\\nOL\in\JOL}}
\sup_{P\in\partition(\nOL)}
\sum_{1\leq\lvert h\rvert<H_2}
\frac 1{\lvert h\rvert}
\min\biggl(1,\frac1{\lvert P\rvert}
\bigl\lVert h\bigl(6\bigl(\nOL2^\zeta+\nOO+\varepsilon r\bigr)T(\nOL)2^{\tau-b_\ell}\bigr)\bigr\rVert^{-1}\biggr)
\\&
\ll
\frac1{H_2}
+
\frac1{2^\tau}
\sum_{\substack{\nOO\in\JOO\\\nOL\in\JOL}}
\sum_{1\leq\lvert h\rvert<H_2}
\frac 1{\lvert h\rvert}
\min\biggl(1,\frac1V
\bigl\lVert h\bigl(6\bigl(\nOL2^\zeta+\nOO+\varepsilon r\bigr)T(\nOL)2^{\tau-b_\ell}\bigr)\bigr\rVert^{-1}\biggr)
\\&
\ll
\frac1{H_2}
+
(\logp H_2)
\frac1{2^\tau}
\sum_{\nOL\in\JOL}
\sum_{0\leq n<6H_2\lvert\JOO\rvert}
\min\biggl(1,\frac1V
\bigl\lVert
nT(\nOL)2^{\tau-b_\ell}
+
\beta(\nOL)
\bigr\rVert^{-1}\biggr),
\end{aligned}
\]
where
\[\beta(\nOL)\eqdef
h\bigl(6\bigl(\nOL2^\zeta+\varepsilon r\bigr)T(\nOL)2^{\tau-b_\ell}\bigr),\]
for all integers $H_2\ge1$ and $0\leq \ell<L$.
Note that he length of summation over $n$,
which is bounded below by
$\geq\lvert \JOO\rvert=
2^\zeta\asymp 2^{\lambda/6-\lambda/139}$,
may be considerably smaller than the maximum \emph{fineness}
\[b_0-\tau=u-\tau-\rho\sim(1-11\Xi)\nu/3\]
with which the unit interval is sampled.
(See the definitions in Section~\ref{sec_finishing} below.)
On the other hand, $T(\nOL)$ is odd,
so that the fineness is at least $b_{L-1}-\tau\geq 65\kappa$~\eqref{eqn_65}.

We therefore distinguish between two cases, corresponding to large and small values of $\ell$, respectively.
\begin{enumerate}
\item Assume first that $\zeta\geq b_\ell-\tau$.
We obtain a regular sampling of the unit interval.
Koksma--Hlawka (Lemma~\ref{lem_Koksma_Hlawka} implies
\[A\ll
\frac1{H_2}
+H_2(\logp H_2)
\biggl(
\frac{\logp V}V
+
2^{\tau-b_\ell}
\biggr),
\]
therefore
\[A\ll
\nu\cdot
\bigl(V^{-1/2}+2^{(\tau-b_\ell)/2}\bigr),
%\logp\max(H_2,V),
\]
where the factor $\nu$ accounts for logarithmic terms.

\item If $\zeta<b_\ell-\tau$, we extend the summation over $n$ to
$[0,6H_22^{b_\ell-\tau})$, yielding
the result from the first case, multiplied by a factor $2^{b_\ell-\tau-\zeta}$.
\end{enumerate}
Summarizing, 
\begin{equation}\label{eqn_Etwelve_preparation}
A\ll
\nu \max\bigl(1,2^{b_\ell-\tau-\zeta}\bigr)
\biggl(V^{-1/2}+2^{(\tau-b_\ell)/2}\biggr).
\end{equation}
\begin{remark}\label{rem_after_Etwelve}
We see that the gain $V^{-1/2}$ has to eliminate the factor $2^{b_0-\tau-\zeta}\gg 2^{\lambda/139}$ appearing for $\ell=0$.
(Note the definition~\eqref{eqn_parameter_choices} further down.)
%``Somewhat arbitrarily'' we chose $(2\cdot 1+1)/\ten$ for $\omega$.)
The factor $V$, in turn, has to be dominated by $2^{\etaO}$ (see~\eqref{eqn_Efive_def}), which is guaranteed by our choices $V\asymp 2^{3\lambda/139}$ and $\etaO\sim 4\lambda/139$, which we will make in~\eqref{eqn_parameter_choices} and~\eqref{eqn_etaO_choice}.
Furthermore, $2^{\etaO}$ is smaller than $T_0^{1/5}$ (see~\eqref{eqn_T0_def}),
and we see that we need to allow our bound $T_0$ to be of size around $2^{20\lambda/139}$.
We will take these conditions into account later.
\end{remark}

Equation~\ref{eqn_Etwelve_preparation} provides an upper bound for the number of \emph{critical values} $(\nOO,\nLO)\in\JOO\times\JLO$ that we have to exclude,
where we cannot guarantee small discrepancy along $\nLO\in P$.
Namely,
\begin{align*}
\mathcal C&\eqdef\bigl\{(\nOO,\nOL)\in\JOO\times\JOL:
F_1(\nOO,\nOL,\ell)\geq 2^{-\eta}
\bigr\}
\end{align*}
satisfies
\begin{equation}\label{eqn_critical_values_upper_bound}
\frac1{2^\tau}
\bigl\lvert \mathcal C\bigr\rvert\ll
\nu\smallspace 2^{\eta}
\max\bigl(1,2^{b_\ell-\tau-\zeta}\bigr)
\biggl(V^{-1/2}+2^{(\tau-b_\ell)/2}\biggr),
\end{equation}
for all $\eta>0$.
The \emph{non-critical values} will entail almost uniform distribution of the digits of $\alpha(\varepsilon,\ts)$ in the interval $[b_\ell-4\kappa'-4,b_\ell)$
along $\nLO\in P$,
where
%{{{ eqn_kappa_1_def
\[ %\begin{equation}\label{eqn_kappa_1_def}
\kappa'\eqdef \kappa+\delta=b_\ell-(a_\ell-\delta).
\] %\end{equation}
%}}} eqn_kappa_1_def
By another Koksma--Hlawka argument, our discrepancy term $2^{-\eta}$ has to be multiplied by $2^{4\kappa'+4}$ in order to account for all possibilities of digits in $[b_\ell-4\kappa'-4,b_\ell)$.
By Lemma~\ref{lem_odd_elimination}, the non-critical values $(\nOO,\nLO)$ allow for the odd elimination property, for most $\nLO\in P$.
We translate this idea into formulas. For each integer $\omega\ge0$,
and $\mathfrak a\ge0$,
\[
\begin{aligned}\hspace{5cm}&\hspace{-5cm}
\frac1{\lvert P\rvert}
\#\bigl\{
\nLO\in P:
\bigl(6(\nOL2^\zeta+\nOO+r\varepsilon)\nLO2^\tau+\mathfrak a\bigr)^{[b_\ell-4\kappa'-4,b_\ell)}=\omega
\bigr\}
\\&=\frac{1}{2^{4\kappa'+4}}
+\LandauO
\bigl(
D_{\lvert P\rvert}
\bigl(6\bigl(\nOL2^\zeta+\nOO+\varepsilon r\bigr)T(\nOL)2^{\tau-b_\ell}
\bigr)\bigr).
\end{aligned}
\]
An average over $\nO=\nOL2^\zeta+\nOO$, coming from the definition~\eqref{eqn_second_factor_structure} of $\SY$, leads to the error term $A$.
We balance the contributions of the trivial (``critical'') terms, whose number is estimated in~\eqref{eqn_critical_values_upper_bound} and where we take the trivial estimate, and the nontrivial terms.
This leads to a useful choice of $\eta$, such that $2^{-\eta}$ is the square root of the expression~\eqref{eqn_critical_values_upper_bound}, and therefore a total contribution
%{{{ eqn_Etwelve_def
\begin{equation}\label{eqn_Etwelve_def}
\Etwelve\eqdef
L
2^{4\kappa'+4}
\nu^{1/2}
\max\bigl(1,2^{(b_\ell-\tau-\zeta)/2}\bigr)
\biggl(V^{-1/4}+2^{(\tau-b_\ell)/4}\biggr).
\end{equation}
%}}} eqn_Etwelve_def

We see that the lowest window of digits that is cut away has to be larger than the size $\kappa\sim\Xi\nu/100$ of the other $L$ intervals
(see the definition~\eqref{eqn_cutout_length} in Section~\ref{sec_windows_margins}).
This is necessary in order to guarantee
$(b_\ell-\tau)/4>4\kappa'+4\sim16\kappa+4$.
Adding a margin $\kappa$, we see that it is sufficient to require
\begin{equation}\label{eqn_65}
b_{L-1}-\tau\geq 65\kappa
\end{equation}
in order to obtain a nontrivial bound for $\Etwelve$.

%TODO here we can optimize. This is the reason for the choice 62...
%In order to see sufficiency, note that the interval $[a_{L-1},b_{L-1})$ has length $\kappa$, and the remaining interval $[c,d)$ of digits comprises $\geq\kappa$ digits too. 

%}}}
%{{{ sec_leftshift
\subsubsection{The leftshift lemma}\label{sec_leftshift}
In this short section, we prove a very useful discrepancy result,
which we will use in Sections~\ref{sec_lowest_digits} and~\ref{sec_fourfold} below.

Let $a,b,c$ be integers, where $0\leq a\leq b\leq c$, let $I,I_1$ be nonempty finite intervals in $\mathbb N$, and $(\alpha_j)_{j\in I}$ be a finite sequence of nonnegative integers.
Suppose that we have information about the distribution of the digits of $\alpha_j$ in the interval $[a,c)$.
We are interested in showing \emph{average uniform distribution} of the digits of the sequences
\[\alpha^{(j)}:I_1\rightarrow\mathbb Z,\quad m\mapsto m\alpha_j,\]
with indices in the smaller interval $[b,c)$.
Such a statement forms the core of the \emph{fourfold independence} argument given in Section~\ref{sec_fourfold}, and is presented below.
It is an application of two fundamental inequalities in uniform distribution theory.
%used in DMS2022? What about Kuipers--Niederreiter?
%{{{ lem_leftshift
\begin{lemma}\label{lem_leftshift}
Assume that
$I,I_1$ are finite nonempty intervals in $\mathbb N$ containing $N$ resp. $M$ integers
and let $(\alpha_j)_{j\in I}$ be a sequence of real numbers.
Assume that $q\ge1$ is an integer.
For each $j\in I$, we define the new sequence
\[\alpha^{(j)}\eqdef \bigl(m\alpha_j/q)_{m\in I_1}.\]
The estimate
\[ %\begin{equation}\label{eqn_mean_discrepancy}
\frac1N\sum_{j\in I}
D_M\bigl(\alpha^{(j)}\bigr)
\ll
\frac{q\logp M\logp N}M+
\frac 1qD_N(\alpha)^{1/2}.
\] %\end{equation}
holds with an absolute implied constant.
\end{lemma}
%}}} lem_leftshift
%{{{ proof
\begin{proof}
Assume that $H\ge1$ is an integer.
By Erd\H{o}s--Tur\'an we have
\[
\begin{aligned}
\frac 1N
\sum_{j\in I}
D_M\bigl(\alpha^{(j)}\bigr)
&\ll
\frac 1H
+
\frac 1N
\sum_{\substack{h\in \mathbb Z\\0<\lvert h\rvert<H}}
\frac 1h
\sum_{j\in I}
\Biggl\lvert\frac 1M\sum_{m\in I_1}
\e\bigl(hm\alpha_j/q\bigr)
\Biggr\rvert.
\\&\ll
\frac 1H
+
\frac 1N
\sum_{\substack{h\in \mathbb Z\\0<\lvert h\rvert<H}}
\frac 1h
\sum_{j\in I}
\min\biggl(1,\frac1M\bigl\lVert h\alpha_j/q\bigr\rVert^{-1}\biggr).
\end{aligned}
\]
For integers $a\neq 0$, the variation of $x\mapsto \min\bigl(1,\lVert ax\rVert^{-1}/M\bigr)$ on $[0,1)$ is bounded by $\LandauO(a)$,
therefore we deduce by Koksma--Hlawka
\[
\begin{aligned}
\frac 1N
\sum_{j\in I}
D_M\bigl(\alpha^{(j)}\bigr)
&\ll \frac1H+
\sum_{\substack{h\in \mathbb Z\\0<\lvert h\rvert<H}}
\frac 1h
\biggl(
\frac{q\logp M}M+
\frac hqD_N(\alpha)
\biggr)
\ll \frac1H+
\frac{q\logp M\logp H}M+
\frac{H}{q}D_N(\alpha).
\end{aligned}
\]
Choosing a positive integer $H\asymp\bigl(D_N(\alpha)\bigr)^{-1/2}$ finishes the proof.
\end{proof}
%}}}
%}}} sec_leftshift
%{{{ sec_lowest_digits
\subsubsection{Removing the lowest interval of digits}\label{sec_lowest_digits}
While the intervals $I_0,\ldots,I_{L-1}$ were removed using the odd elimination property, the elimination of the interval
$[a,c)=[a_L,b_L)$ uses the term
%$[a,c)=[\tau,\tau+62\kappa)$ uses the term
\[\bigl(\nOL2^\zeta+\nOO+\varepsilon r\bigr)^2\]
in~\eqref{eqn_Qprime_def}, which we did not take into account before.
We will show that (for most $m$) the digits of
$\nOO\mapsto(\nOO+m)^2$ with indices in $[a-\delta_2,a)$
are uniformly distributed,
where the margin
$\delta_2=c-a+\delta$
is a bit larger than the window $[a,c)$ we want to eliminate, see~\eqref{eqn_cutout_length},~\eqref{eqn_L_def}.
In particular, our slopes $\alpha(\ts,\varepsilon)$ (see~\eqref{eqn_Qprime_def}) will have uniformly distributed digit blocks on $[a-\delta_2,a)$.
We will feed this equidistribution result into Lemma~\ref{lem_leftshift}, and eliminate the smaller window $[a,c)$ above, while generating $\tL$ on position $c$.

We will first study the discrepancy of the sequence
\[x:\nOO\mapsto \frac{(\nOL2^\zeta+\nOO+\varepsilon r)^2}{2^a},\]
using the Erd\H{o}s--Tur\'an inequality with an integer $H_2$ to be chosen later. For $1\leq h<H_2$, consider
\[
T(\nOL,\varepsilon,r,a,h)\eqdef
\frac1{2^\zeta}
\sum_{\nOO\in\JOO}
\e\biggl(h\frac{(\nOL2^\zeta+\nOO+\varepsilon r)^2}{2^a}\biggr).
\]
In order to estimate $T(\nOL,\varepsilon,r,a,h)$, we apply the Kusmin--Landau Theorem~\cite[Theorem~2.1]{GrahamKolesnik1991}.
Set $f(x)=hx^2/2^\tau$.
As $x$ varies in an interval of length $\ll2^\tau$,
the derivative $f'$ is increasing,
and runs through $\LandauO(h)$ intervals $[n,n+1)$ at a constant rate.
Let \[\sigma_{\nOL}:\nOO\mapsto 2h(\nOL2^\zeta+\nOO+\varepsilon r)/2^\tau\]
be the sequence of slopes.
Consequently,
for each $\varepsilon\in\{0,1\}$ and each $r$,
the number of $\nOL\in\JOL$ such that
%{{{ eqn_kusmin_landau_exceptions
\[ %\begin{equation}\label{eqn_kusmin_landau_exceptions}
\lVert\sigma_{\nOL}(\nOO)\rVert\leq\gamma
\quad\mbox{for some $\nOO\in\JOO$}
\] %\end{equation}
%}}} eqn_kusmin_landau_exceptions
is bounded by
\begin{align*}
&\ll
h\biggl(1+\frac1{h2^{\zeta-\tau}}\biggl(
\gamma+\frac{\lvert\JOO\rvert}{2^\tau}
\biggr)
\biggr)
\\&\ll
H_2+\gamma\smallspace2^{\tau-\zeta}
\end{align*}
That is, for each $r<R$ we have to exclude a fraction

\[\Ethirteen^{(0)}\ll H_22^{\zeta-\tau}+\gamma\]
of integers $\nOL\in\JOL$, which will constitute the first part of our error $\Etwelve$.
For the remaining $\nOL$, we estimate $T$ by Kusmin--Landau:
%{{{ eqn_Kusmin_Landau_good
\begin{equation}\label{eqn_Kusmin_Landau_good}
T(\nOL,\varepsilon,r,a,h)
%\sum_{\nOO\in\JOO}
%f\bigl(\nOL2^\zeta+\nOO+\varepsilon r\bigr)
\ll
\Ethirteen^{(1)}\eqdef \frac1{\gamma\smallspace2^\zeta}.
\end{equation}
%}}} eqn_Kusmin_Landau_good
The estimate~\eqref{eqn_Kusmin_Landau_good} directly carries over to a discrepancy estimate for the digits in $[a-\delta_2,a)$ of our four slopes $\alpha(\ts,\varepsilon)$: for most $\nOL$, we have uniform distribution (in $\nOO$) of the digits of our four slopes, with indices in $[a-\delta_2,a)$.
Balancing the two error terms arising above, from the exclusion of certain $\nOL$, and Kusmin--Landau respectively, we arrive at
\[\gamma\eqdef 2^{-\zeta/2}.\]

Next, these slopes are multiplied by $\nLL$, and we are interested in the digits of $\nLL\mapsto \nLL\smallspace\alpha(\ts,\varepsilon)$
\emph{with indices in} $[a-\delta_2,a+1)$.
This is an application of the \emph{leftshift lemma}, where $q=2$,
and we define
\begin{equation}\label{eqn_Ethirteen_def}
\Ethirteen\eqdef
\frac{\nu^2}{2^{\nu-\rho}}
+2^{-\zeta/4}.
\end{equation}
Taking along this error, we may assume that each of the digit block configurations
\[(\nLL\alpha(\ts,\varepsilon))^{[a-\delta_2,a+1)}=\omega,\]
for all four slopes, appear with the fair frequency $2^{-\delta_2-1}$.
We easily get factors $M_{L,j}$ such as in Proposition~\ref{prp_uniform},
including a multiplication by $2^{c-a}$. This is accounted for by the bound $2^{3(c-a)}$ in~\eqref{eqn_lowest_digits_condition}.

%}}} sec_lowest_digits
%{{{ sec_fourfold
\subsubsection{The fourfold independence}\label{sec_fourfold}
We want to apply Proposition~\ref{prp_uniform} in order to bound the sum $\SY$ defined in~\eqref{eqn_second_factor_structure}.
To this end, we show that the actual slopes~\eqref{eqn_Qprime_def}
\[\Qprime(\varepsilon,\ts)
=6\smallspace\nA(\varepsilon)2^\tau\nLO
+\,6\smallspace\ts\smallspace\multiple(\nOL)\smallspace\nA(\varepsilon)2^\tau
+\,3\smallspace\nA(\varepsilon)^2\]
coming from our problem cause the discrepancy error term $\Eeleven$ defined in~\eqref{eqn_fourfold_discrepancy} to be small,
where suitable averages (over $\nOO,\nOL,\nLO,\sO,\sL$) and a supremum (over $P\in\partition(\nOL)$) are involved.
This error measures the deviation from independent uniform distribution of the blocks of digits, with indices in $[c,d)$, of the four slopes
%{{{ eqn_alpha_choice2
\[ %\begin{equation}\label{eqn_alpha_choice2}
\begin{array}{r@{\hspace{1mm}}lr@{\hspace{1mm}}lr@{\hspace{1mm}}lr@{\hspace{1mm}}l}
\alpha_0&=\tilde\alpha_02^{-d},&\alpha_1&=\tilde\alpha_12^{-d},&\alpha_2&=\tilde\alpha_22^{-d},&\alpha_3&\eqdef \tilde\alpha_32^{-d},\quad\mbox{where}\\
\tilde\alpha_0&=\alpha(0,\sO),& \tilde\alpha_1&=\alpha(1,\sO),&\tilde\alpha_2&=\alpha(0,\sL),&\tilde\alpha_3&=\alpha(1,\sL).
\end{array}
\] %\end{equation}
%}}} eqn_alpha_choice2
Keep in mind that $\alpha(\varepsilon,\ts)$
also depends on $\nLO,\nOL,\nOO$, and $r$.
When the smallness of the discrepancy $\Eeleven$ is established,
only the first term in~\eqref{eqn_uniform} is relevant,
and therefore the problem of estimating $\SY$
is reduced to bounding a Gowers norm.

In order to account for all discrepancies that appear when the sums $\Seight$
in~\eqref{eqn_second_factor_structure}
are estimated with the help of Proposition~\ref{prp_uniform},
it will be sufficient to estimate
%{{{ eqn_average_discrepancy
\begin{equation}\label{eqn_average_discrepancy}
\Efourteen\eqdef
2^{4(d-c)}
\frac1{2^{\tau-\zeta}S^2}
\sum_{\substack{\nOL\in\JOL\\0\leq\sO,\sL<S}}
\sup_{P\in\partition(\nOL)}
\frac1{\lvert P\rvert R\,2^\zeta}
\sum_{\substack{\nLO\in P\\\lvert r\rvert<R\\\nOO\in\JOO}}
D_{\lvert\JLL\rvert}\bigl(\alpha_0,\alpha_1,\alpha_2,\alpha_3\bigr),
\end{equation}
%}}} eqn_average_discrepancy
which comes from the error $\Eeleven$ defined in~\eqref{eqn_fourfold_discrepancy}.
Note that $c$ and $d$ are the limits of the interval of digits that remain after the elimination of $L+1$ intervals, and we will define them later~\eqref{eqn_L_def}.
Also, we quietly used the argument (employed at various occasions) that exponents $\geq 1$ of error terms can be replaced by $1$, by considering the trivial bounds.
We apply the Erd\H{o}s--Tur\'an--Koksma inequality~\eqref{eqn_ETK},
with
\[N=\lvert \JLL\rvert,\quad \ell_0\eqdef \nOL2^\zeta+\nOO,\quad \ell_1\eqdef r,\]
and an integer parameter
\begin{equation}\label{eqn_H0_choice}
H_1=2^{m_1},\quad\mbox{where}\quad m_1\ge0, 
\end{equation}
to be chosen later.
Note that we may omit $\beta(\varepsilon,\ts)$, since~\eqref{eqn_ETK} does not see rotations modulo $1$.
We obtain
\begin{equation}\label{eqn_fourfold_ETK}
\begin{aligned}
\hspace{1cm}&\hspace{-1cm}
D_N\bigl(\alpha_0,\ldots,\alpha_3\bigr)\ll
\frac 1{H_1}+
\sum_{\substack{h\in \mathbb Z^4\\0<\lVert h\rVert_\infty<H_1}}
\frac 1{\mu(h)}
\left\lvert\frac 1N\sum_{0\leq n<N}
\e\bigl(n\bigl(h_0\alpha_0+\cdots+h_3\alpha_3\bigr)\bigr)
\right\rvert
\\&
\ll
\frac 1{H_1}+
\sum_{\substack{h\in \mathbb Z^4\\0<\lVert h\rVert_\infty<H_1}}
\frac 1{\mu(h)}
\min\biggl(1,\frac1N\biggl\lVert 
\frac{1}{2^d}
\bigl(h_0\alpha(0,\sO)+h_1\alpha(0,\sL)
\\&\hspace{6cm}+h_2\alpha(1,\sO)+h_3\alpha(1,\sL)\bigr)
\biggr\rVert^{-1}\biggr)
\\&\leq
\frac 1{H_1}+
\sum_{\substack{h\in \mathbb Z^4\\0<\lVert h\rVert_\infty<H_1}}
\frac 1{\mu(h)}
\min\biggl(1,\frac1N\bigl\lVert 
A_0+A_1+A_2
\bigr\rVert^{-1}\biggr),
\end{aligned}
\end{equation}
where
\begin{equation*}
\begin{aligned}
A_0(\ell_0,\ell_1,\nLO)&\eqdef
6\cdot2^{\tau-d}
\bigl((h_0+h_1+h_2+h_3)\ell_0+(h_2+h_3)\ell_1\bigr)\nLO,
\\A_1(\ell_0,\ell_1,\sO,\sL)&\eqdef6\cdot2^{\tau-d}\multiple(\nOL)
\bigl((h_0\sO+h_1\sL+h_2\sO+h_3\sL)\ell_0+
(h_2\sO+h_3\sL)\ell_1\bigr),
\\A_2(\ell_0,\ell_1)&\eqdef3\cdot 2^{-d}\bigl((h_0+h_1)\ell_0^2+(h_2+h_3)(\ell_0+\ell_1)^2\bigr).
\end{aligned}
\end{equation*}
Next, we make use of the sums over $\nOO\in\JOO$, affecting $\ell_0$, and $\ell_1=r\in\{-R+1,\ldots,R-1\}$.
In the following, we are concerned with arguments concerning the distribution of the digits of $A_0+A_1+A_2$, with indices in $[c,d)$.
%Note that the sum over $\nOO$ is shorter than the interval
For this, we introduce another margin
%{{{ eqn_another_margin
\begin{equation}\label{eqn_another_margin}
\delta_1\eqdef32\kappa,
\end{equation}
%}}} eqn_another_margin
and we study the larger interval $[c-\delta_1,d)$.
By our definition $c=\tau+64\kappa$ in Section~\ref{sec_windows_margins} we clearly have $\delta_1=\frac{c-\tau}2$.
The argument showing smallness of $\Efourteen$ defined in~\eqref{eqn_average_discrepancy}, which is a certain average over the values $D_N(\alpha_0,\alpha_1,\alpha_2,\alpha_3)$, uses a case distinction into two cases, having two subcases each. The treatments of these four cases are similar to each other, and we will only give a detailed proof in the first case (a).

\paragraph{The case $(h_0\neq-h_1\textsf{ or } h_2\neq-h_3)$.}
(a) Let us first assume that $h_0+h_1+h_2+h_3\ne0$.
In this case, $\nu_2(h_1+h_2+h_3+h_4)\leq m_1+2$,
and we note the important restriction that
\begin{equation}\label{eqn_important_restriction}
m_1+3<c-\tau-\delta_1,
\end{equation}
which we will guarantee later.
Smallness of the sum
%{{{ eqn_intsum_def
\[ %\begin{equation}\label{eqn_intsum_def}
C_0\eqdef
\frac1{\lvert P\rvert R2^\zeta}
\sum_{\substack{\nLO\in P\\\lvert r\rvert<R\\\nOO\in\JOO}}
\min\biggl(1,\frac1N\bigl\lVert A_0+A_1+A_2\bigr\rVert^{-1}\biggr),
%
%D_{\lvert\JLL\rvert}\bigl(\alpha_0,\alpha_1,\alpha_2,\alpha_3\bigr)
\] %\end{equation}
%}}} eqn_intsum_def
where $N=\lvert\JLL\rvert$, is achieved in two stages.
In the following, we will assume the important estimate 
%{{{ eqn_zeta_lower_bound
\begin{equation}\label{eqn_zeta_lower_bound}
\zeta\geq d-\tau.
\end{equation}
%}}} eqn_zeta_lower_bound
Our definitions in Section~\ref{sec_choosing} below will guarantee that~\eqref{eqn_zeta_lower_bound} is indeed satisfied.

First, we consider the sum over $\ell_0=\nOO$.
Recall that the odd number $T$
(oddness is guaranteed by $\Efive$ and $\Esix$)
is the difference of each arithmetic progression $P$ from $\partition(\nOL)$.
As $\ell_0$ runs through an interval of length $2^a$, where $a\geq d-\tau$,
the term
\[\gamma_x(\ell_0)=6T2^{\tau-d}(h_0+h_1+h_2+h_3)\ell_0+x
\]
modulo $1$ is distributed in a very regular manner.
More precisely, noting the properties~\eqref{eqn_zeta_lower_bound} that $\zeta\ge d-\tau$
and~\eqref{eqn_another_margin} $c-\delta_1\geq\tau$, which imply $\zeta\geq d-c+\delta_1$,
it has the property that
\[
\#\biggl\{
\ell_0<2^\zeta:
\gamma_x(\ell_0\bigr)\in
\biggl[\frac{k}{2^{d-c+\delta_1}},\frac{k+1}{2^{d-c+\delta_1}}\biggr)
+\mathbb Z
\biggr\}
=
2^{\zeta-d+c-\delta_1}
\]
for all $x\in\mathbb R$.
We translate this into a discrepancy statement, as a preparation for an application of the \emph{leftshift lemma} (Lemma~\ref{lem_leftshift}).
Using decomposition of an interval in the torus $\mathbb R/\mathbb Z$ into intervals of length $2^{c-d-\delta_1}$, and at most two singular intervals,
we obtain
\[D_{2^\zeta}\bigl(\gamma_x\bigr)\ll2^{c-d-\delta_1},\]
where
\[x\eqdef 6T2^{\tau-d}(h_2+h_3)\ell_1.\]
Subsequently, we can vary $\nLO$ along the arithmetic progression $P$, having odd difference $T$.
This amounts to studying $m\mapsto TA_0(m)+B$, and we can apply
Lemma~\ref{lem_leftshift} in order to obtain
\begin{equation*}
\begin{aligned}
%\hspace{4cm}&\hspace{-4cm}
\frac1{2^\zeta}
\sum_{\nOO\in\JOO}
D_{\lvert P\rvert}
\bigl(TA_0(\nOO,r,\cdot)+A_1(\nOO)+A_2(\nOO)\bigr)
&=
\frac1{2^\zeta}
\sum_{\nOO\in\JOO}
D_{\lvert P\rvert}
\bigl(TA_0(\nOO,r,\cdot)\bigr)
\\&\ll
2^{(c-d-\delta_1)/2}
+\frac{\nu^2}{\lvert P\rvert}.
\end{aligned}
\end{equation*}
In particular, for each $K>0$, the sequence
$\pi:P\rightarrow\mathbb R, \nLO\mapsto A_0(\nLO)+A_1+A_2$ has a discrepancy bounded by
\[K2^{(c-d-\delta_1)/2}\]
for at least $2^\zeta\bigl(1-\LandauO(1/K)\bigr)$ integers $\nOO\in\JOO$.
Using Koksma--Hlawka again (the total variation is $\ll1$ as before), we immediately obtain
\[
C_0\ll \frac 1K+K2^{(c-d-\delta_1)/2}
+\frac{\nu^2}{\lvert P\rvert},
\]
in particular,
%{{{ eqn_CO_estimate
\begin{equation}\label{eqn_CO_estimate}
C_0\ll 2^{(c-d-\delta_1)/4}
+\frac{\nu^2}{\lvert P\rvert}.
\end{equation}
%}}} eqn_CO_estimate
The calculation above obviously yields a bound that is independent of the intercept of the progression $P$, which eliminates the supremum.
Also, the estimate does not involve $\nOL$, $\sO$, and $\sL$.
These summands appearing in~\eqref{eqn_average_discrepancy} are therefore irrelevant.
Note that we have $C_0$ needs to eliminate an additional factor $2^{4(d-c)}$ --- the number of possible configurations of four windows $[c,d)$ of digits.
The choice~\eqref{eqn_another_margin} is sufficient, as $d-c\leq 2\kappa$ by~\eqref{eqn_L_def}, and we obtain
\begin{equation}\label{eqn_4dcCO_estimate}
2^{4(d-c)}C_0\ll
2^{15(d-c)/4-32\kappa/4}
+2^{4(d-c)}\nu^2/\lvert P\rvert
\ll
2^{-\Xi\nu/200}
+2^{2\Xi\nu/25}\nu^2/\lvert P\rvert.
\end{equation}

Note that $\lvert P\rvert\asymp V$, which is assigned a value later on~\eqref{eqn_parameter_choices}. In particular, this error, leading to $\Efourteen$, will be negligible.

\bigskip\noindent
(b)
If $h_0+h_1+h_2+h_3=0$, we necessarily have $h_2+h_3\ne0$, and we use the sum over $\ell_1=r$ instead of $\ell_0$ as before.
Assume that
\begin{equation}\label{eqn_zeta_R_large_enough}
R\geq 2^{d-\tau},\quad \zeta\geq d-\tau\geq 4(d-c).
\end{equation}
An analogous argument as in (a) above finishes the case.
\paragraph{The case $(h_0=-h_1\textsf{ and } h_2=-h_3)$.}
In this case, the terms $A_0$ and $A_2$ are zero.
We can therefore discard the $\sup$ in~\eqref{eqn_average_discrepancy}.
Again, we split the argument into two cases.
In the first stage of each of the cases, we use the sum over $\ell_0=\nOO$
resp. $\ell_1=r$, followed by the second stage, using $\sO$ and $\sL$.
We have
\[
A_1=6\cdot2^{\tau-d}\multiple(\nOL)
\bigl((h_0+h_2)\ell_0+h_2\ell_1\bigr)(\sL-\sO).
\]

\bigskip\noindent
(c) Assume first that $h_0\neq-h_2$.
In this case, we have
\[\nu_2\bigl(6\multiple(\nOL)(h_0+h_1)\bigr)\leq m_1+2.\]
We use the sum over $\ell_0=\nOO$ first, followed by the sum over $\sL$ (for each given $\sO$).
This summation has a length $S=R$, and this case, too, is finished by an argument as before.

\bigskip\noindent
(d) In the last case, we assume that $h_0=-h_2$. As $\lVert h\rVert_\infty>0$,
this implies $h_2\ne0$, in particular, $\nu_2(h_2)\leq m_1$.
The first stage is accomplished by the variable $\ell_1=r$,
followed by the sum over $\sL$, for each fixed $\sO$, as in case (c).

Summarizing, choosing $H_1=\lfloor2^{-7\Xi\nu/100}\rfloor$ it follows that~\eqref{eqn_important_restriction} is satisfied, and from~\eqref{eqn_average_discrepancy} and~\eqref{eqn_4dcCO_estimate}
we obtain
%and~\eqref{eqn_fourfold_ETK} we get
%{{{ eqn_Efourteen_estimate
\begin{equation}\label{eqn_Efourteen_estimate}
\Efourteen
\ll
\nu^6
\bigl(2^{-\Xi\nu/200}
+2^{2\Xi\nu/25}\nu^2/\lvert P\rvert
\bigr),
\end{equation}
%}}} eqn_Efourteen_estimate
where the fourfold sum over $(h_0,\ldots,h_3)$ yields the logarithmic term $(\logp H_1)^4\ll\nu^4$, and the first term in Lemma~\ref{lem_leftshift} contributes a factor $\nu^2$.
%}}}

\[\ast\quad\ast\quad\ast\]

%}}}
%}}}

%\input{cubes_finishing.tex}
%{{{ sec_finishing
\section{Tying up loose ends}\label{sec_finishing}
In this section, we put together all the pieces in order to prove Proposition~\ref{prp_sufficient}.

Starting from $\nu$, parameters 
$\lambda,u,\rho,\tau,\xi,\etaO,B,H,L,R,S,\kappa$ have to be chosen such that the conditions~\eqref{eqn_S9_requirements},~\eqref{eqn_additional_requirements0},~\eqref{eqn_additional_requirements1},~\eqref{eqn_additional2},~\eqref{eqn_parameters_stack},~\eqref{eqn_tau_large},~\eqref{eqn_eta_condition} are satisfied.
We will show that our choices of variables
causes the error terms $\Ezero$--$\Efourteen$ to be small.
Moreover, an estimate for the Gowers norm in~\eqref{eqn_uniform} is available~\cite{Konieczny2019,Spiegelhofer2020}.

%{{{ sec_choosing
\subsection{Choosing some free parameters}\label{sec_choosing}
We begin with the choice of an auxiliary quantity,
%{{{ eqn_Xi_def
\begin{equation}\label{eqn_Xi_def}
\Xi\eqdef 1/15000. %1/9000.
\end{equation}
%}}} eqn_Xi_def
Regarding this definition, we note that we do not strive to obtain optimal constants in our estimates.
The essential content of our main theorem --- uniform distribution of $\mathsf t(n^3)$ in $\mathbb Z/2\mathbb Z$ with an error term $N^{-c}$ for \emph{some} $c>0$ ---
is not changed by the particular choice of variables.
Starting from $\nu$, we set
%{{{ eqn_parameter_choices
\begin{equation}\label{eqn_parameter_choices}
\begin{array}{r@{\hspace{1mm}}lr@{\hspace{1mm}}lr@{\hspace{1mm}}lr@{\hspace{1mm}}lr@{\hspace{1mm}}lr@{\hspace{1mm}}lr@{\hspace{1mm}}l}
\tilde\lambda&\eqdef (2+2\Xi)\nu,&
\tilde\rho&\eqdef(1-2\Xi)\nu,&
\tilde u&\eqdef (2-5\Xi)\nu,
\\
\lambda&\eqdef 3\lfloor\tilde\lambda/3\rfloor,&
\rho&\eqdef\lfloor\tilde\rho\rfloor,&
u&\eqdef\lfloor \tilde u\rfloor,\\[2mm]
\tau&\eqdef\lambda/3,&
\zeta&\eqdef
\bigl\lfloor\lambda\bigl(1/6-1/139\bigr)\bigr\rfloor,&
\omega&\eqdef\bigl\lfloor 3\lambda/139\bigr\rfloor.
\end{array}
\end{equation}
%}}} eqn_parameter_choices
Define $V\eqdef 2^\omega$.
We also choose the ``small values'' $B,H,R_1,R,S$:
%{{{ eqn_small_parameter_choices
\begin{equation}\label{eqn_small_parameter_choices}
\begin{array}{r@{\hspace{1mm}}lr@{\hspace{1mm}}lr@{\hspace{1mm}}lr@{\hspace{1mm}}lr@{\hspace{1mm}}lr@{\hspace{1mm}}lr@{\hspace{1mm}}l}
\tilde B&\eqdef 2^{180\Xi\nu},&
\tilde H&\eqdef 2^{8\Xi\nu},&
\tilde R_1&\eqdef 2^{\Xi\nu/40},& %needed in sec_odd_elimination: of course R_1<2^\delta.
\tilde S&\eqdef 2^{17\Xi\nu},
\\
B&\eqdef \lfloor \tilde B\rfloor,&
H&\eqdef2^{\lambda-u}\bigl\lfloor2^{u-\lambda}\tilde H\bigr\rfloor,&
R_1&\eqdef\lfloor\tilde R_1\rfloor,&
R\eqdef S&\eqdef\lfloor\tilde S\rfloor.
\end{array}
\end{equation}
%}}} eqn_small_parameter_choices
Provided that $\nu$ is greater than some absolute constant,
these choices imply
\[\lambda\ge u\ge\nu\ge\rho\ge\tau\ge\zeta\ge0\]
(see~\eqref{eqn_S9_requirements}), and the validity of
~\eqref{eqn_additional_requirements0},
~\eqref{eqn_additional_requirements1},
~\eqref{eqn_additional2},
~\eqref{eqn_parameters_stack},
~\eqref{eqn_tau_large}.
Note that $R$ is much larger than $R_1$. The first variable is used in our second application of van der Corput's inequality, while the second is used for the iterative reduction of digits.
We need the fact that $R$ is larger than $2^{c-\tau}$ in order to ensure ``fourfold independence'', see Section~\ref{sec_fourfold}.

We choose another parameter $\etaO$, which measures the quality~\eqref{eqn_varepsilon_small_enough} of Diophantine approximation by the factor $T$.
When varying the indices of $\varphi_H(K(\cdot))$ along an arithmetic progression with difference $T$, only small fluctuations of the values should appear.
In particular, the quality $\etaO$ of approximation in~\eqref{eqn_varepsilon_small_enough} has to eliminate factors $H^2SB$ arising in~\eqref{eqn_declass},
in order to guarantee small steps along the geometric sum $\varphi_H$,
see the definition~\eqref{eqn_Efive_def} of $\Efive$.
As we noted in Remark~\ref{rem_after_Etwelve}, we set
%{{{ eqn_etaO_choice
\begin{equation}\label{eqn_etaO_choice}
\tilde\etaO\eqdef
4\lambda/139,
\quad
\etaO\eqdef\lfloor \tilde\etaO\rfloor,
\end{equation}
%}}} eqn_etaO_choice
so that
\begin{align*}
2^{\etaO}&\gg H^2SBV2^{\Xi\nu}
\\&\asymp 2^{214\Xi\nu}V.
\end{align*}
This is needed for the smallness of $\Efive$~\eqref{eqn_Efive_def}.
Our choice of $\Xi$ easily guarantees validity of this estimate, which already handles the error term $\Efive$.

Our definition of $\zeta$ implies
\[\lambda-2\tau-\zeta\sim
\lambda/6+\lambda/139
\geq 4\etaO+4\sim 16\lambda/139
\]
for large $\nu$,
therefore~\eqref{eqn_eta_condition} is asymptotically satisfied.
%\LS{as soon as $\ten\geq 90$.
%but we need more: $TV\leq\lvert\JLO\rvert$ is needed too!}
For this condition alone to hold, it would be enough to choose any constant larger than $90$, instead of $139$.
But we also need~\eqref{eqn_TV_restriction} as a preparation for the decomposition of $\JLO$ into arithmetic progressions, stating that $TV\leq\lvert\JLO\rvert$.
Note first that
\[\lvert\JLO\rvert=2^{\rho-\tau}\asymp2^{(1-8\Xi)\nu/3}.\]
Concerning the sizes of $T$ and $V$, consult~\eqref{eqn_varepsilon_small_enough},~\eqref{eqn_parameter_choices}, and the considerations in Remark~\ref{rem_after_Etwelve}. Combining these considerations, we see that any constant larger than $138$ is sufficient in order to guarantee $TV\leq\lvert \JLO\rvert$.
%origin of 139 is here
We also set
%{{{ eqn_etaL_choice
\begin{equation}\label{eqn_etaL_choice}
\tilde\etaL\eqdef
16\Xi\nu,
\quad
\etaL\eqdef\lfloor \tilde\etaL\rfloor.
\end{equation}
%}}} eqn_etaL_choice
This choice yields $\SX\ll\logp H$ (see~\eqref{eqn_SX_estimate}),
and $\Eseven\ll (SBH\lambda)^22^{\zeta-\tau}$ (see~\eqref{eqn_Eseven_estimate}).

\subsection{Windows and margins}\label{sec_windows_margins}
In order to eliminate windows of digits,
we use the sum over $\nLL$.
The factors $M$ in Lemma~\ref{lem_vdC} have to be assumed to be of size up to $2^{5(\kappa+\delta)+7}$, where $\kappa$ is the width of each intervals to be removed, and $\delta$ is the margin width for each of these intervals.
These factors are multiplied by $R$, and we want to have
%{{{ eqn_MR_fits
\begin{equation}\label{eqn_MR_fits}
MR=o(\JLL)
\quad\mbox{and}\quad
R\ll 2^{\delta(1-\varepsilon)}
\end{equation}
%}}} eqn_MR_fits
in order to obtain nontrivial results.
Note that the original sum over $n$ has length $\lvert J\rvert=2^\nu$, while $\lvert\JLL\rvert\asymp 2^{2\Xi\nu}$ and $\Xi=1/15000$.
The interval $[\tau,u-\rho)$ of digits remains after the linearization procedure (see~\eqref{eqn_S9_linear3}).
Let us define
%{{{ eqn_showdown_limits
\begin{align}
a&\eqdef \tau,\quad b\eqdef u-\rho,\label{eqn_showdown_limits}\\
\kappa&\eqdef\lfloor\Xi\nu/100\rfloor,
\label{eqn_cutout_length}\\
  \delta&\eqdef\lfloor3\Xi\nu/100\rfloor.
\end{align}
%}}} eqn_showdown_limits

Clearly, the requirement~\eqref{eqn_MR_fits} is satisfied for these choices.
In particular, note that $2^\delta$ is larger than $R_1$, where $\delta$ is the margin used in digit block elimination:
we need to guarantee that no carries from below $a_\ell-\delta$ propagate into the interval $[a_\ell,b_\ell)$ we want to eliminate, when adding $r_{\ell,j}M_{\ell,j}\alpha_j$ (see~\eqref{eqn_K1_def}).
On the other hand, $R_1$ has to be larger than $2^{2\kappa}$ in order to be useful for the transition to a Gowers norm.
Note that $2\kappa$ is the maximal size of the remaining interval $[c,d)$: this follows from the definitions
\begin{equation}\label{eqn_L_def}
\begin{aligned}
L&\eqdef \min\bigl\{
\ell\ge0: b-\ell\kappa\leq a+66\kappa
\bigr\},\\
c&\eqdef a+64\kappa,\quad d\eqdef b-L\kappa,\\
\delta_2&\eqdef c-a+\delta=64\kappa+\delta.
\end{aligned}
\end{equation}
With definitions~\eqref{eqn_parameter_choices},~\eqref{eqn_small_parameter_choices},~\eqref{eqn_L_def},
we see that~\eqref{eqn_important_restriction},~\eqref{eqn_zeta_lower_bound}, and~\eqref{eqn_zeta_R_large_enough} are satisfied.
The interval $[c,d)$ that remains may be as large as
$2\kappa$,
while it has a size at least
$\kappa$.
The interval $[\tau,c)$ will be cut out later.
It is $64$ times as big as the $L$ regular intervals that we eliminated.
The reason for this discrepancy lies in \emph{odd elimination}, which we have to guarantee, and which requires enough digits below the $L$ regular intervals (see the final part of Section~\ref{sec_odd_elimination}).
\begin{remark}
\begin{enumerate}
\item
Note that $b-a=u-\rho-\tau\sim(1/3-11\Xi/3)\nu$, while the eliminated intervals have size $\sim \Xi\nu/100$.
It follows that the number $L+1$ of ``slices'' will be around half a million for $\nu$ large enough.
This number can be lowered significantly by a more detailed study, but the fact that we need to cut off multiple times will not be altered by these considerations.
Consequently, the structure of our estimate in the main theorem --- a term of the form $N^{1-c}$ --- will remain the same, as we noted above.
Meanwhile, the gain $c$ is \emph{halved} in each of the $500000$ applications of the Cauchy--Schwarz inequality (compare~\eqref{eqn_vdC_iterated}).
\item
In the treatment of the complete Gelfond conjecture there is an essential difficulty left: cutting off digits multiple times is possible when the argument of the sum-of-digits function is linear, but we do not see a way for polynomials of higher degree.
Moreover, trying to repeat the presented argument for $n^4$,
we see that quadratic polynomials remain, not only as arguments of $\digitsum$, but also in the exponentials.
Dirichlet approximation, as used in the linear case for finding $\multiple(\nOL)$ and $T$, is not sufficient here.
\end{enumerate}
\end{remark}
%}}} sec_choosing
%{{{ sec_gowers_applied
\subsection{Applying a Gowers uniformity norm estimate}\label{sec_gowers_applied}
The following statement was essentially given by Konieczny~\cite{Konieczny2019},
see~\cite[Proposition~3.3]{Spiegelhofer2020} for this version.
%{{{ lem_gowers
\begin{lemma}\label{lem_gowers}
Let $k\geq 2$ be an integer.
There exist real numbers $\eta>0$ and $C$ such that
\[ %\begin{equation}\label{eqn_gowers}
\frac 1{2^{(k+1)\rho}}
\lVert
\digitsum
\rVert^{2^k}_{U^k(\mathbb Z/2^\rho\mathbb Z)}
\leq C 2^{-\rho\eta}
\] %\end{equation}
for all $\rho\geq 0$,
where the expression on the left hand side is defined in~\eqref{eqn_Unorm_def}.
\end{lemma}
%}}} lem_gowers
This term appears in
our upper bound for $\SY$ in Proposition~\ref{prp_uniform}, see~\eqref{eqn_uniform},~\eqref{eqn_Unorm_def}.
Bounding the remaining four errors $\Eeight$--$\Eeleven$ appearing in that proposition will be performed in Section~\ref{sec_finishing_1} below.
%}}} sec_gowers_applied

%{{{ sec_finishing_1
\subsection{Finishing the proof of Proposition~\ref{prp_sufficient}}\label{sec_finishing_1}
We combine the estimate $\bigl\lVert\Szero(\nu,\xi)\bigr\rvert^4\leq\Snine+\LandauO\bigl(\Ezero+\Eone+\Etwo+\Ethree+\Efour\bigr)$ from %Proposition~\ref{prp_linearize} and
Corollary~\ref{cor_linearize} with the estimate $\lvert \Snine\rvert\ll \SX+\SY+\Efive+\Esix$ from Proposition~\ref{prp_glycerol}.
The term $\SX$ is treated in Section~\ref{sec_losing_only_a_logarithm}, and yield only a logarithmic factor, provided that $\Geta$ is chosen in a reasonable way (see~\eqref{eqn_GetaO_def},~\eqref{eqn_mathcalGetaL_def}, $\Geta=\GetaO\cap\GetaL$).
The new main term $\SY$ is handled in Section~\ref{sec_uniform}:
after many iterations, say $k$, of van der Corput's inequality, we can estimate the remaining expression by a Gowers norm estimate, yielding some positive exponent $\eta$.
This exponent has to be reduced by a factor $2^k$, coming from that many applications of the Cauchy--Schwarz inequality. The resulting factor $N^{-\eta/2^k}$ easily swallows the logarithm coming from $\SX$, despite the exponent coming from our method is rather small.

It remains to show that our choices of variables
causes each of the error terms to be negligible.

%\subsection{The error terms $\Ezero$--$\Efourteen$}
%\label{sec_estimating}

%{{{ sec_Ezero_estimate
\subsubsection{The term $\Ezero$: trigonometric approximation error}
\label{sec_Ezero_estimate}
We are going to investigate the error $E_0$, which is the first error term that we collected along the way, and which arose from approximation by trigonometric polynomials.
\begin{lemma}\label{lem_Ezero_estimate}
Let %$H\ne0$, $\nu\ge0$, $\lambda$, and $\mu$ be integers, and let
$E_0(\nu,\lambda,\mu,H)$ be defined by~\eqref{eqn_Ezero_def}.
For some absolute implied constant, we have
\[ %\begin{align}\label{eqn_Ezero_estimate}
E_0\bigl(\nu,u,\lambda,H\bigr)
\ll
\frac{2^\kappa}H
+2^{\lambda/3-\nu}
+H^{3/2}2^{(\nu-\lambda)/2},
\] %\end{align}
for all $\nu,u,\lambda,H$ satisfying
$\nu\ge1$, $u\in\{0,\ldots,2\nu\}$, $\lambda\ge2\nu$, $H\geq 1$, and $2^\kappa\mid H$.
\end{lemma}
\begin{proof}
The estimate of $E_0$ will involve dyadic intervals.
Let $\lambda\ge0$, $h\in\mathbb Z$, $M\ge1$, and set
\[F(\lambda,h,M)\eqdef\sum_{M\leq n<2M}\e\bigl(hn^32^{-\lambda}\bigr).
\]
We are going to find a nontrivial estimate for $E_0$ provided that $h\ne0$.
Let us apply van der Corput's theorem~\cite[Theorem~2.2]{GrahamKolesnik1991}.
Set $f(x)=hx^3/2^\lambda$, then $f''(m)\asymp hM2^{-\lambda}$ for $m\in[M,2M)$, therefore
\[F(\lambda,h,M)\ll M\sqrt{hM}2^{-\lambda/2}+2^{\lambda/2}/\sqrt{M}\]
with some absolute implied constant.
Choose $\ell\eqdef \lambda/3$. 
Dyadic decomposition of $[2^\ell,2^\nu)$ yields
\[ %\begin{equation}\label{eqn_dyadic}
\begin{aligned}
\sum_{n\in J}\e\bigl(hn^32^{-\lambda}\bigr)
&=2^\ell+
\sum_{\ell\leq j<\nu}
F\bigl(\lambda,h,2^j\bigr).
\\&\ll
2^\ell+
\Bigl(2^{3\nu}H/2^\lambda\Bigr)^{1/2}
+\Bigl(2^{\lambda}2^{-\ell}\Bigr)^{1/2}.
\\&=2\cdot2^{\lambda/3}+
\Bigl(2^{3\nu}H/2^\lambda\Bigr)^{1/2}.
\end{aligned}
\] %\end{equation}
Treating the term corresponding $h=0$ separately,
and taking together $h$ and $-h$,
the statement of the lemma follows.
\end{proof}
The three summands in Lemma~\ref{lem_Ezero_estimate}
are bounded by $2^{-\Xi\nu}$ by our choice of variables.
%}}} sec_Ezero_estimate
%{{{ sec_Eone_estimate
\subsubsection{The term~$\Eone$: minor modification of summation limits}
\label{sec_Eone_estimate}
The error term $\Eone=SBH2^{-(\nu-\tau)}$
arises when the conditions $n+\sO\in \JL$, $n+\sL\in\JL$ coming from the van der Corput inequality are discarded (see~\eqref{eqn_S3_S5}).
We easily see that $\Eone\ll2^{-\Xi\nu}$.
%}}} sec_Eone_estimate
%{{{ sec_Etwo_estimate
\subsubsection{The term~$\Etwo$: choosing the factors $\multiple(\nOL)$}
\label{sec_Etwo_estimate}
We need to find odd factors $\multiple(\nOL)$
eliminating some digits of $\nO$ directly below $\lambda-\tau-\rho$.
This is needed in order to remove the difference terms coming from summation by parts, see~\eqref{eqn_S6_split}.
We need to cancel the contribution $SH^22^{\nu-\rho}\asymp 2^{35\Xi\nu}$,
using an odd factor $\multiple(\nOL)$ of size $\leq B$.
We therefore aim for a size $\ll 2^{-36\Xi\nu}$ of the $\lVert\cdot\rVert$-expression in~\eqref{eqn_E2_def}.
Considering Lemma~\ref{lem_odd_elimination}, we need to admit $B\gg 2^{5\cdot 36\Xi\nu}$. This explains the choice $B\asymp 2^{180\Xi\nu}$.

Moreover, the window of digits to be eliminated needs a margin below that is about three times the size of the interval.
This leads to the condition
\[4\cdot36\Xi\nu\leq\lambda-\tau-\rho
\asymp (1+10\Xi)\nu/3,\]
which is satisfied since $3\cdot4\cdot 36\Xi<1$.
Lemma~\ref{lem_odd_elimination} is applicable, and we only need to exclude a proportion $\sim 2^{-36\Xi\nu}$ of integers $\nO\in\JO$.
(Also note that the sum over $\nO\in\JO$, of length $2^\tau$, is longer than $2^{\lambda-\tau-\rho}$, and we trivially obtain uniform distribution of the digits of $\nO$ in the interval to be eliminated.)

We note that the factor $T$ introduced in Section~\ref{sec_uncoupling} (see the definition~\eqref{eqn_Efive_def} of $\Efive$) eliminates digits of the same number $\nO$, at a different index.
An analogous argument is applicable for this case, and we will comment on this in Section~\ref{sec_Efivesix_estimate}.
below.
%}}} sec_Etwo_estimate
%{{{ sec_Ethree_estimate
\subsubsection{The term~$\Ethree$: removing the $\zeta$ least significant digits of $\nO$} %of the summation variable}
\label{sec_Ethree_estimate}
The error term $\Ethree$ captures the effect of the lowest $\zeta$ digits of $\nO$ on the position at which $\varphi_H$ has to be evaluated, see~\eqref{eqn_K_def} and~\eqref{eqn_Kprime_def}.
Multiplying this by the total variation $\asymp H^2$, we obtain a bound for the introduced error.
Considering our definitions and~\eqref{eqn_E3_def}, we have
\[\Ethree\ll2^{A\nu}+2^{B\nu},\]
where
\begin{align*}
\begin{array}{l@{\hspace{1mm}}l@{\hspace{1mm}}l}
A&=16\Xi+(1/3+\Xi/3-2/139)+(180+17)\Xi-(1+10\Xi)/3&\leq -\Xi,\\
B&=16\Xi+(1/3+\Xi/3-2/139)+(360+34)\Xi-(2+2\Xi)/3&\leq -\Xi.
\end{array}
\end{align*}
For these inequalities to hold, we also need a quite small value for $\Xi$.
This is taken care of by the choice~\eqref{eqn_Xi_def}.

%}}} sec_Ethree_estimate
%{{{ sec_Efour_estimate
\subsubsection{The term~$\Efour$: applying the Carry Lemma}
\label{sec_Efour_estimate}
The error term $\Efour$ is the reason for which we have to choose $\lambda$ strictly larger than $2\nu$.
It mainly concerns the error coming from the omission of the digits with indices $\geq\lambda$, justified by van der Corput's inequality.
This term requires a closer look, and we discussed it in Section~\ref{sec_prp_linearize_finishing}, which completed the proof of Proposition~\ref{prp_linearize}
We get a gain $\Efour\ll 2^{-c\nu}$, for some $c>0$, also in this case.
%}}} sec_Efour_estimate
%{{{ sec_Efivesix_estimate
\subsubsection{The terms~$\Efive$ and~$\Esix$: choosing the factor $T$}
\label{sec_Efivesix_estimate}
We eliminate digits of $\nOL$, with indices in a different interval (compare to $\Etwo$, which handles digits directly below $\lambda-\tau-\rho$).
This time we are concerned with the digits directly below $\lambda-2\tau-\zeta$.
Error terms $\Efive$ and $\Esix$ come from Proposition~\ref{prp_glycerol}.
The term $\Efive$ accounts for Dirichlet approximation by $T$.
In order to complete the treatment of $\Efive$, it is sufficient to take Remark~\ref{rem_after_Etwelve} into account: allowing $T$ to be of size around $2^{20\lambda/139}$, the error term is guaranteed to be small.
The term $\Esix$ takes care of the exceptional parameters $\sO$, $\sL$, and $\nOL$, where concentration using an odd $T$, leading to~\eqref{eqn_declass}, is not (easily) possible.
In order to bound~\eqref{eqn_Esix_estimate} nontrivially, it is sufficient to consider the sizes of $H$, $S$, and $\etaO$:
noting~\eqref{eqn_small_parameter_choices} and~\eqref{eqn_etaO_choice}, 
there is nothing to be worried about.

%In order to complete the treatment of these error terms, we just note that
%$\lambda-2\tau-\zeta$ is smaller than $\tau-\zeta$, where $\lvert\JOL\rvert=2^{\tau-\zeta}$,
%but larger than $4\etaO+4$. Therefore odd elimination of $\etaO$ digits of $\nOL$ is applicable for all but very few exceptions, in particular, we only need to exclude a proportion $\LandauO\bigl(2^{-\Xi\nu}\bigr)$ of all $\nOL\in\JOL$.

%}}} sec_Efivesix_estimate
%{{{ sec_Eseven_estimate
\subsubsection{The term~$\Eseven$: ensuring representative sampling of the geometric sum}
\label{sec_Eeeight_estimate}
We need to exclude some more $\nOL\in\JOL$, leaving only $\nOL\in\bigl(\JOL\setminus\GetaO\bigr)\setminus\GetaL$.
For these exceptional $\nOL\in\GetaL$, uniform distribution of sufficient quality is not guaranteed.
By~\eqref{eqn_Eseven_estimate}, the choice~\eqref{eqn_etaL_choice}, and $3\tau=\lambda$, we have
$\Eseven\ll (SBH\lambda)^22^{\zeta-\tau}$.
The ``large'' value $2^{\tau-\zeta}$ easily eliminates the ``small'' values $S$, $B$, and $H$, since $\Xi$ was chosen sufficiently small.
An estimate $\Eseven\ll2^{-\Xi\nu}$ follows.

%\[\SX\ll\logp H+H^2B
%\biggl(U^{-1/2}+\alpha+\frac1{M\lVert\alpha\rVert}
%\biggr)\bigl(\logp U\bigr)^2,\]
%where $M=6SB2^{\tau-\zeta}$,
%$U=2^{\rho-\tau}$, and
%$\alpha=2^{-(\lambda-2\tau-\zeta)}$.
%
%This gives $H^2/S$ times a logarithmic term,
%which is fine since $S\asymp 2^{\Xi\nu}H^2$ by definition.

%}}} sec_Eseven_estimate
%{{{ sec_Eeight_estimate
\subsubsection{The term~$\Eeight$: adjusting summation limits} % in the iterated van der Corput inequality}
\label{sec_Eeight_estimate}
Note that in Proposition~\ref{prp_uniform}, the index set $\{0,\ldots,2^\mu-1\}$ takes the role of $\JLL$, see the definition~\eqref{eqn_Eeight_def}.
We wish to omit the marginal intervals in van der Corput's inequality, that is, the additional conditions under the inner summation sign.
For this, we need to take care that the sum over $\JLL$
is still larger than the van der Corput variable $R_1$, multiplied by our odd interval deletion coefficients $M_{\ell,j}$.
Note that $\lvert \JLL\rvert=2^{\nu-\rho}\asymp2^{2\Xi\nu}$ and $R_1\asymp 2^{\Xi 3\nu/100}$.
We commented on this case just before the definitions of $\kappa$ and $\delta$ (see~\eqref{eqn_cutout_length}, which takes care of the case $1\leq\ell<L$.
For $\ell=L$, we have 
\[R_1M_{L,j}/2^\mu\leq 2^{\Xi\nu/40+3(c-a)-2\Xi\nu}
\ll 2^{(5/2+192)\kappa-2\Xi\nu}\ll 2^{-\Xi\nu/20}\]
(note the size restriction $M_{L,j}<2^{3(c-a)}$~\eqref{eqn_lowest_digits_condition}),
and we see that $\Eeight\ll R_1^{-1}\asymp2^{-\Xi\nu/40}$.
%here we need the smallness of kappa!!
%}}} sec_Eeight_estimate
%{{{ sec_Enine_estimate
\subsubsection{The term~$\Enine$: avoiding long carry propagation}
\label{sec_Enine_estimate}
We need to handle carry overflow in the margins
\[[a_\ell-\delta,a_\ell)\]
of the $L$ first intervals $I_\ell=[a_\ell,b_\ell)$ to be cut out.
More precisely, the error term $\Enine$ defined in~\eqref{eqn_Enine_def},
\[\Enine=2^{-\delta}+2^{4(L+1)}LD_{2^\mu}\bigl(\alpha_j2^{-a_\ell}\bigr)\]
has to be estimated \emph{on average},
where the slopes $\alpha_j$ are essentially given by\[\nLO\mapsto\frac{6(\nO+\varepsilon r)2^\tau\nLO}{2^{a_\ell}},\]
and $\nLO$ varies in an arithmetic progression $P$~\eqref{eqn_essential_slope}.
This average is treated in an manner analogous to the ``odd elimination property'' established in Section~\ref{sec_odd_elimination}, see Lemma~\ref{lem_MU},
yielding the same bound $2^{-c\Xi\nu}$ as for $\Etwelve$.
%}}} sec_Enine_estimate
%{{{ sec_Eten_estimate
\subsubsection{The term~$\Eten$: towards a Gowers norm}
This is an ancillary error term, taking care of the replacement of the index set $\{1,\ldots,R_1-1\}$ by a dyadic interval.
We obtain full sums over the period of our periodic function, leading directly to a Gowers norm.
The property~\eqref{eqn_L_def} combined with $R_1\asymp2^{(5/2)\kappa}$ guarantees that $R_1$ is larger than the number $2^{d-c}$ of digit combinations in the remaining interval, by a factor $2^{\kappa/2}$.
\label{sec_Eten_estimate}
%}}} sec_Eten_estimate
%{{{ sec_Eeleven_estimate
\subsubsection{The term~$\Eeleven$: a general (fourfold) discrepancy}
\label{sec_Eeleven_estimate}
This term is not estimated independently.
We will only be interested in a certain mean value~\eqref{eqn_average_discrepancy}, which is the term $\Efourteen$.
%}}} sec_Eeleven_estimate
%{{{ sec  _Etwelve_estimate
\subsubsection{The term~$\Etwelve$: digit blocks of the four slopes are uniformly distributed}
\label{sec_Etwelve_estimate}
The term~$\Etwelve$ simultaneously takes care of the \emph{odd elimination property}~\eqref{eqn_odd_elimination},
and \emph{carry overflows}~\eqref{eqn_many_APs} on our $L$ intervals to be eliminated.
In order to establish smallness of the error term~\eqref{eqn_Etwelve_def},
we finally need the definition $V=\lfloor 2^{3\lambda/139}\rfloor$ (see~\eqref{eqn_parameter_choices}),
and we easily get a bound $\Etwelve\ll2^{-c\nu}$ for some $c>0$.

%}}} sec_Etwelve_estimate
%{{{ sec_Ethirteen_estimate
\subsubsection{The term~$\Ethirteen$: eliminating the lowest digit block}
\label{sec_Ethirteen_estimate}
The error~$\Ethirteen$ defined in~\eqref{eqn_Ethirteen_def}
clearly has the property~$\Ethirteen\ll2^{-\Xi\nu}$ for large $\nu$.
%}}} sec_Ethirteen_estimate
%{{{ sec_Efourteen_estimate
\subsubsection{The term~$\Efourteen$: the fourfold independence}
\label{sec_Etsixteen_estimate}
When a certain average of $\Eeleven$
%\[D_{\lvert\JLL\rvert}(\alpha_0,\alpha_1,\alpha_2,\alpha_3)\]
in $\alpha_0,\alpha_1,\alpha_2,\alpha_3$ is formed,
we obtain $\Efourteen$~\eqref{eqn_average_discrepancy}.
An estimate for this expression (the ``fourfold independence'' result) was given in Section~\ref{sec_fourfold}.
The second term in~\eqref{eqn_Efourteen_estimate}
will be small by our choice $V\asymp2^{3\lambda/139}$ above.
The total contribution of the sum $\Efourteen$ can therefore be estimated by $2^{-c\Xi\nu}$ for some $c>0$.
%}}} sec_Efourteen_estimate
%}}} sec_finishing_1
%}}} sec_finishing

%{{{ arXiv
\bigskip
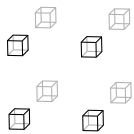
\begin{center}
\definecolor {gr1}{rgb}{0.7,0.7,0.7}
\newcommand{\mul}{5}
\begin{tikzpicture}[xscale=0.2,yscale=0.2,anchor=base]
\node (a000) at ({(0+\eps)*\mul},{(0+\eps)*\mul}) {$\cube$};
\node (a001) at (\mul,0) {$\cube$};
\node (a010) at (0,\mul) {$\cube$};
\node (a011) at ({(1-\eps)*\mul},{(1-\eps)*\mul}) {$\cube$};
\node[color=gr1] (a100) at ({(0+\eps+\deptx)*\mul},{(0+\eps+\depty)*\mul}) {$\cube$};
\node[color=gr1] (a101) at ({(1+\deptx)*\mul},\depty*\mul) {$\cube$};
\node[color=gr1] (a110) at ({(0+\deptx)*\mul},{(1+\depty)*\mul}) {$\cube$};
\node[color=gr1] (a111) at ({(1-\eps+\deptx)*\mul},{(1-\eps+\depty)*\mul}) {$\cube$};
\end{tikzpicture}
\end{center}
%}}}

%\input{cubes_finalremarks.tex}
%{{{ subsection Final remarks
\section*{Final remarks}
\begin{enumerate}
\item
The attentive reader
\footnote{\url{http://www.ma.rhul.ac.uk/~uvah099/Sat/reader.html}}
will have noticed that in our arguments, the Thue--Morse sequence does not play a special role among the sequences $\digitsum_q\bmod m$ satisfying $\gcd(m,q-1)=1$.
We are confident that, at the cost of adding arguments involving divisors of $q$ and $m$ and appropriate residues classes at numerous places,
our method is sufficiently strong to handle the general case $\digitsum_q\bigl(n^3\bigr)\bmod m$.
We decided to avoid these additional complications and decided in favour of a more concise presentation.
The same line of thought applies to the generalization to arbitrary polynomials $P$ of degree $3$ such that $P(\mathbb N)\subseteq\mathbb N$.
In order to keep the key arguments clear, we only proved the ``base case'' of the ``degree-$3$ case'' of Gelfond's third problem.
\item
The real difficulty that remains is the generalization to polynomials of arbitrary degree --- the complete solution to Gelfond's third problem.
This problem is of a different kind and certainly requires major new ideas.

\end{enumerate}
%}}} subsection

%\input{cubes_acknowledgements.tex}
%{{{ subsection Acknowledgements
\section*{Acknowledgements}
We thank Thomas Stoll and J\"org Thuswaldner for valuable discussions on the digits of cubes.
Special thanks to Michael Drmota, who suggested that the obvious decomposition
$n=q^\tau n_1+n_0$
could in fact be a useful approach to the sum of digits of cubes.
This idea started our research on the question answered in the present paper.
%}}} subsection Acknowledgements
%
%%{{{ Bibliography
\bibliographystyle{siam}
\bibliography{gelfond}

\def\cprime{$'$}
\begin{thebibliography}{10}

\bibitem{AlloucheShallit1999}
{\sc J.-P. Allouche and J.~Shallit}, {\em The ubiquitous
  {P}rouhet-{T}hue-{M}orse sequence}, in Sequences and their applications
  ({S}ingapore, 1998), Springer Ser. Discrete Math. Theor. Comput. Sci.,
  Springer, London, 1999, pp.~1--16.

\bibitem{AlloucheShallit2003}
\leavevmode\vrule height 2pt depth -1.6pt width 23pt, {\em Automatic
  sequences}, Cambridge University Press, Cambridge, 2003.
\newblock Theory, applications, generalizations.

\bibitem{Besineau1972}
{\sc J.~B\'esineau}, {\em Ind\'ependance statistique d'ensembles li\'es \`a la
  fonction ``somme des chiffres''}, Acta Arith., 20 (1972), pp.~401--416.

\bibitem{ByszewskiKoniecznyMuellner2020}
{\sc J.~Byszewski, J.~Konieczny, and C.~Müllner}, {\em Gowers norms for
  automatic sequences}, 2020.

\bibitem{DartygeTenenbaum2006}
{\sc C.~Dartyge and G.~Tenenbaum}, {\em Congruences of sums of digits of
  polynomial values}, Bull. Lond. Math. Soc., 38 (2006), pp.~61--69.

\bibitem{DrmotaMauduitRivat2011}
{\sc M.~Drmota, C.~Mauduit, and J.~Rivat}, {\em The sum-of-digits function of
  polynomial sequences}, J. Lond. Math. Soc. (2), 84 (2011), pp.~81--102.

\bibitem{DrmotaMauduitRivat2019}
{\sc M.~Drmota, C.~Mauduit, and J.~Rivat}, {\em Normality along squares}, J.
  Eur. Math. Soc. (JEMS), 21 (2019), pp.~507--548.

\bibitem{DrmotaMuellnerSpiegelhofer2021}
{\sc M.~Drmota, C.~Müllner, and L.~Spiegelhofer}, {\em Primes as sums of
  {F}ibonacci numbers}, 2021.
\newblock 135 pages. Accepted for publication in {M}em. {A}mer. {M}ath. {S}oc.
  (2022).

\bibitem{Gelfond1968}
{\sc A.~O. Gel{\cprime}fond}, {\em Sur les nombres qui ont des propri\'et\'es
  additives et multiplicatives donn\'ees}, Acta Arith., 13 (1967/1968),
  pp.~259--265.

\bibitem{GrahamKolesnik1991}
{\sc S.~W. Graham and G.~Kolesnik}, {\em van der {C}orput's method of
  exponential sums}, vol.~126 of London Mathematical Society Lecture Note
  Series, Cambridge University Press, Cambridge, 1991.

\bibitem{Granville1992}
{\sc A.~Granville}, {\em Zaphod {B}eeblebrox's brain and the fifty-ninth row of
  {P}ascal's triangle}, Amer. Math. Monthly, 99 (1992), pp.~318--331.

\bibitem{Kim1999}
{\sc D.-H. Kim}, {\em On the joint distribution of {$q$}-additive functions in
  residue classes}, J. Number Theory, 74 (1999), pp.~307--336.

\bibitem{Konieczny2019}
{\sc J.~Konieczny}, {\em Gowers norms for the {T}hue-{M}orse and
  {R}udin-{S}hapiro sequences}, Ann. Inst. Fourier (Grenoble), 69 (2019),
  pp.~1897--1913.

\bibitem{Kummer1852}
{\sc E.~E. Kummer}, {\em {\"Uber die Erg\"anzungss\"atze zu den allgemeinen
  Reciprocit\"atsgesetzen}}, J. Reine Angew. Math., 44 (1852), pp.~93--146.

\bibitem{MauduitRivat2009}
{\sc C.~Mauduit and J.~Rivat}, {\em La somme des chiffres des carr\'es}, Acta
  Math., 203 (2009), pp.~107--148.

\bibitem{MauduitRivat2010}
\leavevmode\vrule height 2pt depth -1.6pt width 23pt, {\em Sur un probl\`eme de
  {G}elfond: la somme des chiffres des nombres premiers}, Ann. of Math. (2),
  171 (2010), pp.~1591--1646.

\bibitem{Moshe2007}
{\sc Y.~Moshe}, {\em On the subword complexity of {Thue}-{Morse} polynomial
  extractions}, Theor. Comput. Sci., 389 (2007), pp.~318--329.

\bibitem{MuellnerSpiegelhofer2017}
{\sc C.~M\"ullner and L.~Spiegelhofer}, {\em Normality of the {T}hue--{M}orse
  sequence along {P}iatetski-{S}hapiro sequences, {II}}, Israel J. Math., 220
  (2017), pp.~691--738.

\bibitem{OEIS}
{\sc {N. J. A. Sloane}}, {\em {T}he {O}n-{L}ine {E}ncyclopedia of {I}nteger
  {S}equences}.
\newblock published electronically at {\tt https://oeis.org}.

\bibitem{Singmaster1980}
{\sc D.~Singmaster}, {\em Divisibility of binomial and multinomial coefficients
  by primes and prime powers}, Fibonacci Assoc., Santa Clara, Calif., 1980.

\bibitem{Spiegelhofer2020}
{\sc L.~Spiegelhofer}, {\em The level of distribution of the {T}hue--{M}orse
  sequence}, Compos. Math., 156 (2020), pp.~2560--2587.

\bibitem{SpiegelhoferWallner2021}
{\sc L.~Spiegelhofer and M.~Wallner}, {\em The binary digits of $n+t$}, Ann.
  Sc. Norm. Super. Pisa, Cl. Sci. (5), 24 (2023), pp.~1--31.

\bibitem{Stoll2012}
{\sc T.~Stoll}, {\em The sum of digits of polynomial values in arithmetic
  progressions}, Funct. Approximatio, Comment. Math., 47 (2012), pp.~233--239.

\bibitem{Tao2012}
{\sc T.~Tao}, {\em Higher order {F}ourier analysis}, vol.~142 of Graduate
  Studies in Mathematics, American Mathematical Society, Providence, RI, 2012.

\bibitem{Vaaler1985}
{\sc J.~D. Vaaler}, {\em Some extremal functions in {F}ourier analysis}, Bull.
  Amer. Math. Soc. (N.S.), 12 (1985), pp.~183--216.

\end{thebibliography}
%%}}} Bibliography

\smallskip
%Lukas Spiegelhofer
\begin{center}
\begin{tabular}{c}
Department Mathematics and Information Technology,\\
Montanuniversit\"at Leoben,\\
Franz-Josef-Strasse 18, 8700 Leoben, Austria\\
\texttt{lukas.spiegelhofer@unileoben.ac.at}\\
ORCID iD: \texttt{0000-0003-3552-603X}
\end{tabular}
\end{center}

\end{document}